\documentclass[11pt,reqno]{amsart}
\usepackage[a4paper,margin=25mm]{geometry}
\usepackage[T1]{fontenc}
\usepackage{lmodern,microtype}
\usepackage{graphicx,placeins}
\usepackage{amsmath,amssymb,amsthm,mathtools,booktabs,longtable,array}
\usepackage{xcolor}
\newcommand{\revision}[1]{#1}
\usepackage[hidelinks]{hyperref}
\allowdisplaybreaks
\numberwithin{equation}{section}
\numberwithin{figure}{section}
\newtheorem{theorem}{Theorem}[section]
\newtheorem{proposition}[theorem]{Proposition}
\newtheorem{lemma}[theorem]{Lemma}
\newtheorem{corollary}[theorem]{Corollary}
\theoremstyle{definition}
\theoremstyle{remark}\newtheorem{remark}[theorem]{Remark}
\newcommand{\R}{\mathbb R}
\newcommand{\Sph}{\mathbb S^2}
\newcommand{\Per}{\operatorname{Per}}
\newcommand{\tr}{\operatorname{tr}}
\newcommand{\cof}{\operatorname{cof}}
\newcommand{\Div}{\operatorname{div}}
\newcommand{\Hess}{\operatorname{Hess}}
\newcommand{\E}{\mathbb E}
\newcommand{\pos}[1]{\left(#1\right)_+}
\newcommand{\dd}{\,d}
\newcommand{\sig}{\,d\sigma}
\newcommand{\V}{\operatorname{Vol}}
\title[Relative isoperimetry in the three-dimensional cube]{Relative isoperimetry in the three-dimensional cube:\\ classification and quantitative stability}
\hypersetup{pdftitle={Relative isoperimetry in the three-dimensional cube: classification and quantitative stability}}
\author{Bin Deng}
\author{Jiahuan Li}
\author{Yilu Liu}
\date{October 6, 2026}
\begin{document}
\begin{abstract}
We compute the relative isoperimetric profile of the three-dimensional
unit cube. Up to cube isometries and complementation, its minimizers are
corner eighth-balls, edge quarter-cylinders and coordinate slabs, with both
adjacent families minimizing at $4\pi/81$ and $1/\pi$. At each fixed volume
the perimeter deficit bounds the squared $L^1$ distance to the minimizing
family, with optimal exponent two. Near each transition the deficit
controls the minimum, over the two adjacent families, of squared
distance plus excess perimeter. The best constant for distance to the
minimizing family decays linearly as the volume approaches a transition.
We also obtain the conjectured profile for $(0,\beta)\times(0,1)^2$,
$0<\beta\le1$, and resolve the three-dimensional case of Brezis's
Open Problem~10.1.
\end{abstract}
\maketitle
\tableofcontents

\section{Introduction}\label{sec:introduction}
The relative isoperimetric problem in a cube is to minimize the
perimeter inside the cube at prescribed volume. Ros~\cite[Sections
1.5--1.6]{Ros} discusses the conjecture that the minimizers are corner
eighth-balls, edge quarter-cylinders and coordinate slabs, up to cube
isometries and complementation. Milman~\cite[Conjecture 1.1 and
Theorem 1.2]{Milman} formulates the existence assertion for cuboids
of side lengths $(\beta,1,1)$ and proves it in explicit volume ranges.

We determine the profile and classify all minimizers in the unit cube,
including both transition volumes. The interval on which the profile
equals one resolves the three-dimensional case of
\cite[Open Problem 10.1]{BrezisOpen}.

\subsection{The unit cube and the principal result}
Let $Q=(0,1)^3$, and let $\Per(E;Q)$ denote the relative perimeter of
a finite-perimeter set $E\subset Q$. Set
\[
 I(v)=\inf_{\substack{E\subset Q\\ |E|=v}}\Per(E;Q),\qquad
 a=\frac{4\pi}{81},\qquad b=\frac1\pi,
\]
and, for $0\leq v\leq1/2$, define
\[
 m(v)=\min\{(9\pi/2)^{1/3}v^{2/3},\sqrt{\pi v},1\}.
\]
\revision{Write $p_B(v)=(9\pi/2)^{1/3}v^{2/3}$, $p_C(v)=\sqrt{\pi v}$ and
$p_S(v)=1$ for the three model perimeters.}
Complementation gives $I(v)=I(1-v)$. The three candidate families are
corner eighth-balls, edge quarter-cylinders and coordinate slabs.
Their volumes and relative perimeters are
\[
\begin{array}{c|cc}
 &v&P\\\hline
 \text{corner eighth-ball of radius }r&\pi r^3/6&\pi r^2/2\\
 \text{edge quarter-cylinder of radius }r&\pi r^2/4&\pi r/2\\
 \text{coordinate slab of width }v&v&1.
\end{array}
\]
The models realizing $m(v)$ fit inside $Q$, so $I(v)\leq m(v)$.
The candidate profile changes formula at $a$ and $b$
\cite[Conjecture 1.1 and Section 4.1]{Milman}.

\begin{theorem}\label{thm:main}
For $0\le v\le1$, put $w=\min\{v,1-v\}$. Then
\begin{equation}\label{intro:sharp-profile}
 I(v)=\min\{(9\pi/2)^{1/3}w^{2/3},\sqrt{\pi w},1\}.
\end{equation}
For $0<v\leq1/2$, all minimizing sets, modulo null sets and cube
isometries, are as follows:
\[
\begin{array}{c|l}
 0<v<a&\text{corner eighth-balls},\\
 v=a&\text{corner eighth-balls or edge quarter-cylinders},\\
 a<v<b&\text{edge quarter-cylinders},\\
 v=b&\text{edge quarter-cylinders or coordinate slabs},\\
 b<v\leq1/2&\text{coordinate slabs}.
\end{array}
\]
For $1/2<v<1$, the minimizing sets are the complements of the listed sets
at volume $1-v$. At $v=0,1$ the minimizers are $\varnothing$ and $Q$,
respectively, modulo null sets.
\end{theorem}

In the notation of \cite[Section 10]{BrezisOpen}, Theorem~\ref{thm:main}
states that $f_3(v)=I(v)=m(\min\{v,1-v\})$. In particular, the
maximal interval on which $f_3$ equals one is exactly
$[1/\pi,1-1/\pi]$. The theorem also establishes the stronger,
all-minimizers form of \cite[Conjecture 1.1]{Milman} for the unit
cube, with both geometric families retained at each transition.

For $0<v\le1/2$, let $\mathcal F_B(v),\mathcal F_C(v),\mathcal F_S(v)$
be the orbits under cube isometries of, respectively,
\[
 Q\cap B_{(6v/\pi)^{1/3}}(0),\qquad
 Q\cap\{x_1^2+x_2^2<4v/\pi\},\qquad Q\cap\{x_3<v\}.
\]
These families are finite modulo null sets. Let $\mathcal M_v$ be the
minimizing family in Theorem~\ref{thm:main}.
In particular,
\[
 \mathcal M_a=\mathcal F_B(a)\cup\mathcal F_C(a),\qquad
 \mathcal M_b=\mathcal F_C(b)\cup\mathcal F_S(b),\qquad
 \mathcal M_v=\{Q\setminus M:M\in\mathcal M_{1-v}\}\quad(v>1/2).
\]
For $0<v<1$, let $\mathcal A_v$ be the finite-perimeter subsets of $Q$
of volume $v$. For $E\in\mathcal A_v$, define
\begin{equation}\label{stab:eq:global}
 \delta(E)=\Per(E;Q)-I(v),\qquad
 \alpha_v(E)=\min_{M\in\mathcal M_v}|E\triangle M|.
\end{equation}
The optimal constant in the quadratic estimate is
\begin{equation}\label{stab:eq:best-constant}
 C_*(v)=\inf_{\substack{E\in\mathcal A_v\\\alpha_v(E)>0}}
              \frac{\delta(E)}{\alpha_v(E)^2}.
\end{equation}
The open phases are $(0,a)$, $(a,b)$, $(b,1-b)$, $(1-b,1-a)$ and $(1-a,1)$.

\begin{theorem}\label{stab:thm:global}
For every $v\in(0,1)$,
\begin{gather}
 C_*(v)>0,\qquad
 \delta(E)\ge C_*(v)\alpha_v(E)^2\quad(E\in\mathcal A_v),
 \label{intro:stability-quadratic}\\
 \inf_{\substack{E\in\mathcal A_v\\\alpha_v(E)>0}}
          \frac{\delta(E)}{\alpha_v(E)^p}=0\qquad(p<2).
 \label{intro:stability-optimal}
\end{gather}
If $K$ is a nonempty compact subset of any one open phase, then
$\inf_{v\in K}C_*(v)>0$.
\end{theorem}

\subsection{A consequence for cuboids}\label{intro:cuboid}
The profile formula extends to the full family of cuboids appearing
in \cite[Conjecture 1.1]{Milman}. For $0<\beta\leq1$, set
\[
 Q_\beta=(0,\beta)\times(0,1)^2,\qquad
 I_\beta(v)=\inf\left\{\frac{\Per(E;Q_\beta)}{\beta}:
                         E\subset Q_\beta,\ |E|=\beta v\right\},
\]
where the infimum is over finite-perimeter sets. Thus both volume
and perimeter are normalized by $\beta$, and $I_1=I$. For
$0\leq v\leq1/2$, put
\[
 m_\beta(v)=\min\left\{
       \left(\frac{9\pi}{2\beta}\right)^{1/3}v^{2/3},
       \sqrt{\pi v},1\right\},\qquad
 a_\beta=\frac{4\pi\beta^2}{81}.
\]

\begin{corollary}\label{cor:cuboid-profile}
For every $0<\beta\leq1$ and $0\leq v\leq1$,
\[
 I_\beta(v)=m_\beta\bigl(\min\{v,1-v\}\bigr).
\]
For $0<v\leq1/2$, the profile is attained by a corner eighth-ball
when $v\leq a_\beta$, by a quarter-cylinder along an edge of length
$\beta$ when $a_\beta\leq v\leq b$, and by the coordinate slab
$Q_\beta\cap\{x_3<v\}$ when $b\leq v\leq1/2$.
At either transition both indicated models attain
the minimum. Complementation gives optimal sets for $v>1/2$.
\end{corollary}

This proves the profile and existence assertion in
\cite[Conjecture 1.1]{Milman}. Quantitative stability for each fixed
$Q_\beta$ is stated in Remark~\ref{stab:rem:cuboid-stability}.

\subsection{Earlier results and the remaining obstruction}
In two dimensions, the corresponding profile of the square is
$\min\{\sqrt{\pi w},1\}$, where $w=\min(v,1-v)$.
Brezis--Bruckstein~\cite{BrezisBruckstein} prove a sharp relative
inequality for the square; related classifications on flat surfaces
are discussed by Howards--Hutchings--Morgan~\cite{HHM}.
Adding a third interval introduces a cylindrical phase, but comparing
the three model formulas does not exclude other minimizing surfaces.

The standard reflection argument identifies $I$ with the normalized
profile on $T=\R^3/(2\mathbb Z)^3$ and sends a cube minimizer to a global
minimizer on $T$ \cite[Remark 2.4]{Milman}. See
Lemma~\ref{g:reflection} for the normalization. The reflected boundary
is smooth, embedded and volume-constrained stable. For the periodic
variational theory see Ritor\'e--Ros~\cite{RitoreRos} and
Hauswirth--P\'erez--Romon--Ros~\cite{HPRR}. Pedrosa--Ritor\'e~\cite{PR}
treat the singly periodic problem. Section~\ref{intro:exclusion} below
distinguishes the classical genus bounds from the description by
occupied coordinate faces in Lemma~\ref{g:downset}.

Small-volume results follow from Morgan--Johnson~\cite{MJ} and
Morgan~\cite{MorganPoly}, as described in \cite[Section 1.2]{Milman}.
Hadwiger~\cite{Hadwiger} proved $I(1/2)=1$ (see also \cite{BM}).
Acerbi--Fusco--Morini~\cite[Theorem 5.3 and Remark 5.4]{AFM}
establish minimality of slabs near half volume; Glaudo~\cite{Glaudo}
gives another proof of the near-half-volume inequality for cubes.
For profile concavity, variational inequalities and connectedness,
see \cite{BR,SZ,Ros}.

There are also sharp results for restricted competitors.
Chambers--Mouill\'e~\cite{ChambersMouille} solve the relative problem
for orthogonal polyhedra in the three-dimensional cube.
Chleb\'ik--Cianchi--Fusco~\cite{CCF} characterize equality under
one-dimensional Steiner symmetrization. Barchiesi--Cagnetti--Fusco~\cite{BCF}
treat symmetrization in arbitrary codimension and prove quantitative
estimates for convex sets.

Milman~\cite[Theorem 1.2]{Milman} obtains explicit ranges for
$(0,\beta)\times(0,1)^2$, $0<\beta\leq1$, with volume and perimeter
normalized by $\beta$. For the unit cube, his result establishes the
ball phase for $0<v\leq v^{\min}$, where
$v^{\min}\simeq0.120582$. For
$\beta\leq0.919431\ldots$, it covers the entire proposed ball phase.
He also obtains explicit cylindrical and planar ranges for smaller
aspect ratios. At $\beta=1$ the two explicit cylindrical and planar
intervals in \cite[Theorem 1.2]{Milman} are empty.

OpenAI~\cite[Theorem 1.1]{OpenAI2026} gives a proof of the complete
profile and classification of minimizers for the cubic flat three-torus,
including both transition volumes. Under the reflection and normalization
in Lemma~\ref{g:reflection}, its profile and equality classification
correspond to Theorem~\ref{thm:main}.
\revision{The first version of this paper
(\href{https://arxiv.org/abs/2610.08644v1}{arXiv:2610.08644v1})
was submitted on October 6, 2026, before the OpenAI preprint appeared
in the public GitHub release.}

The endpoint argument in \cite[Sections 4 and 6]{OpenAI2026} estimates
nested planar sections and uses a calibration of the complementary
bottom section. Our argument at $b$ uses the Gauss-map moments,
curvature measures of projections and boundary traces in
Sections~\ref{sec:reduction}--\ref{sec:range-exclusion}.
Section~\ref{sec:stability} proves the quantitative estimates stated above.

\subsection{Reduction to the transition volumes}
Our starting point is Milman's reduction of the cubical profile
conjecture to its two transition values
\cite[Theorem 1.3(2) and Proposition 4.6]{Milman}. At $\beta=1$, the required equalities are
\begin{equation}\label{intro:contacts}
 I(a)=\frac{2\pi}{9},\qquad I(b)=1.
\end{equation}
By \cite[Proposition 4.6]{Milman}, the two equalities
in \eqref{intro:contacts} and complement symmetry give $I=m$,
with classification at the interior volumes of each phase.
In \cite[Corollary 1.4]{Milman}, the reduction to $I_\beta(b)=1$
for $\beta\leq0.919431\ldots$ uses the ball phase already proved there.
Lemma~\ref{p:first-contact} combines Ros's genus bound
\cite[Theorem 19]{Ros2006}, the curvature--area estimate
\cite[Theorem 7(2)]{HPRR} and the profile differential inequality of
\cite[\S7]{BP}, in the form used in \cite[Proposition 4.5]{Milman}, to obtain
$I(b)=1\Longrightarrow I(a)=2\pi/9$. For $Q_\beta$, the corresponding
inequalities are \eqref{p:cuboid-hersch} and
\eqref{p:cuboid-differential}, used in Section~\ref{p:cuboid-profile}.

For a product slab of height $L$ over a closed base with normalized
isoperimetric profile $I_0$, the classical BV slicing inequality
\cite[Proposition 3.4]{BCF}, in the slab form of
\cite[Proposition 2.5 and Corollary 2.7]{Milman}, gives the lower bound
\begin{equation}\label{intro:slicing}
 \inf_{\substack{u\in BV((0,L);[0,1])\\ \int_0^L u\,dt=Lv}}
 \frac1L\left(\int_0^L\sqrt{u'(t)^2+I_0(u(t))^2}\,dt
                    +|D^su|((0,L))\right).
\end{equation}
Here $Du=u'\,dt+D^su$. For the unit-cube comparison, take the
flat two-torus as base, so that
\begin{equation}\label{intro:slicing-base}
 L=1,\qquad
 I_0(s)=\min\{\sqrt{\pi s},1,\sqrt{\pi(1-s)}\}
 \quad(0\le s\le1).
\end{equation}
The profile in \eqref{intro:slicing-base} is continuous, concave and
symmetric about $1/2$. By \cite[Proposition 2.11]{Milman}, the infimum
in \eqref{intro:slicing} is then the profile of the two-dimensional
weighted model slab. At $v=b=1/\pi$ it is strictly less than $1$
\cite[Section 4.2]{Milman}. We study the three-dimensional minimizers
at this volume.

\subsection{Exclusion of higher genus at $v=b$}\label{intro:exclusion}

Milman's ODE analysis \cite[\S4.3]{Milman} combines the low-genus
classification of Ritor\'e--Ros \cite[Proposition 4.4]{Milman} with
classical genus bounds. Ros's theorem \cite[Theorem 19]{Ros2006}
gives $g(\Sigma)\le3$ for a connected volume-constrained stable CMC
surface in a flat three-torus. A global isoperimetric boundary of
genus at most one is a sphere, a tube around a shortest closed
geodesic, or a pair of parallel totally geodesic tori. Consequently,
for $0<v\le1/2$ the classical reduction already gives
\begin{equation}\label{intro:classical-genus}
 I(v)<m(v)\quad\Longrightarrow\quad
 \Sigma\text{ is connected},\qquad g(\Sigma)\in\{2,3\}.
\end{equation}

We apply this reduction at $v=b$. Coordinate compression selects a
minimizer $E$ that is downward monotone in each coordinate, and its
reflection has boundary $\Sigma\subset T$. The coordinate-graph
property of nonflat isoperimetric surfaces is due to Ritor\'e; see
\cite[Proposition 2.3 and its proof]{Ros}. If $P=I(b)<1$, concavity,
$I(1/2)=1$ and \eqref{g:profile-slope} give
\begin{equation}\label{intro:positive-curvature-at-b}
 h\ge I'_+(b)\ge\frac{I(1/2)-I(b)}{1/2-b}
     =\frac{1-P}{1/2-b}>0,
 \qquad h=\kappa_1+\kappa_2.
\end{equation}
The three classical models restrict to a corner eighth-ball, an edge
quarter-cylinder and a coordinate slab in $Q$. At $b$ their perimeters
are
\begin{equation}\label{intro:standard-perimeters-at-b}
 p_B(b)=\left(\frac9{2\pi}\right)^{1/3}>1,
 \qquad p_C(b)=p_S(b)=1,
\end{equation}
so they cannot have perimeter $P<1$.

The genus--periodicity relation is described, for reflection-symmetric
stable CMC surfaces, in \cite[Theorems 6.11, 6.13--6.15]{RitoreThesis}.
Ros~\cite[Theorems 1--2]{Ros2007Periodic} shows that, for embedded
stable CMC surfaces of genus two or three in a flat three-torus,
the connected components of their Euclidean lifts are doubly or
triply periodic, respectively.

In the present coordinate-monotone setting, Lemma~\ref{g:downset}
expresses this topology through the occupied areas $A_{i,1}$ of
the faces $x_i=1$:
\begin{equation}\label{intro:remaining-genera}
 g(\Sigma)=\#\{i:A_{i,1}>0\}\in\{2,3\}.
\end{equation}
These areas also enter the integral identities of Lemma~\ref{g:stress},
which relate the face traces to the normal moments used in the
convex-geometric estimates. In genus two, after a permutation,
$A_{x,1},A_{y,1}>0$ and $A_{z,1}=0$; the latter gives
$\overline E\subset\{z\le1-\varepsilon\}$ for some $\varepsilon>0$
and a compact lift to $(\R/2\mathbb Z)^2\times\R$.
The classical estimate \cite[Theorem 7(2)]{HPRR} gives $h^2P\le\pi$
in our normalization (Lemma~\ref{g:hersch}).
Proposition~\ref{s:n3} excludes genus three by an exact rational
calculation of Bernstein coefficients.

\revision{The main step is the genus-two exclusion. The Gauss-map pushforward of
the second fundamental form determines an auxiliary convex body. Its
mixed volumes represent the Jacobi form on functions pulled back by
the Gauss map (Lemma~\ref{g:support}), and the area measures of its
coordinate projections encode weighted curvature on the coordinate
sections (Lemma~\ref{g:full-wall}). This gives the section inequality
\eqref{g:raw-dual-eq}. The classical height estimate
$P^3\ge16\pi v^2/9$ of \cite[Proposition 12 and Theorem 13]{HPRR}
is recorded in Proposition~\ref{s:height-stability}.
Proposition~\ref{s:q-bound} uses two mean-zero stability tests to obtain
$hv/P<7/12$. These estimates and the curvature-distribution bounds
apply to the same surface.}
Sections~\ref{sec:parameter-comparisons} and~\ref{sec:range-exclusion}
\revision{combine these inequalities to exclude genus two at $b$, including
$P=1$ when $h>0$. Together with Proposition~\ref{s:n3}, this proves
Theorem~\ref{t:coverage-theorem}.}

Lemmas~\ref{p:second-contact}--\ref{p:first-contact} give
$I(b)=1$ and $I(a)=2\pi/9$. The propagation result
\cite[Proposition 4.6]{Milman} then gives the profile and classifies
minimizers for $0<v<a$, $a<v<b$ and $b<v\le1/2$; reflection fixes
their positions in $Q$. At $a$, the curvature--area argument in the
proof of \cite[Corollary 15]{HPRR} applies directly to every original
minimizer; \eqref{p:first-transition-exclusion} excludes genus at least
two. At $b$, parallel variations give a slab when $h=0$. For $h>0$,
coordinate compression and Theorem~\ref{t:coverage-theorem} leave a
cylinder or a slab; the mean curvature may change at each compression
but always satisfies \eqref{g:profile-slope}.
Lemma~\ref{p:recovery} reverses these compressions at $b$.
The width comparison \cite[Lemma 2.33]{Milman}, together with $I_1(b)=1$,
gives $I_\beta(b)=1$; the corresponding profile estimates complete
the cuboid profile in Section~\ref{p:cuboid-profile}.

\subsection{Quantitative stability and competing models}
Theorem~\ref{stab:thm:global} follows from the second variation at the
reflected models, Theorem~\ref{thm:main}, and BV compactness. Its bound
is uniform on compact subsets of each open phase. For related quantitative
estimates and selection arguments see \cite{FMP,CL,AFM}.

\revision{Near $v_*\in\{a,b\}$, the adjacent model families remain separated
in $L^1$, while their perimeter functions cross with different slopes.
Theorem~\ref{stab:thm:multiwell} estimates the deficit uniformly using
the distances and perimeter excesses of both families.}
Corollary~\ref{stab:cor:degeneration} gives
\[
 C_*(v)\asymp |v-v_*| \quad(0<|v-v_*|\ll1),\qquad C_*(v_*)>0.
\]

\section{Variational reduction and topology}\label{sec:reduction}

\subsection{Reflection, volume-constrained stability and low-genus classification}\label{g:foundations}

Let $T=\R^3/(2\mathbb Z)^3$ with volume $dx/8$ and area $d\sigma=dA/8$.
Write $I_T$ for the isoperimetric profile with these measures. For
$E\subset Q$, write $\widehat E\subset T$ for its even reflection in
all three coordinates.

\begin{lemma}\label{g:reflection}
$I_T=I$ on $[0,1]$. If $E$ is isoperimetric in $Q$, then $\widehat E$
is isoperimetric in $T$.
\end{lemma}
\begin{proof}
Apply the doubling argument of \cite[Remark 2.4]{Milman} in each
coordinate. Each doubling preserves the normalized profile, and
the reflected set has physical measures
\begin{equation}\label{g:reflection-normalization}
 |\widehat E|_{\mathrm{phys}}=8|E|,\qquad
 \Per_{\mathrm{phys}}(\widehat E;T)=8\Per(E;Q).
\end{equation}
Thus $\widehat E$ attains $I_T(|E|)=I(|E|)$ among all torus competitors.
\end{proof}

Isoperimetric sets exist for $0<v<1$ by the direct method \cite{Maggi}.
For an isoperimetric set $\widehat E$ in $T$, write $\Sigma=\partial\widehat E$.
Its regular representative has smooth embedded boundary by
\cite[Proposition 2.20(1)]{Milman}. Let $\nu$ be the outward normal,
$B=d\nu$, and $P=\int_\Sigma d\sigma$. The mean curvature is
$h=\tr B=\kappa_1+\kappa_2$ in our convention. For a normal variation
with speed $f$, the first variation gives a common constant $h$ on
all components \cite[Proposition 2.21]{Milman}:
\begin{equation}\label{g:first-variation}
 v'(0)=\int_\Sigma f\,d\sigma,\qquad P'(0)=h v'(0).
\end{equation}
The constant-mean-curvature equation is analytic and elliptic in
local graph coordinates. Its interior analytic regularity gives
analyticity of $\Sigma,\nu,B$ \cite[Section 5.8, Theorem 5.8.6]{Morrey}.
Define
\[
 Q(f,g)=\int_\Sigma(\nabla_\Sigma f\cdot\nabla_\Sigma g-|B|^2fg)\,d\sigma
 \quad(f,g\in H^1(\Sigma)).
\]
The volume-constrained second variation \cite[Equation (3.1)]{Milman}
gives
\begin{equation}\label{g:stability}
 \int_\Sigma f\,d\sigma=0\quad\Longrightarrow\quad Q(f,f)\ge0.
\end{equation}

We use the classical component test \cite[Lemma 2.1]{Ros}.
Let $\Sigma_1,\Sigma_2$ be distinct components, with normalized
areas $P_1,P_2$. The function equal to $1$ on $\Sigma_1$,
$-P_1/P_2$ on $\Sigma_2$, and zero elsewhere has integral zero, and
\begin{equation}\label{g:component-test}
 Q(f,f)=-\int_{\Sigma_1}|B|^2d\sigma
         -\left(\frac{P_1}{P_2}\right)^2\int_{\Sigma_2}|B|^2d\sigma .
\end{equation}
By \eqref{g:stability}, a boundary with a component on which
$B\not\equiv0$ is connected.

For $0<v\le1/2$, the low-genus classification of global minimizers
on $T$ gives a ball, a tube around a shortest closed geodesic,
or a slab between two parallel totally geodesic tori
\cite[Proposition 4.4]{Milman}, based on \cite{RitoreRos}.
The shortest nonzero vectors of $(2\mathbb Z)^3$ and of its dual
lattice are parallel to coordinate axes. Ros's theorem
\cite[Theorem 19]{Ros2006} gives $g(\Sigma)\le3$ for a smooth, closed,
connected, orientable, two-sided, volume-constrained stable CMC surface
in a flat three-torus. Together with the preceding classification,
this restricts a nonstandard isoperimetric boundary to genera two and
three, as in \cite[\S4.3]{Milman}.

The following bound is \cite[Theorem 7(2)]{HPRR} with the trivial
symmetry group. Its proof uses the balanced conformal-map method of
Hersch~\cite{Hersch} and Yang--Yau~\cite{YangYau}.
We state the bound under the stability hypothesis used in that proof.

\begin{lemma}\label{g:hersch}
For a smooth, closed, connected, two-sided, volume-constrained stable
CMC surface in $T$ of genus two or three, $h^2P\leq\pi$.
\end{lemma}
\begin{proof}
Put $H=h/2$ and $A=\mathcal H^2(\Sigma)$. The proof of
\cite[Theorem 7(2)]{HPRR} uses only \eqref{g:stability} and the
coordinates of a balanced meromorphic map $\Sigma\to\Sph$ of degree
at most $g(\Sigma)$ for $g(\Sigma)\in\{2,3\}$. For the balancing
step under the stability hypothesis, see also \cite[Lemma 1]{Frensel}.
It therefore gives
\begin{equation}\label{g:hersch-physical}
 H^2\mathcal H^2(\Sigma)\le2\pi.
\end{equation}
Since $A=8P$, this is $h^2P=(2H)^2A/8=H^2A/2\le\pi$.
\end{proof}

\begin{lemma}\label{g:profile}
The profile $I$ is continuous, concave, and symmetric about $1/2$,
with $I(1/2)=1$. For an isoperimetric boundary $\Sigma$ at
$v_0\in(0,1)$ there is an open interval $U\ni v_0$ and a smooth
$J:U\to\R$ with $I|_U\le J$ and $J(v_0)=I(v_0)$. It satisfies
\begin{align}
 J(w)^2J''(w)+J(w)(J'(w))^2
 &=\frac\pi2\chi(\Sigma)\qquad(w\in U),\label{g:profile-ode}\\
 I'_+(v_0)&\le J'(v_0)=h\le I'_-(v_0).\label{g:profile-slope}
\end{align}
\end{lemma}
\begin{proof}
Use Lemma~\ref{g:reflection} and \cite[Proposition 2.3]{Milman} for
continuity, concavity and symmetry. The parallel-surface comparison
below is the classical profile method of \cite[\S7]{BP}; see also
\cite[Theorem 9(b)]{HPRR}.
The half-volume identity is
\cite[Equation (4.1)]{Milman}, see also \cite{Hadwiger,BM}.
For $|t|<\varepsilon$ small, let $E_t$ be bounded by the parallel surface
$\Sigma_t=\{p+t\nu(p):p\in\Sigma\}$ in the flat torus, with $E_0=\widehat E$.
Put $V(t)=|E_t|_{\rm phys}/8$ and $P(t)=\mathcal H^2(\Sigma_t)/8$.
The parallel-surface formulas of \cite[Section 2, equations (2)--(3)]{HPRR},
with area and volume divided by eight, give
\begin{equation}\label{g:parallel}
 P(t)=P+hPt+\frac{\pi\chi(\Sigma)}4t^2,\qquad
 V(t)=v_0+Pt+\frac{hP}2t^2+\frac{\pi\chi(\Sigma)}{12}t^3.
\end{equation}
Shrink $\varepsilon$ so that $P(t)>0$ and $V(t)\in(0,1)$.
On $U=(V(-\varepsilon),V(\varepsilon))$ set $J(w)=P(V^{-1}(w))$.
The sets $E_t$ are competitors, so $I(V(t))\le P(t)$ with equality
at $t=0$. Since $V'=P$, the classical reparametrization in
\cite[Theorem 9(b)]{HPRR} gives
\begin{equation}\label{g:support-derivatives}
 J'(V(t))=\frac{P'(t)}{P(t)},\qquad
 (J^2J''+J(J')^2)(V(t))=P''(t)=\frac\pi2\chi(\Sigma).
\end{equation}
Thus \eqref{g:profile-ode} holds on all of $U$; at the contact point
it gives the differential comparison in \cite[Proposition 4.5]{Milman}.
At $v_0$, \eqref{g:parallel} gives $J'(v_0)=h$ and
\begin{equation}\label{g:support-curvature}
 J''(v_0)=\frac{\pi\chi(\Sigma)/2-h^2P}{P^2}
 =-\frac1{P^2}\int_\Sigma|B|^2d\sigma .
\end{equation}
Right and left difference quotients of $I-J\le0$ at $v_0$ give
\eqref{g:profile-slope}.
\end{proof}

\subsection{Coordinate compression, genus, and integral identities}\label{g:selection}

A representative $E\subset Q$ is downward monotone if
\[
 x\in E,\quad 0<y_i\le x_i\ (i=1,2,3)\quad\Longrightarrow\quad y\in E.
\]
In the decomposition $Q=(0,1)^2\times(0,1)$ along direction $e_i$,
the compression $S_iE$ has fibres $(0,|E_{x'}|)$, where
$E_{x'}=\{t:(x',t)\in E\}$. Let $P_{\rm Mink}$ be relative exterior
Minkowski perimeter and $P_{\rm BV}(E)=\Per(E;Q)$.
Ros~\cite[Proposition 3.6]{Ros} gives
$|S_iE|=|E|$ and $P_{\rm Mink}(S_iE)\le P_{\rm Mink}(E)$.
When $P_{\rm Mink}(E)<\infty$, take $t\downarrow0$ along a
sequence attaining its defining liminf. For
$u_t(x)=\max\{1-\operatorname{dist}(x,E)/t,0\}$,
\[
 \begin{aligned}
 \|u_t-\chi_E\|_1
 &\le|\{x\in Q:\operatorname{dist}(x,E)<t\}|-|E|\longrightarrow0,\\
 \int_Q|\nabla u_t|
 &\le\frac{|\{x\in Q:\operatorname{dist}(x,E)<t\}|-|E|}{t}.
 \end{aligned}
\]
Thus $P_{\rm BV}(E)\le P_{\rm Mink}(E)$ by BV lower semicontinuity.

For a minimizer, define $F:T\to\overline Q$ by
$F_i(x)=\operatorname{dist}_{\R/2\mathbb Z}(x_i,0)$ for $i=1,2,3$.
Then $F|_{\overline Q}=\operatorname{Id}$, $F$ is $1$-Lipschitz and
$F(\overline{\widehat E})=\overline E$. For $x\in Q$,
$|x-F(y)|\le d_T(x,y)$ for $y\in\widehat E$, so
$\operatorname{dist}(x,E)\le\operatorname{dist}_T(x,\widehat E)$.
The reverse inequality follows from $E\subset\widehat E$. By reflection and the tubular neighbourhood formula,
\[
 \lim_{t\downarrow0}
 \frac{|\{x\in Q:\operatorname{dist}(x,E)<t\}|-|E|}{t}
 =\frac18\mathcal H^2(\partial\widehat E)=I(v).
\]
Thus
\begin{equation}\label{g:compression-perimeter}
 |S_iE|=|E|,\qquad
 I(v)\le P_{\rm BV}(S_iE)\le P_{\rm Mink}(S_iE)\le I(v).
\end{equation}
The nested-section argument in the proof of
\cite[Proposition 2.16]{Milman} applies to each pair of coordinates here:
for $j\ne i$, if $E_{x'+ae_j}\subset E_{x'}$ modulo null sets for
almost every $x'$ and fixed $a>0$, then
\[
 |E_{x'+ae_j}|\le |E_{x'}|,\qquad
 (S_iE)_{x'+ae_j}\subset(S_iE)_{x'}.
\]
Thus a compression preserves monotonicity in the other directions.
Let $E^\circ$ be the smooth open representative of a minimizer with
this almost-everywhere monotonicity. For a ball $B\subset E^\circ$ with
$B-ae_j\subset Q$, monotonicity gives $|(B-ae_j)\setminus E|=0$.
Every neighbourhood of a point of $\partial E^\circ\cap Q$ contains
a subset of positive measure outside $E^\circ$, so $B-ae_j\subset E^\circ$.
Taking closures gives a downward monotone representative of $S_zS_yS_xE$.

For $c\in\{0,1\}$ let $S_{i,c}$ denote the occupied trace of $E$
on the open face $x_i=c$, identified with $(0,1)^2$:
\[
 S_{i,c}=\{x'\in(0,1)^2:\operatorname{Tr}_{x_i=c}\chi_E(x')=1\},
 \qquad A_{i,c}=\mathcal H^2(S_{i,c}),\qquad
 L_{i,c}=\Per(S_{i,c};(0,1)^2).
\]
For geometric containment of the trace we use its regular closed
representative in the closed face. The coordinate $i$ is active if
$A_{i,1}>0$.

For the classical genus--periodicity structure, see
\cite[Theorems 6.11, 6.13--6.15]{RitoreThesis} and
\cite[Theorem 1]{Ros2007Periodic}. More generally, a nonplanar compact
two-sided stable CMC surface immersed in $\R^3/\Lambda$, where
$\Lambda$ is a discrete translation group of rank $k$, satisfies
$g\le k$ by \cite[Theorem 2]{Ros2007Periodic}.
The following lemma gives the formulation in terms of the areas
$A_{i,1}$, together with positive separation from each inactive
upper face. The general bound $g\le3$ is the classical theorem
\cite[Theorem 19]{Ros2006}.

\begin{lemma}\label{g:downset}
Suppose $E$ is a downward monotone cube minimizer whose reflected
boundary $\Sigma$ is connected, has genus at least two, and has
$h>0$. Then $g(\Sigma)=\#\{i:A_{i,1}>0\}\le3$. If $A_{i,1}=0$,
the closed representative of $E$ has positive distance from $x_i=1$.
In particular the hypotheses hold for the above choice at any
$v\in(0,1/2)$ with $I(v)<m(v)$.
\end{lemma}
\begin{proof}
In the last case, the low-genus classification in
Section~\ref{g:foundations} and $P=I(v)<m(v)$ exclude low genus.
Using \eqref{g:profile-slope}, concavity and $I(1/2)=1$ gives
\[
 h\ge I'_+(v)\ge\frac{I(1/2)-I(v)}{1/2-v}
 =\frac{1-P}{1/2-v}>0.
\]
Connectedness then follows from \eqref{g:component-test}.

We use the normal-component argument of Ritor\'e recorded in
\cite[Proposition 2.3]{Ros}. For $i=x,y,z$, downward monotonicity
and reflection in the two other coordinates give
\[
 \nu_i\ge0,\qquad (\Delta_\Sigma+|B|^2)\nu_i=0
 \qquad\text{on }\Sigma\cap\{0<x_i<1\}.
\]
The equation is \cite[Equation (3.2)]{Milman} at constant density.
The coordinate-graph argument in \cite[Proposition 2.3 and its proof]{Ros}
then gives $\nu_i>0$ on the strip: the identically zero alternative
would make $e_i$ a nowhere-zero tangent field, contrary to
$\chi(\Sigma)<0$.
On $\Sigma\cap\{x_i=c\}$ for $c=0,1$, reflection gives
$\nu_i=0$, $\nabla_\Sigma x_i=e_i$ and $Be_i=B_{ii}e_i$. The inward
conormal of the strip is $(1-2c)e_i$. Hopf's lemma applied to
$\nu_i$ gives $(1-2c)\partial_{e_i}\nu_i>0$ there, hence
\begin{equation}\label{g:hopf}
 B_{xx}|_{x=0}>0,\qquad B_{xx}|_{x=1}<0,\qquad
 \kappa_{x,c}=h-B_{xx}|_{x=c}\quad(c=0,1).
\end{equation}
Here $\kappa_{x,c}$ is the signed curvature in the plane $x=c$,
and the formulas hold for each coordinate by permutation.

At a point of $\Sigma$ on $x=y=1$, $0<z<1$, the two reflections
force $\nu=\pm e_z$, so \eqref{g:hopf} would give
\begin{equation}\label{g:upper-edge-exclusion}
 h=B_{xx}+B_{yy}<0.
\end{equation}
At a point with all three coordinates in $\{0,1\}$, the three
reflections would force $\nu=0$. Thus $\Sigma$ is disjoint from
$\{x=y=1,\ 0\le z\le1\}$. If $\overline E$ met this edge, its
intersection would be $\{(1,1,z):0\le z\le z_0\}$ by downward monotonicity.
For $z_0<1$, $(1,1,z_0)$ would lie on $\Sigma$. For $z_0=1$,
downward monotonicity gives $\overline E=[0,1]^3$, contrary to $v<1$.
The same argument applies to the other two edges with both fixed
coordinates equal to one.

Let $D\subset[-1,1]^3$ be the union of the eight reflected closed
representatives before the opposite faces are identified. Downward
monotonicity gives $0\in\operatorname{int}D$ and
$[0,p]\subset D$ for every $p\in D$. At a point $p\in\Sigma$ in
this box, reflection and the strict inequalities for $\nu_i$ give
\[
 p\cdot\nu(p)=\sum_{0<|p_i|<1}p_i\nu_i(p)>0.
\]
The sum is nonempty because $\Sigma\cap\{0,\pm1\}^3=\varnothing$.
For $p\in D$ and $0<t<1$, a boundary point $tp$ would lie in
$\Sigma\cap(-1,1)^3$. The inequality $p\cdot\nu(tp)>0$ would put
$(t+\varepsilon)p$ outside $D$ for small $\varepsilon>0$, contrary
to $[0,p]\subset D$. Thus
\[
 tp\in\operatorname{int}D\qquad(p\in D,\ 0<t<1).
\]
Every ray from $0$ consequently meets $\partial D$ exactly once.
Radial projection $\partial D\to\Sph$ is a continuous bijection,
hence a homeomorphism by compactness.

Suppose $D_x^+=D\cap\{x=1\}\ne\varnothing$. Monotonicity gives
$(1,0,0)\in D_x^+$, and $(1,0,0)\notin\Sigma$ puts this point
in the relative interior of $D_x^+$. Its relative boundary lies
in $\Sigma$, is disjoint from the edges of the face, and satisfies
\[
 \nu_x=0,\qquad
 (p-e_x)\cdot\nu(p)=y\nu_y(p)+z\nu_z(p)>0
 \quad(p\in\partial_{\{x=1\}}D_x^+).
\]
The same radial projection in this plane shows that $D_x^+$ is a
closed disk with one simple closed boundary curve. The set
\[
 \partial_{\{x=1\}}D_x^+\cap\{y,z\ge0\}
\]
is one arc joining $y=0$ to $z=0$, and is the free boundary on the
original upper face $x=1$. The argument holds for each coordinate.
In particular,
$D\cap\{x_i=1\}\ne\varnothing$ if and only if $A_{i,1}>0$.
If $A_{i,1}=0$, compactness gives
$\operatorname{dist}(D,\{x_i=1\})>0$.

For $n=\#\{i:A_{i,1}>0\}$ the $2n$ face disks are disjoint by the
edge exclusion. Removing their relative interiors from $\partial D$
gives a sphere with $2n$ boundary circles. The torus quotient glues
these circles in opposite pairs and gives $\Sigma$. Since
$\chi(S^1)=0$, Euler characteristic is unchanged by the gluing:
\[
 2-2g(\Sigma)=\chi(\Sigma)=\chi(S^2)-2n=2-2n.
\]
\end{proof}

Figure~\ref{fig:active-coordinate-sections} shows coordinate traces for
two and three active coordinates. Figure~\ref{fig:reflected-downsets}
shows the regions $E\subset Q$ and their even reflections in $[-1,1]^3$.

\begin{figure}[!htbp]
 \centering
 \includegraphics[width=\linewidth]{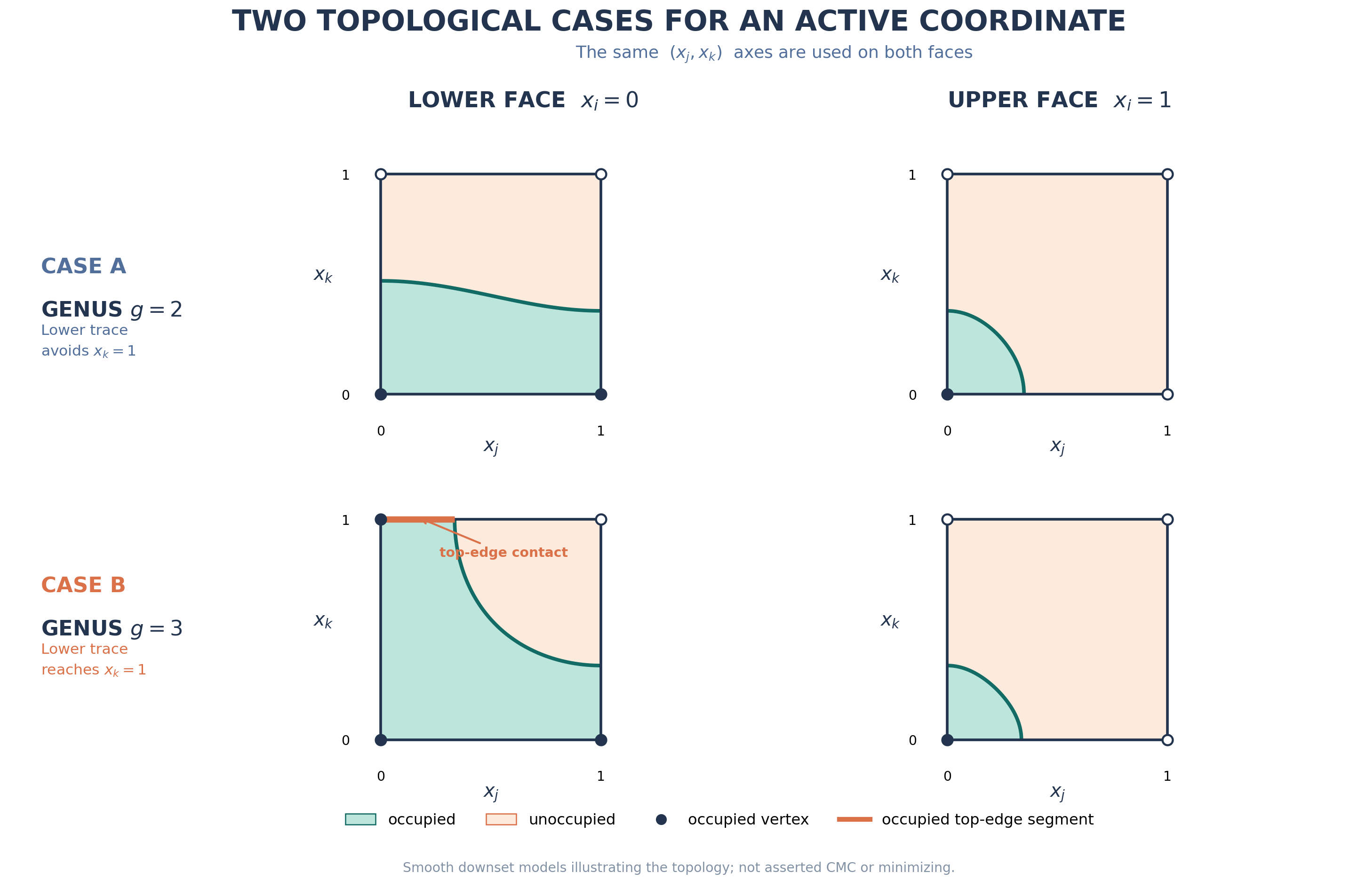}
 \caption{Coordinate traces $S_{i,0}$ and $S_{i,1}$, with $i,j$ active.
 In the first row $A_{k,1}=0$, and in the second $A_{k,1}>0$.
 Both columns use the same $(x_j,x_k)$ coordinates; teal denotes
 the occupied trace. The regions are schematic downward monotone
 examples, without a CMC or minimality assumption.}
 \label{fig:active-coordinate-sections}
\end{figure}

\begin{figure}[!htbp]
 \centering
 \includegraphics[width=\linewidth]{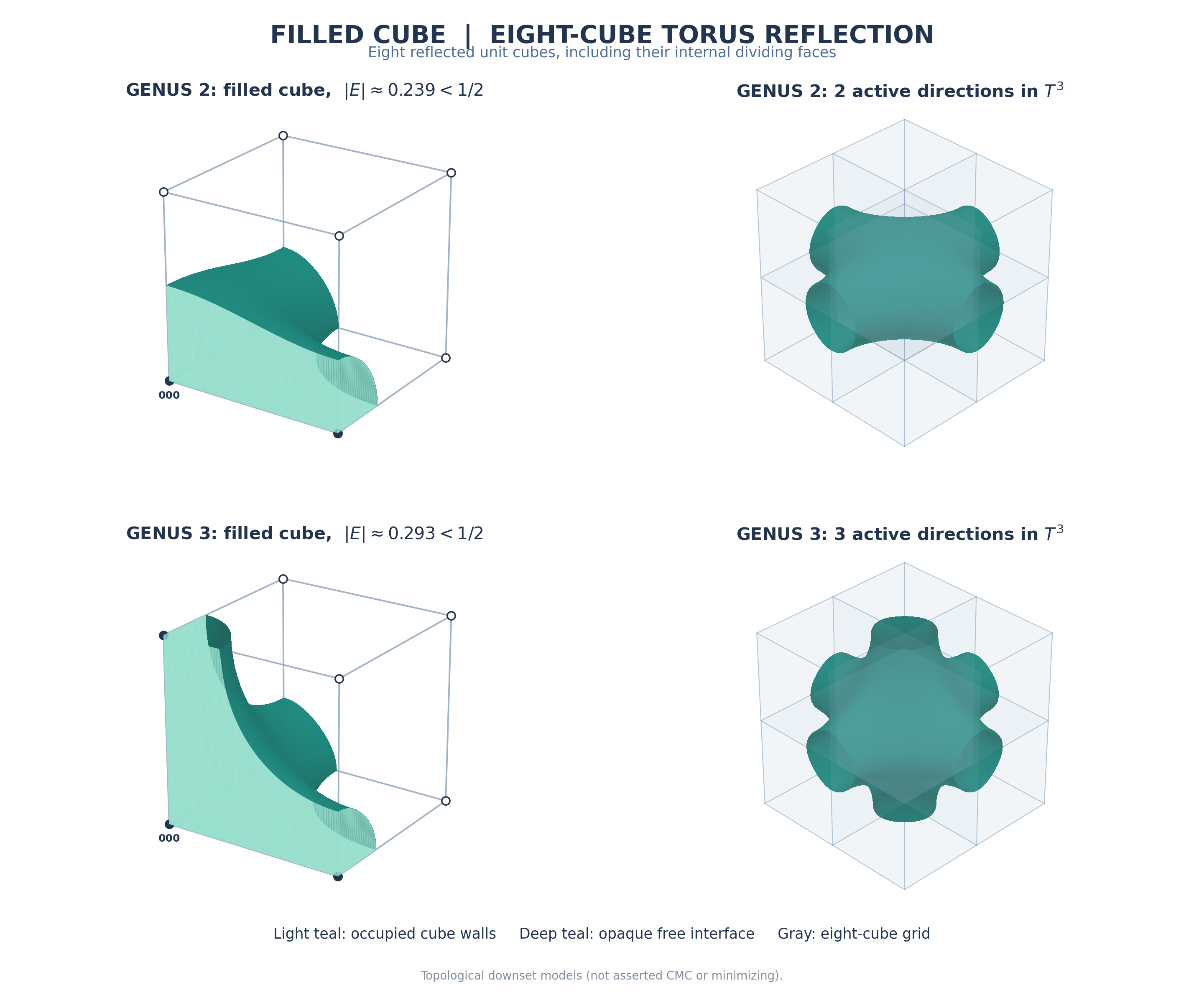}
 \caption{Downward monotone regions $E\subset Q$ (left) and their
 eightfold reflections in $[-1,1]^3$ (right). The first row has $x,y$
 active and $z$ inactive; all three directions are active in the second.
 Identifying opposite faces gives boundaries of genera two and three.
 Light teal marks occupied cube faces, dark teal the free boundary,
 and grey the internal dividing faces of the eight unit cubes.
 These are schematic topological examples.}
 \label{fig:reflected-downsets}
\end{figure}

For the selected minimizers with $h>0$ and $g(\Sigma)=n\ge2$,
the positive distance in each inactive direction gives a compact
lift to $(\R/2\mathbb Z)^n\times\R^{3-n}$. The lift agrees
locally with the reflected set in $T$.
Write $\Gamma=\partial E\cap Q$ for the free boundary in the original
cube. With $v=|E|$ define
\begin{equation}\label{g:normal-moments}
 M_i=\int_\Sigma\nu_i^2\,d\sigma=\int_\Gamma\nu_i^2\,dA,\qquad
 N_i=\int_\Sigma|\nu_i|\,d\sigma=\int_\Gamma\nu_i\,dA,\qquad
 m_i=\frac{M_i}{P},\qquad \sum_iM_i=P .
\end{equation}
Here $\nu_i\ge0$ on $\Gamma$. The eight congruent patches and the
reflection invariance of $1,\nu_i^2,|\nu_i|$ give these equalities
and $P=\int_\Gamma dA$.

\FloatBarrier
\begin{lemma}\label{g:stress}
For a downward monotone minimizer and each $i\in\{x,y,z\}$,
\begin{align}
 A_{i,0}-A_{i,1}&=N_i,\label{g:stress-area}\\
 L_{i,0}-L_{i,1}&=hN_i,\label{g:stress-length}\\
 L_{i,1}-hA_{i,1}&=P-M_i-hv=L_{i,0}-hA_{i,0}.\label{g:stress-eq}
\end{align}
\end{lemma}
\begin{proof}
Work in $Q$ and take $i=x$. Integration by parts on $E$ with
the fields $e_x$ and $xe_x$ gives, respectively,
\begin{align}
 0=\int_E\Div e_x\,dx
 &=\int_\Gamma\nu_x\,dA+A_{x,1}-A_{x,0},\label{g:volume-constant}\\
 v=\int_E\Div(xe_x)\,dx
 &=\int_\Gamma x\nu_x\,dA+A_{x,1}.\label{g:volume-fluxes}
\end{align}
On $\Gamma$, the outer conormals at $x=0,1$ are $-e_x,e_x$;
the fields $e_x^T$ and $xe_x^T$ have zero conormal component on
the other faces. Integration by parts on $\Gamma$ gives
\begin{align}
 L_{x,1}-L_{x,0}
 &=\int_\Gamma\Div_\Gamma e_x^T\,dA
   =-h\int_\Gamma\nu_x\,dA=-hN_x,\label{g:surface-constant}\\
 L_{x,1}
 &=\int_\Gamma\Div_\Gamma(xe_x^T)\,dA
   =\int_\Gamma(1-\nu_x^2-hx\nu_x)\,dA
   =P-M_x-h(v-A_{x,1}).\label{g:surface-fluxes}
\end{align}
Use \eqref{g:normal-moments} in \eqref{g:volume-constant} and
combine the resulting \eqref{g:stress-area} with
\eqref{g:surface-constant} and \eqref{g:surface-fluxes}.
\end{proof}

\begin{lemma}\label{g:genus-three}
Let $E$ be a downward monotone minimizer at $0<v\le1/2$
with $P\le1$. Suppose its reflected boundary is connected of
genus three and has $h>0$.
There are numbers $u,k,\ell$ satisfying
\begin{equation}\label{g:genus-three-eq}
 \begin{gathered}
 0<u\leq1/\sqrt3,\qquad 0\leq k\leq\sqrt\pi,
       \qquad 0\leq\ell\leq1-u^2,\\
 \ell\leq1-u^2-k/\pi+(k/\pi)\ell^2,
       \qquad (\ell+ku)^2+\ell^2+\pi u\geq\pi.
 \end{gathered}
\end{equation}
\end{lemma}
\begin{proof}
By \eqref{g:normal-moments}, permute the coordinates so that
$M_x\le P/3$ and put $u=\sqrt{M_x/P}$, $k=hP$, $\ell=L_{x,1}/P$.
The strict strip inequalities in Lemma~\ref{g:downset} give $M_x>0$.
On the plane $x=1$, the closed trace $S_{x,1}$ avoids the sides
$y=1,z=1$ by \eqref{g:upper-edge-exclusion} and the edge argument
following it. Reflect it in $y=0,z=0$. The resulting planar set
has area $4A_{x,1}$ and perimeter $4L_{x,1}$: the traces coincide
across the reflection lines, and the remaining sides are at
positive distance from $S_{x,1}$.

On $x=0$, consider $(0,1)^2\setminus S_{x,0}$. Since $y,z$ are active,
choose $(x_0,1,z_0),(x_1,y_1,1)\in\overline E$ with
$x_0,x_1,y_1,z_0>0$. The coordinatewise lower rectangles of
these points meet $x=0$ in the following two strips, with
$\varepsilon=\min\{y_1,z_0\}$:
\begin{equation}\label{g:lower-face-strips}
 ([0,1]\times[0,\varepsilon])\cup
 ([0,\varepsilon]\times[0,1])\subset\overline{S_{x,0}}
 \qquad\text{in the }(y,z)\text{ plane}.
\end{equation}
Equation \eqref{g:lower-face-strips} uses only that $y,z$ are active.
It separates the complement from $y=0,z=0$. Reflecting the complement
in $y=1$ and $z=1$ gives area $4(1-A_{x,0})$ and perimeter
$4L_{x,0}$, since the traces match. Planar isoperimetry on the two
reflected sets gives
\[
 (4L_{x,1})^2\ge 4\pi(4A_{x,1}),\qquad
 (4L_{x,0})^2\ge 4\pi\bigl(4(1-A_{x,0})\bigr).
\]
By \eqref{g:stress-area},
\begin{equation}\label{eq:body-002}
 A_{x,1}\le L_{x,1}^2/\pi,\qquad
 L_{x,0}^2\ge\pi(1-A_{x,0})=\pi(1-A_{x,1}-N_x).
\end{equation}
The quarter-cylinder of radius $2\sqrt{v/\pi}<1$ gives
$P^2\le\pi v$ for $v\le1/2$. Also $(0,1)\times S_{x,1}\subset E$
modulo null sets and hence $A_{x,1}\le v$. By Cauchy--Schwarz in
\eqref{g:normal-moments}, $N_x^2\le PM_x=P^2u^2$. From
\eqref{g:stress-eq} and \eqref{eq:body-002} we obtain
\begin{align}
 \ell&=1-u^2-\frac hP(v-A_{x,1})\le1-u^2,\notag\\
 \ell&\le1-u^2-\frac hP\left(\frac{P^2}{\pi}
          -\frac{P^2\ell^2}{\pi}\right)
      =1-u^2-\frac{k}{\pi}+\frac{k}{\pi}\ell^2.
 \label{g:genus-three-refinement}
\end{align}
By \eqref{g:stress-length}, \eqref{g:stress-area} and \eqref{eq:body-002},
\begin{equation}\label{eq:body-003}
 L_{x,0}\le P\ell+ku,\qquad
 1-A_{x,0}\ge 1-P^2\ell^2/\pi-Pu .
\end{equation}
Consequently
\begin{equation}\label{eq:body-004}
 \pi-P^2\ell^2-\pi Pu\le\pi(1-A_{x,0})\le L_{x,0}^2
 \le(P\ell+ku)^2 .
\end{equation}
Since $P\le1$ and $u,k,\ell\ge0$,
$\pi\le(P\ell+ku)^2+P^2\ell^2+\pi Pu\le(\ell+ku)^2+\ell^2+\pi u$.
Finally $k^2=h^2P^2\le\pi P\le\pi$ by Lemma~\ref{g:hersch}.
\end{proof}

For genus two, take $x,y$ active and $z$ inactive. By
Lemma~\ref{g:downset}, $\overline E\subset\{z\le1-\varepsilon\}$
for some $\varepsilon>0$. Set
\[
 w(x,y)=\max\bigl(\{z:(x,y,z)\in\overline E\}\cup\{0\}\bigr),
 \qquad D_0=\{(x,y)\in(0,1)^2:w(x,y)>0\}.
\]
Then $E=\{(x,y,z):(x,y)\in D_0,\ 0<z<w(x,y)\}$ modulo null sets.
Closedness gives $\limsup_{(x',y')\to(x,y)}w(x',y')\le w(x,y)$.
At a point of positive height, $\nu_z>0$ and the implicit function
theorem writes the reflected boundary locally as a smooth graph in
$z$. Points just below this graph are interior, which gives
$\liminf_{(x',y')\to(x,y)}w(x',y')\ge w(x,y)$.
At height zero the same lower bound follows from $w\ge0$.
Thus $w$ is continuous on $[0,1]^2$ and smooth where $w>0$
in the reflected base coordinates.

Since $A_{y,1}>0$, choose $\xi,\zeta>0$ with
$(\xi,1,\zeta)\in\overline E$. Downward monotonicity gives
\begin{equation}\label{g:genus-two-lower-box}
 [0,\xi]\times[0,1]\times[0,\zeta]\subset\overline E.
\end{equation}
Consequently,
\begin{equation}\label{g:genus-two-axis-height}
 w(0,y)\ge\zeta>0\qquad(0\le y\le1).
\end{equation}
Interchanging $x,y$ and using $A_{x,1}>0$ gives
$w(x,0)>0$ for $0\le x\le1$.

The following projection and flux estimates are those used in
\cite[Proposition 17, proof, equations (16)--(18)]{HPRR}; the conversion
is $A=8P$, $V=8v$, $H=h/2$, and the reflected footprint has area
$4|D_0|=4N_z$. We retain the intermediate inequalities in $N_z$.
Since $A_{z,1}=L_{z,1}=0$, \eqref{g:stress-area}--\eqref{g:stress-eq}
give $A_{z,0}=N_z$, $L_{z,0}=hN_z$ and $M_z=P-hv$. The graph projection
gives $|D_0|=\int_\Gamma\nu_z\,dA=N_z$. Thus
\begin{equation}\label{g:hole}
 M_z=P-hv,\qquad M_x+M_y=hv,\qquad N_z^2\le PM_z=P(P-hv).
\end{equation}
For the complement of $D_0$ in the plane $z=0$, use
\eqref{g:lower-face-strips} with $z$ as the fixed coordinate
and $x,y$ as the active coordinates. Reflection in $x=1,y=1$
gives a planar set of area $4(1-N_z)$ and perimeter $4hN_z$.
The planar isoperimetric inequality gives
\begin{equation}\label{g:hole-isoperimetry}
 16\pi(1-N_z)\le (4hN_z)^2,\qquad
 1-N_z\le h^2N_z^2/\pi .
\end{equation}

\section{Gauss-map identities and inequalities for coordinate sections}\label{sec:gauss-wall}

Throughout this section $E\subset Q$ is downward monotone. Its reflected
boundary $\Sigma\subset T$ is connected of genus two, $h>0$, and $z$
\revision{is inactive. We construct from the Gauss pushforwards an auxiliary convex
body whose mixed volumes represent $Q$ on Gauss-map tests. Its projection
area measures encode curvature data on the coordinate sections. These
identities lead to the coordinate-section inequality \eqref{g:raw-dual-eq},
which will be used with \eqref{g:active-axis-gates} and
\eqref{g:angular-bounds}.}

\subsection{Gauss pushforwards and the support-function representation}\label{g:gauss}

Gauss--Bonnet on $\Sigma\subset T$ gives
\begin{equation}\label{g:gauss-degree}
 \deg\nu=\frac1{4\pi}\int_\Sigma K\,dA=-1,\qquad
 \int_\Sigma K\,d\sigma=-\frac{\pi}{2}.
\end{equation}
By Section~\ref{g:foundations}, $K$ is analytic on the connected
surface $\Sigma$, and \eqref{g:gauss-degree} gives $K\not\equiv0$.
Thus $\mathcal H^2(\{K=0\})=0$ by \cite{Mityagin}.
The area formula therefore gives integrable densities
\begin{equation}\label{g:gauss-push}
 \nu_*(K\,d\sigma)=-\frac18d\omega,\qquad
 \nu_*(d\sigma)=r\,d\omega,\qquad
 \nu_*(B\,d\sigma)=T_0\,d\omega,
\end{equation}
where $d\omega$ is standard sphere area, of total mass $4\pi$.
For a regular value $n$ of $\nu$,
\begin{equation}\label{g:gauss-fibre-sums}
 r(n)=\frac18\sum_{\nu(p)=n}\frac1{|K(p)|},\qquad
 T_0(n)=\frac18\sum_{\nu(p)=n}\frac{B(p)}{|K(p)|},\qquad
 \int_{\Sph}r\,d\omega=P,\quad |T_0|\leq\|B\|_\infty r.
\end{equation}
The tensors act in the common plane $n^\perp$.
For every smooth tangent sphere field $X$, surface divergence gives
\begin{equation}\label{g:weak-gauss-divergence}
 0=\int_\Sigma\Div_\Sigma(X\circ\nu)\,d\sigma
   =\int_{\Sph}T_0:\nabla X\,d\omega,\qquad
 \Div T_0=0,\quad \tr T_0=hr.
\end{equation}
All tensor equalities below hold almost everywhere, and divergences
are distributional. Write $\mathrm{Id}$ for the tangent identity.
The Cayley--Hamilton identity for $B$ and \eqref{g:gauss-push} imply
\begin{equation}\label{g:gauss-second-moment}
 \nu_*(B^2d\sigma)=S\,d\omega,\qquad S=hT_0+\mathrm{Id}/8.
\end{equation}
For a regular value $n$, the weights $1/(8|K(p)|)$ in
\eqref{g:gauss-fibre-sums} sum to $r(n)$. For $w\in n^\perp$,
Cauchy--Schwarz and \eqref{g:gauss-second-moment} give
\begin{equation}\label{g:gauss-matrix-bound}
 |T_0w|^2
 \le r(n)\sum_{\nu(p)=n}\frac{|B(p)w|^2}{8|K(p)|}
 =r(n)\langle Sw,w\rangle .
\end{equation}
Since $T_0$ is symmetric, $T_0^2\le r(hT_0+\mathrm{Id}/8)$.
By Cayley--Hamilton and $\tr T_0=hr$,
$T_0^2-hrT_0=-\det(T_0)\mathrm{Id}\le (r/8)\mathrm{Id}$.
Hence $\det T_0\ge-r/8$.
Define
\begin{equation}\label{g:C}
 \begin{gathered}
 C=\mathrm{Id}+8h\cof T_0,\qquad \cof C=\mathrm{Id}+8hT_0,\\
 \tr C=2+8h^2r>0,\qquad
 \det C=1+8h^2r+64h^2\det T_0\geq1.
 \end{gathered}
\end{equation}
As $C$ is symmetric and two-dimensional, its positive trace and determinant
give $C>0$.

For a nonempty compact convex set $A\subset\R^3$, let
$h_A(n)=\sup_{x\in A}x\cdot n$. The coefficient of $t_1t_2t_3$
in $\V(t_1A_1+t_2A_2+t_3A_3)$ is $6V(A_1,A_2,A_3)$, where
$V(A,A,A)=\V(A)$. We use the symmetric trilinear extension of $V$
to differences of support functions, including $C^2(\Sph)$
\cite[Lemma 1.7.8 and Theorem 5.1.7]{Schneider}.

\begin{lemma}\label{g:support}
There is a continuous, even, coordinate-reflection-invariant support
function $H$ of a centred convex body $\mathcal K$ with nonempty interior,
such that in distributions
\begin{equation}\label{g:support-eq}
 C=\Hess_{\Sph}H+H\mathrm{Id}.
\end{equation}
For $g,k\in C^\infty(\Sph)$,
\begin{equation}\label{g:mixed}
 Q(g\circ\nu,k\circ\nu)
 =\frac18\int_{\Sph}
 [\cof C(\nabla g,\nabla k)-\tr C\,gk]\,d\omega
 =-\frac34V(g,k,H).
\end{equation}
The identity also holds for $g=h_A$ and $k=h_{A'}$, where $A,A'\subset\R^3$ are nonempty compact convex sets.
\end{lemma}
\begin{proof}
Taking the trace of \eqref{g:support-eq} leads to
\begin{equation}\label{g:support-trace}
 (\Delta_{\Sph}+2)H=2+8h^2r .
\end{equation}
The kernel of $\Delta_{\Sph}+2$ consists of $n\mapsto a\cdot n$.
The right-hand side is even by reflection symmetry, so the equation
has at most one even distributional solution.
We use the Green representation for the Christoffel equation;
see \cite[Theorems 3.3 and 5.1]{Berg} and
\cite[\S8.3.2]{Schneider}. In this normalization the even Green
kernel for $(\Delta_{\Sph}+2)/2$ is
$(1-t\operatorname{atanh}t)/(2\pi)$.
Put $G(t)=t\operatorname{atanh}t$. Its positive partial sums
$G_N(t)=\sum_{j=1}^Nt^{2j}/(2j-1)$ satisfy
\begin{equation}\label{g:green-kernel}
 (\Delta+2)G_N(n\cdot m)=2-(2N+2)(n\cdot m)^{2N},\qquad
 (\Delta+2)(G(n\cdot m)-1)=-2\pi(\delta_m+\delta_{-m}).
\end{equation}
In geodesic polar coordinates about $m$, the measure
$(2N+2)(n\cdot m)^{2N}d\omega(n)$ tends weakly to
$2\pi(\delta_m+\delta_{-m})$. Since
$\int_{\Sph}G(n\cdot m)d\omega(n)$ is finite and independent of $m$,
Tonelli gives $F_\nu\in L^1(\Sph)$ and \eqref{g:green-kernel} gives
\begin{equation}\label{g:potential}
 F_\nu(n)=\int_\Sigma G(n\cdot\nu)\,d\sigma,\qquad
 (\Delta+2)F_\nu=2P-4\pi r,\qquad
 H=1+\frac{2h^2P}{\pi}-\frac{2h^2F_\nu}{\pi}.
\end{equation}
Here the two pole contributions combine because $r$ is even.
Applying $\Delta_{\Sph}+2$ to the expression for $H$ in \eqref{g:potential} gives \eqref{g:support-trace}.

Set $D=C-(\Hess H+H\mathrm{Id})$.
Its trace and divergence vanish, since
\begin{equation}\label{g:trace-divergence}
 \Div(\Hess H+H\mathrm{Id})=\nabla(\Delta H+2H),\qquad
 \Div C=8h^2\nabla r.
\end{equation}
In an isothermal chart $z=x+iy$ with metric
$e^{2\lambda}(dx^2+dy^2)$, write the covariant tensor as
\[
 D=a(dx^2-dy^2)+b(dx\otimes dy+dy\otimes dx).
\]
The conformal connection terms cancel for a trace-free tensor, so
$e^{2\lambda}\Div D=(a_x+b_y,b_x-a_y)$ as a covector. Hence
$2\partial_{\bar z}(a-ib)=(a_x+b_y)+i(a_y-b_x)=0$ in distributions.
Weyl's lemma, applied after $\Delta=4\partial_z\partial_{\bar z}$,
gives smooth coefficients satisfying the Cauchy--Riemann equations.
Thus $(a-ib)dz^2$ is holomorphic, and it vanishes on $\Sph$.
This proves \eqref{g:support-eq}.

Set $\Psi(x)=|x|H(x/|x|)$ for $x\ne0$. Writing $n=x/|x|$ and
$\Pi_n=\operatorname{Id}-n\otimes n$, the polar-coordinate Hessian is
\[
 D^2\Psi(x)[v,w]=|x|^{-1}\langle C(n)\Pi_n v,\Pi_n w\rangle
\]
in distributions, first by differentiation for smooth $H$ and then
by distributional approximation. Since $C\ge0$, the distributional criterion for convexity gives a
continuous, locally convex representative on $\R^3\setminus\{0\}$.
It is positively homogeneous of degree one, so $\Psi(0)=0$ gives a
continuous extension. Given $x,y\in\R^3$, choose $a_j\to0$ with
$0\notin[x+a_j,y+a_j]$. For $0\le t\le1$,
\[
 \Psi((1-t)x+ty+a_j)\le (1-t)\Psi(x+a_j)+t\Psi(y+a_j).
\]
Letting $j\to\infty$ gives convexity on $[x,y]$. Thus $\Psi$ is a
support function \cite[Theorem 1.7.1]{Schneider}.
Its symmetries centre $\mathcal K$ at the origin.
If $\mathcal K$ lay in a proper linear subspace, some $a\ne0$ would satisfy
$\Psi(x+ta)=\Psi(x)$ and hence
\begin{equation}\label{g:interior-kernel}
 C(n)\bigl[a-(a\cdot n)n\bigr]=0
\end{equation}
almost everywhere away from the two parallel normals.
This contradicts $\det C\geq1$ and proves nonempty interior.

Equations \eqref{g:gauss-second-moment} and \eqref{g:C} give
\begin{equation}\label{g:quadratic-pushforward}
 \nu_*(B^2d\sigma)=\frac{\cof C}{8}\,d\omega,\qquad
 \nu_*(|B|^2d\sigma)=\frac{\tr C}{8}\,d\omega,\qquad
 \Div(\cof C)=0.
\end{equation}
Since $\nabla_\Sigma(g\circ\nu)=B\nabla_{\Sph}g(\nu)$, these yield the
first equality in \eqref{g:mixed}.
For smooth convex support functions, the mixed area measure formula
\cite[Section 5.3]{Schneider} gives
\[
 6V(g,k,H)=\int_{\Sph}g\,\cof C:(\Hess_{\Sph} k+k\mathrm{Id})\,d\omega.
\]
Using $\Div(\cof C)=0$ from \eqref{g:quadratic-pushforward} and
$\tr(\cof C)=\tr C$, integration by parts gives
\begin{equation}\label{g:mixed-volume-integral}
 6V(g,k,H)=\int_{\Sph}
 [\tr C\,gk-\cof C(\nabla g,\nabla k)]\,d\omega .
\end{equation}
Trilinearity gives the formula for all smooth tests.
Use the standard approximation and Hausdorff continuity of support
functions and mixed volumes
\cite[Lemma 1.8.14 and Theorems 1.8.20 and 5.1.7]{Schneider}.
For rotational mollifications,
$H_\epsilon\to H$ uniformly and
$C_\epsilon=\Hess_{\Sph}H_\epsilon+H_\epsilon\mathrm{Id}\to C$
in $L^1(\Sph)$ by \eqref{g:support-eq}. For $g=h_A$, the same
mollification gives $g_\epsilon\to g$ uniformly and
$\nabla g_\epsilon\to\nabla g$ almost everywhere, with uniformly
bounded gradients. In particular,
\[
 \int_\Sigma |\nabla_\Sigma[(g_\epsilon-g)\circ\nu]|^2\,d\sigma
 \le \|B\|_\infty^2\int_{\Sph}|\nabla(g_\epsilon-g)|^2r\,d\omega
 \longrightarrow0 .
\]
Here \eqref{g:gauss-push} also ensures that $\nu(p)$ is a
differentiability point of $g$ for almost every $p$. The same limits
hold for $k$. The integral in \eqref{g:mixed} converges by the
$L^1$ convergence of $C_\epsilon$ and dominated convergence against
$|C|\,d\omega$; mixed volumes converge by the cited Hausdorff continuity.
\end{proof}

Let $\Pi_i x=x-(x\cdot e_i)e_i$, write $\mathcal K|e_i^\perp=\Pi_i\mathcal K$,
and put $\mathbb S_i^1=\Sph\cap e_i^\perp$. The planar area measure
$S_{\Pi_i\mathcal K}$ is the pushforward of the boundary arc length
by the outer normal, on $\mathbb S_i^1$.

\begin{lemma}\label{g:full-wall}
For every $i\in\{x,y,z\}$,
\begin{equation}\label{g:full-wall-eq}
 dS_{\Pi_i\mathcal K}=(\nu|_{W_i})_*(|B_{ii}|\,ds),\qquad
 W_i=\{\nu_i=0\}=\Sigma\cap(\{x_i=0\}\cup\{x_i=1\}),
\end{equation}
where $ds=d\mathcal H^1$ on $W_i\subset T$ and $x_i$ is taken modulo two.
\end{lemma}
\begin{proof}
On $W_i$, reflection symmetry and \eqref{g:hopf} give
\[
 e_i\in T_p\Sigma,\qquad
 \nabla_\Sigma\nu_i=B e_i=B_{ii}e_i,\qquad |B_{ii}|>0 .
\]
Since $|\nabla_\Sigma\nu_i|>0$ on $W_i$, $0$ is a regular value.
Integrate by parts on $\{\nu_i>0\}$ and $\{\nu_i<0\}$, using
$(\Delta_\Sigma+|B|^2)\nu_i=0$ and $d\sigma=dA/8$, to obtain
\begin{align}
 (\Delta_\Sigma+|B|^2)|\nu_i|&=2|B_{ii}|\,ds_{W_i},
 \label{g:abs-jacobi}\\
 Q(|\nu_i|,g\circ\nu)&=-\frac14\int_{W_i}g(\nu)|B_{ii}|\,ds.
 \label{g:absolute-jacobi-wall}
\end{align}
The distribution in \eqref{g:abs-jacobi} is paired using $dA$.
Let $\mathcal L_i=[-e_i,e_i]$ and
$G\subset e_i^\perp$ be compact convex with support $\varphi$ on
$\mathbb S_i^1$. The projection formula with a segment and the
mixed-area formula \cite[Sections 5.1 and 5.3]{Schneider} give
\begin{equation}\label{g:wall-mixed-coefficients}
 V_3(\mathcal L_i,G,\mathcal K)=\tfrac23V_2(G,\Pi_i\mathcal K),\qquad
 2V_2(G,\Pi_i\mathcal K)=\int_{\mathbb S_i^1}\varphi\,dS_{\Pi_i\mathcal K}.
\end{equation}
Since $|\nu_i|=h_{\mathcal L_i}\circ\nu$, \eqref{g:mixed} and
\eqref{g:absolute-jacobi-wall} imply
\begin{equation}\label{eq:body-005}
 \int_{W_i}\varphi(\nu)|B_{ii}|\,ds
 =3V_3(\mathcal L_i,G,\mathcal K)
 =\int_{\mathbb S_i^1}\varphi\,dS_{\Pi_i\mathcal K}.
\end{equation}
By linearity the identity holds for differences of planar supports,
which include $C^\infty(\mathbb S_i^1)$
\cite[Lemma 1.7.8]{Schneider}. Uniform approximation gives
\eqref{eq:body-005} for every $\varphi\in C(\mathbb S_i^1)$, proving
the equality of finite measures in \eqref{g:full-wall-eq}.
\end{proof}

\begin{lemma}\label{g:body-area}
The projection $\Pi_z\mathcal K$ satisfies
\begin{align}
 \operatorname{Area}(\Pi_z\mathcal K)&\ge\pi+4h^2,
 \label{g:body-lower}\\
 \operatorname{Area}(\Pi_z\mathcal K)
 &\le \pi(1+2h^2N_z/\pi)^2-(H(e_x)-H(e_y))^2\notag\\
 &\quad-\frac{\pi}{4-\pi}
       (H(e_x)+H(e_y)-2-4h^2N_z/\pi)^2 .
 \label{g:axis-area}
\end{align}
\end{lemma}
\begin{proof}
In the cube, let
$\gamma=\overline{\partial D_0\cap(0,1)^2}\times\{0\}$.
By \eqref{g:lower-face-strips} and \eqref{g:upper-edge-exclusion}, its
endpoints are $(x_0,1,0)$ and $(1,y_0,0)$ with $x_0,y_0\in(0,1)$.
For $x_0<x<1$, monotonicity gives
$D_0\cap(\{x\}\times[0,1])=\{x\}\times[0,y(x)]$ and
$\nu_y>0$ gives $y'(x)=-\nu_x/\nu_y<0$ on $\gamma$.
Orient $\gamma$ by increasing $x$. The endpoint normals are $e_x$
and $e_y$ by reflection. Along $\gamma$, write
$\nu=(\cos\theta,\sin\theta,0)$ and
$\tau=(\sin\theta,-\cos\theta,0)$. Then
$B(\tau,\tau)=\partial_s\nu\cdot\tau=-\partial_s\theta$. For $\varphi$ on the first quadrant of
$\mathbb S_z^1$, and $\Phi'(\theta)=\varphi(\cos\theta,\sin\theta,0)$,
\[
 \int_\gamma\varphi(\nu)B(\tau,\tau)\,ds=-\Phi(\pi/2)+\Phi(0).
\]
Reflecting in $x=0,y=0$ extends this identity to $W_z\subset T$:
\begin{equation}\label{g:fold-turn-measure}
 \nu_*(B(\tau,\tau)\,ds_{W_z})=-d\theta
 \qquad\text{on }\mathbb S_z^1 .
\end{equation}
Here $B_{zz}=h-B(\tau,\tau)>0$ by \eqref{g:hopf}. By \eqref{g:full-wall-eq},
$dS_{\Pi_z\mathcal K}=d\theta+h\,\nu_*ds_{W_z}$.
The planar Minkowski theorem and additivity of the area measure
\cite[Theorem 8.3.1 and Corollary 8.3.2]{Schneider} give
\begin{align}
 \Pi_z\mathcal K&=\mathbb D+L_z,\label{g:projection-minkowski}\\
 dS_{L_z}&=h\,\nu_*ds_{W_z},\label{g:projection-measure}\\
 \Per(L_z)&=4h^2N_z,\label{g:projection-summand}\\
 \frac{\Per(\Pi_z\mathcal K)}{2\pi}&=1+2h^2N_z/\pi .
 \label{g:projection-mean-support}
\end{align}
Here $\mathbb D$ is the unit disk, and reflection symmetry fixes
the translation of $L_z$. Its area measure has total mass
$h\mathcal H^1(W_z)=4h L_{z,0}=4h^2N_z$ by
\eqref{g:stress-length}. Adding the circumference $2\pi$ gives
\eqref{g:projection-mean-support}.

In coordinates $(u,v)=(1-x,1-y)$, orient the arc from the
$v$-axis to the $u$-axis and write its unit tangent
$(u',v')=(\cos\phi,-\sin\phi)$, $0\le\phi\le\pi/2$.
If $[0,L]$ is its arc-length interval, the area of the complement
of $D_0$ in the square is
\begin{equation}\label{g:hole-area-order}
 1-N_z=\int_0^L v(s)u'(s)\,ds
      =\int_{s<t}\cos\phi(s)\sin\phi(t)\,ds\,dt.
\end{equation}
For $0<s<L$, let
\[
 \phi^\uparrow(s)=\inf\{\vartheta\in[0,\pi/2]:
               |\{t\in[0,L]:\phi(t)\le\vartheta\}|\ge s\}.
\]
The arc
$\gamma^\uparrow(s)=(0,v(0))+\int_0^s
(\cos\phi^\uparrow(t),-\sin\phi^\uparrow(t))\,dt$
has the original endpoints because $(\phi^\uparrow)_*ds=\phi_*ds$.
Its slopes $-\tan\phi^\uparrow$ are nonincreasing, so the region
bounded by $\gamma^\uparrow$ and the axes is convex. For $s<t$,
\[
 \cos\phi^\uparrow(s)\sin\phi^\uparrow(t)
 -\cos\phi^\uparrow(t)\sin\phi^\uparrow(s)
 =\sin(\phi^\uparrow(t)-\phi^\uparrow(s))\ge0 .
\]
Thus by \eqref{g:hole-area-order} and equimeasurability the area of
this region is
\begin{equation}\label{g:sorted-hole-area}
 \frac12\int_0^L\int_0^L
  \max\{\cos\phi(s)\sin\phi(t),\cos\phi(t)\sin\phi(s)\}\,ds\,dt
 \ \ge 1-N_z .
\end{equation}
Reflect the region in both coordinate axes. Its area measure on the
first quadrant is the pushforward of $ds$ by
$s\mapsto(\sin\phi^\uparrow(s),\cos\phi^\uparrow(s))$.
Since $\phi^\uparrow$ and $\phi$ have the same distribution, the full
area measure is $\nu_*ds_{W_z}=h^{-1}dS_{L_z}$.
Planar Minkowski uniqueness identifies the reflected region with $L_z/h$.
Steiner's formula and \eqref{g:projection-summand} therefore give
\begin{equation}\label{g:projection-steiner}
 \operatorname{Area}(\Pi_z\mathcal K)
 =\pi+\Per(L_z)+\operatorname{Area}(L_z)
 \ge\pi+4h^2N_z+4h^2(1-N_z)=\pi+4h^2 .
\end{equation}

By \eqref{g:projection-mean-support} and coordinate reflection symmetry,
the Fourier expansion of the circle support is
\[
 H(\cos\theta\,e_x+\sin\theta\,e_y)
 =1+2h^2N_z/\pi+\sum_{j\ge1}a_j\cos(2j\theta).
\]
The planar support-function area formula and Parseval give
\begin{equation}\label{g:fourier-area}
 \operatorname{Area}(\Pi_z\mathcal K)
 =\pi(1+2h^2N_z/\pi)^2-\frac{\pi}{2}\sum_{j\ge1}(4j^2-1)a_j^2 .
\end{equation}
The area formula extends by uniform and $H^1$ approximation.
The sums of $(4j^2-1)^{-1}$ over odd and even $j$ are
$\pi/8$ and $(4-\pi)/8$, respectively. By Cauchy--Schwarz and
\eqref{g:fourier-area}, $\sum_{j\ge1}|a_j|<\infty$, so the
Fourier series gives the support values at $\theta=0,\pi/2$:
\[
 H(e_x)-H(e_y)=2\sum_{j\ \mathrm{odd}}a_j,\qquad
 H(e_x)+H(e_y)-2-4h^2N_z/\pi=2\sum_{j\ \mathrm{even}}a_j .
\]
Applying Cauchy--Schwarz to these sums with weights $4j^2-1$ gives
\begin{align}
 \sum_{j\ \mathrm{odd}}(4j^2-1)a_j^2
 &\ge \tfrac2\pi(H(e_x)-H(e_y))^2,\notag\\
 \sum_{j\ \mathrm{even}}(4j^2-1)a_j^2
 &\ge \tfrac2{4-\pi}(H(e_x)+H(e_y)-2-4h^2N_z/\pi)^2 .
 \label{g:axis-fourier-estimates}
\end{align}
Substituting these two terms into \eqref{g:fourier-area} proves
\eqref{g:axis-area}.
\end{proof}

The Green representation in \eqref{g:potential} holds almost
everywhere. We verify its pointwise form, including the coordinate axes.
\begin{lemma}\label{g:green}
At every $n\in\Sph$,
\begin{equation}\label{g:green-eq}
 H(n)=1+\frac{2h^2P}\pi
                -\frac{2h^2}\pi\int_\Sigma
                      (n\cdot\nu)\operatorname{atanh}(n\cdot\nu)d\sigma.
\end{equation}
In particular each axis integral is finite.
\end{lemma}
\begin{proof}
Let $F_\nu$ be the integral in \eqref{g:potential} and define independently
$F_0=P+\pi(1-H)/(2h^2)$. Equation \eqref{g:potential} gives $F_\nu=F_0$
almost everywhere, and $F_0$ is continuous.
The positive partial sums
$F_N(n)=\int_\Sigma\sum_{j=1}^N(n\cdot\nu)^{2j}/(2j-1)\,d\sigma$
increase to the lower semicontinuous $F_\nu$. For every $n$, choose
$n_j\to n$ with $F_\nu(n_j)=F_0(n_j)$. Then
$F_\nu(n)\le\liminf_j F_0(n_j)=F_0(n)$.
By \eqref{g:green-kernel},
\begin{equation}\label{g:potential-laplacian}
 \Delta F_N=2P-2F_N-(2N+2)\int_\Sigma(n\cdot\nu)^{2N}\,d\sigma\leq2P.
\end{equation}
For $0<r<\pi$, set
$M_N(r)=(2\pi\sin r)^{-1}\int_{\partial B_r(n)}F_N\,ds$.
Differentiating in polar coordinates and using \eqref{g:potential-laplacian},
\[
 M_N'(r)=\frac{\int_{B_r(n)}\Delta F_N\,d\omega}{2\pi\sin r}
 \le 2P\frac{1-\cos r}{\sin r}=2P\tan(r/2).
\]
Thus $M_N(r)\le F_N(n)-4P\log\cos(r/2)$.
Average over $0<r<\epsilon$ with weight $2\pi\sin r$ and use
$-\log\cos(r/2)\le-\log\cos(\epsilon/2)$:
\begin{equation}\label{g:potential-ball-average}
 F_N(n)\geq\frac1{2\pi(1-\cos\epsilon)}
 \int_{B_\epsilon(n)}F_N\,d\omega+4P\log\cos(\epsilon/2).
\end{equation}
By monotone convergence in \eqref{g:potential-ball-average},
\[
 F_\nu(n)\ge\frac{\int_{B_\epsilon(n)}F_0\,d\omega}{2\pi(1-\cos\epsilon)}
              +4P\log\cos(\epsilon/2)\ \longrightarrow\ F_0(n)
 \qquad(\epsilon\downarrow0).
\]
Hence $F_\nu=F_0$ on $\Sph$.
\end{proof}

On $\gamma_{x,1}$ in the face $x=1$, choose $\theta$ by
$(\nu_y,\nu_z)=(\cos\theta,\sin\theta)$. Orient the arc from $z=0$
towards $y=0$. Its endpoint angles are $0$ and $\pi/2$, and
$d\theta/ds=\kappa_{x,1}=h-B_{xx}>h$ by \eqref{g:hopf}.
For this section put
\begin{equation}\label{g:wall-variables}
 \lambda=h/\kappa\in(0,1),\quad
 \ell=\int_0^{\pi/2}\lambda\,d\theta,\quad
 B=\int_0^{\pi/2}\cos\theta\,\lambda\,d\theta,\quad
 R=\int_0^{\pi/2}\sin\theta\,\lambda\,d\theta,\quad J=\ell-A(\lambda),
\end{equation}
where
$A(\lambda)=\int_{\theta<\varphi}\cos\theta\sin\varphi
\lambda(\theta)\lambda(\varphi)\,d\theta d\varphi$.
From $\gamma'_{x,1}(s)=(0,-\sin\theta,\cos\theta)$,
\begin{equation}\label{g:upper-wall-coordinates}
 y(\theta)=\frac1h\int_\theta^{\pi/2}\sin\varphi\,\lambda(\varphi)\,d\varphi,
 \qquad
 z(\theta)=\frac1h\int_0^\theta\cos\varphi\,\lambda(\varphi)\,d\varphi,
 \qquad ds=\frac{\lambda}{h}\,d\theta.
\end{equation}
The endpoints of $\gamma_{x,1}(\theta)=(1,y(\theta),z(\theta))$
and the integral $\int_{\gamma_{x,1}}-z\,dy$ give
\begin{equation}\label{g:section-endpoints}
 \gamma_{x,1}(0)=(1,R_x/h,0),\quad
 \gamma_{x,1}(\pi/2)=(1,0,B_x/h),\quad
 \ell_x=hL_{x,1},\quad A(\lambda_x)=h^2A_{x,1}.
\end{equation}
The intersection at $y=1$ gives $(R_y,B_y,\ell_y)$ by interchanging
$x,y$. In \eqref{g:wall-variables} the symbols $(B,R,\ell,J)$ denote
the corresponding four quantities for a single section.

\begin{lemma}\label{g:pressure}
For the quantities in \eqref{g:wall-variables},
\begin{equation}\label{g:pressure-eq}
 J\geq F(R)\geq\pi R/4,\qquad
 F(t)=\frac{\arcsin t+t\sqrt{1-t^2}}2.
\end{equation}
\end{lemma}
\begin{proof}
Fix $R$ and put $r=\arccos R$,
$\lambda_*=\mathbf1_{[r,\pi/2]}$, and $f=\lambda-\lambda_*$.
The symmetric area kernel is
$k(\theta,\varphi)=\cos(\min\{\theta,\varphi\})
\sin(\max\{\theta,\varphi\})$.
For $0<\varphi<\pi/2$, $\partial_\theta k=0$ at $0,\pi/2$ and
$[\partial_\theta k]_{\theta=\varphi^-}^{\theta=\varphi^+}=1$.
For signed integrable $f$, set $g_f(\theta)=\int_0^{\pi/2}k(\theta,\varphi)
f(\varphi)\,d\varphi$. Then
\begin{equation}\label{g:wall-kernel-equation}
 g_f''+g_f=f,\qquad g_f'(0)=g_f'(\pi/2)=0,\qquad
 A(f)=\frac12\int_0^{\pi/2}f g_f\,d\theta.
\end{equation}
The fixed-$R$ constraint is $\int_0^{\pi/2}\sin\theta\, f(\theta)d\theta=0$,
so $g_f(0)=0$. By \eqref{g:wall-kernel-equation},
\begin{equation}\label{g:wall-area-variation}
 2A(f)=[g_f g_f']_0^{\pi/2}
       +\int_0^{\pi/2}(g_f^2-(g_f')^2)\,d\theta\le0.
\end{equation}
The endpoint term vanishes by $g_f(0)=g_f'(\pi/2)=0$.
The integral is nonpositive by the Dirichlet--Neumann Poincar\'e
inequality on $[0,\pi/2]$, whose first eigenfunction is $\sin\theta$
and first eigenvalue is $1$.
For $\lambda_*$, $g_{\lambda_*}(\theta)=R\cos\theta$ on $[0,r]$
and $g_{\lambda_*}(\theta)=1-\sin r\sin\theta$ on $[r,\pi/2]$.
Consequently
\[
 J(\lambda_*)=\int_r^{\pi/2}(1-g_{\lambda_*}(\theta)/2)\,d\theta
 =(\pi/2-r+\sin r\cos r)/2=F(R).
\]
Symmetry of the kernel and the fixed-$R$ constraint give
\[
 J(\lambda)-J(\lambda_*)
 =\int_0^{\pi/2}[1-g_{\lambda_*}(\theta)-\sin r\sin\theta]
                         f(\theta)\,d\theta-A(f).
\]
The bracket is $1-\cos(\theta-r)$ on $[0,r)$ and zero on
$[r,\pi/2]$. Since $f=\lambda$ on $[0,r)$,
\begin{equation}\label{g:wall-quadratic-expansion}
 J-F(R)=\int_0^r[1-\cos(\theta-r)]\lambda(\theta)\,d\theta-A(f)\geq0.
\end{equation}
By concavity of $F$ on $[0,1]$,
$F(R)\ge(1-R)F(0)+RF(1)=\pi R/4$.
\end{proof}

Put $H_i=H(e_i)$ and
$L_i=\E[|\nu_i|\operatorname{atanh}|\nu_i|]$ for the probability measure
$d\sigma/P$. By \eqref{g:genus-two-axis-height}, the section at $x=0$
has the graph $z=w(0,y)$ for $0\le y\le1$. Orient its free arc
$\gamma_{x,0}$ from $y=0$ to $y=1$ and parametrize by arc length $s$:
\begin{equation}\label{g:lower-section-angle}
 (\nu_y,\nu_z)=(\sin\vartheta,\cos\vartheta),\qquad
 \gamma'_{x,0}=(0,\cos\vartheta,-\sin\vartheta),\qquad
 \vartheta(0)=\vartheta(L_{x,0})=0,\quad \kappa_{x,0}=\vartheta'.
\end{equation}
The endpoint heights are $w(0,0)$ and $B_y/h=w(0,1)$.
Using \eqref{g:hopf}, \eqref{g:lower-section-angle} and
\eqref{g:section-endpoints}, we obtain
\begin{align}
 \int_{\gamma_{x,0}}\nu_z|B_{xx}|\,ds
 &=h[y]_0^{L_{x,0}}-[\sin\vartheta]_0^{L_{x,0}}=h,
 \label{g:lower-wall-flux}\\
 \int_{\gamma_{x,1}}\nu_z|B_{xx}|\,ds
 &=\int_0^{\pi/2}\sin\theta\,(1-\lambda_x)\,d\theta=1-R_x,
 \label{g:upper-wall-flux}\\
 \int_{\gamma_{x,0}}\nu_y|B_{xx}|\,ds
 &=-h[z]_0^{L_{x,0}}+[\cos\vartheta]_0^{L_{x,0}}=hw(0,0)-B_y,
 \label{g:lower-horizontal-flux}\\
 \int_{\gamma_{x,1}}\nu_y|B_{xx}|\,ds
 &=\int_0^{\pi/2}\cos\theta\,(1-\lambda_x)\,d\theta=1-B_x.
 \label{g:upper-horizontal-flux}
\end{align}
With $\psi_j=h_{[-e_j,e_j]}\circ\nu=|\nu_j|$, reflection in
$y=0,z=0$ and \eqref{g:absolute-jacobi-wall} give
\[
 Q(\psi_x,\psi_j)=-\frac14\int_{W_x}|\nu_j||B_{xx}|\,ds
 =-\sum_{c=0,1}\int_{\gamma_{x,c}}\nu_j|B_{xx}|\,ds,\qquad j=y,z.
\]
For $\mathcal L_i=[-e_i,e_i]$, the mixed-volume projection formula gives
$6V(\mathcal L_x,\mathcal L_z,\mathcal K)=4[H(e_y)+H(-e_y)]=8H_y$.
Using \eqref{g:mixed}, \eqref{g:lower-wall-flux} and
\eqref{g:upper-wall-flux} gives
\begin{equation}\label{g:perpendicular-segment-flux}
 -H_y=Q(\psi_x,\psi_z)=-h-1+R_x.
\end{equation}
With $\mathcal L_y$ in place of $\mathcal L_z$, \eqref{g:lower-horizontal-flux} and
\eqref{g:upper-horizontal-flux} give
\begin{equation}\label{g:horizontal-segment-flux}
 -H_z=Q(\psi_x,\psi_y)=-1-hw(0,0)+B_x+B_y.
\end{equation}
Interchanging $x,y$ in \eqref{g:perpendicular-segment-flux} and
evaluating \eqref{g:green-eq} at $e_x,e_y$ yields
\begin{equation}\label{g:source}
 \begin{aligned}
 R_x&=1+h-H_y=h+\tfrac{2h^2P}{\pi}(L_y-1),\\
 R_y&=1+h-H_x=h+\tfrac{2h^2P}{\pi}(L_x-1).
 \end{aligned}
\end{equation}

For $v=b$ and $P\le1$, \eqref{g:stress-eq}, \eqref{g:wall-variables}
and \eqref{g:hole} give $J_x+J_y=2hP-3h^2/\pi$.
By \eqref{g:pressure-eq}, $R_x+R_y\le4(J_x+J_y)/\pi$.
Substitution in \eqref{g:source} gives
\begin{equation}\label{g:axis-sum-bound}
 H_x+H_y\ge2+2h-\frac{8hP}{\pi}+\frac{12h^2}{\pi^2}.
\end{equation}
In the right side of \eqref{g:axis-area}, replace $N_z$ by $n$
with $N_z\le n\le N_z^U\le1$ and hold $H_x,H_y$ fixed. The derivative
with respect to $n$ is
\begin{equation}\label{g:rho-monotonicity}
 \frac{4h^2}{4-\pi}\bigl[2(H_x+H_y)-\pi(1+2h^2n/\pi)\bigr]>0.
\end{equation}
In fact \eqref{g:axis-sum-bound} bounds the bracket below by
$(4-\pi)+(4-16/\pi)h+(24/\pi^2-2)h^2$. Its leading coefficient
is positive and its discriminant is
$-8(4-\pi)(4+2\pi-\pi^2)/\pi^2<0$.
Thus \eqref{g:axis-area} remains valid with $N_z^U$ in place of $N_z$.
Dropping both squares in that upper bound and using
\eqref{g:body-lower} gives $\pi+4h^2\le\pi(1+2h^2N_z^U/\pi)^2$.
Define
\begin{equation}\label{g:body-gates}
 \rho_U=1+2h^2N_z^U/\pi,\qquad e=\rho_U^2-1-4h^2/\pi\ge0,\qquad
 r_b=\pos{1+h-\rho_U-\sqrt e}.
\end{equation}
The squared terms of \eqref{g:axis-area} at $N_z^U$ satisfy
\[
 (H_x-H_y)^2+\frac{\pi}{4-\pi}(H_x+H_y-2\rho_U)^2
 =\pi(H_x-\rho_U)^2+
 \frac4{4-\pi}\left[H_y-\rho_U+
                   (\pi/2-1)(H_x-\rho_U)\right]^2 .
\]
Together with \eqref{g:body-lower} this gives
$(H_x-\rho_U)^2\le e$, and similarly $(H_y-\rho_U)^2\le e$.
By \eqref{g:source},
$R_i=1+h-H(e_j)\ge1+h-\rho_U-\sqrt e$ for $(i,j)=(x,y),(y,x)$.
Since $0\le R_i\le1$, this gives $r_b\le R_i\le1$. Using also
\eqref{g:stress-eq} and \eqref{g:pressure-eq},
\begin{equation}\label{g:active-axis-gates}
 r_b\le1,\qquad R_x,R_y\ge r_b,\qquad
 m_i\le1-\frac{h}{\pi P}-\frac{F(r_b)}{hP}\quad(i=x,y).
\end{equation}
Indeed $m_i=1-h/(\pi P)-J_i/(hP)$ and
$J_i\ge F(R_i)\ge F(r_b)$ for $i=x,y$.
For $i=z$, \eqref{g:hole} gives $m_z=1-h/(\pi P)$.

\subsection{Curvature-distribution constraints and angular estimates}\label{g:probability}\label{g:angular-section}

Fix $i$, and take expectation with respect to $d\sigma/P$ on
$\Sigma$. The law of $X=|\nu_i|$ is $\varrho(t)dt$ by
\eqref{g:gauss-push} and sphere coarea. For an integrable $Y$,
$\varrho(t)\E[Y\mid X=t]$ is the density of $X_*(Yd\sigma/P)$, so
\[
 \E[\psi(X)Y]=\int_0^1\psi(t)\E[Y\mid X=t]\,\varrho(t)\,dt .
\]
Put $\alpha=\E X$, $\mu=\E X^2$, $a(t)=\int_t^1s\varrho(s)ds$,
$z=1-t^2$, and $c=\pi/(2h^2P)$. In
\eqref{g:weak-gauss-divergence}, test with
$\psi(|n_i|)\operatorname{sign}(n_i)(e_i-n_in)$ for
$\psi\in C_c^\infty(0,1)$. Its derivative is
$\psi'(|n_i|)(e_i-n_in)^{\otimes2}-|n_i|\psi(|n_i|)\mathrm{Id}$.
Writing $e=(e_i-\nu_i\nu)/\sqrt{1-\nu_i^2}$ and
$\beta=P\varrho\E[B(e,e)\mid X=t]$, this gives
\[
 \int_0^1(z\beta\psi'-hPt\varrho\psi)\,dt=0,\qquad
 (z\beta)'=-hPt\varrho=hPa'.
\]
Thus $z\beta=hPa+C$. Since $a(t)\to0$ as $t\uparrow1$ and
$\beta\in L^1(0,1)$, $C=0$: otherwise $|\beta(t)|\ge |C|/(2z)$
near $1$. By \eqref{g:gauss-push} and $|n_i|_*d\omega=4\pi\,dt$,
$X_*(Kd\sigma)=-(\pi/2)dt$. Taking densities in
$|Be|^2=hB(e,e)-K$ and using $z\beta=hPa$ gives
\begin{equation}\label{g:conditional-curvature}
 \varrho\E[B(e,e)/h\mid X=t]=a/z,\qquad
 \varrho\E[|Be|^2/h^2\mid X=t]=a/z+c.
\end{equation}
Conditional Cauchy--Schwarz and $a'=-t\varrho$ give
\begin{equation}\label{g:cap}
 e_t:=\varrho(a/z+c)-(a/z)^2\geq0,\qquad
 \left(\frac{a^2}{1-t^2}+2ca\right)'=-2te_t\leq0,
\end{equation}
and hence the upper bound
\begin{equation}\label{g:tail-envelope}
 a(t)\leq U(\alpha,c;t)
       :=\sqrt{c^2z^2+\alpha(\alpha+2c)z}-cz.
\end{equation}
The density identities hold almost everywhere. The formula
$a(t)=\int_t^1s\varrho(s)ds$ gives an absolutely continuous
nonincreasing representative on $[0,1]$. Integration of
\eqref{g:cap} gives $a(t)^2/z+2ca(t)\le\alpha^2+2c\alpha$ for
$0<t<1$. Solving for $a(t)\ge0$ gives \eqref{g:tail-envelope},
including the endpoints by continuity.

\revision{We now retain the tail constraint \eqref{g:tail-envelope} and optimize
over all probability measures with the prescribed second moment. The next
lemma identifies a common optimizer for the first moment and costs with
nondecreasing tail kernels.}
For a probability measure $\lambda$ on $[0,1]$, put
$\alpha_\lambda=\int x\,d\lambda$ and
$a_\lambda(t)=\int_{(t,1]}x\,d\lambda(x)$. For $c>0$ define
\begin{equation}\label{g:tail-feasible-laws}
 \mathcal P_{\mu,c}=
 \{\lambda\in\mathcal P([0,1]):\ \int x^2\,d\lambda=\mu,\
              a_\lambda(t)\le U(\alpha_\lambda,c;t)\ (0\le t\le1)\}.
\end{equation}
For fixed $\kappa$ and a nondecreasing kernel $w$ bounded below, set
$\mathcal C_w(\lambda)=\kappa+\int_0^1w(t)a_\lambda(t)\,dt$,
using $w_M=\min(w,M)\uparrow w$ at positive singularities.

\begin{lemma}\label{g:canonical}
The set in \eqref{g:tail-feasible-laws} is nonempty exactly for
$0\le\mu\le\mu_{\rm crit}(c)$, where $\mu_{\rm crit}$ is defined in
the proof. For each such pair $(\mu,c)$ a single $\lambda_*$ satisfies
\begin{align}
 A_{\rm can}(\mu,c)&:=\max_{\lambda\in\mathcal P_{\mu,c}}\alpha_\lambda
 =\alpha_{\lambda_*},\label{g:canonical-mean}\\
 C_w(\mu,c)&:=\min_{\lambda\in\mathcal P_{\mu,c}}\mathcal C_w(\lambda)
 =\mathcal C_w(\lambda_*)\qquad\text{for every such }w.
 \label{g:canonical-cost}
\end{align}
For $\mu=0$, $\lambda_*=\delta_0$. For $\mu>0$, its tail is
$U(\alpha_{\lambda_*},c;t)\mathbf 1_{[0,R)}(t)$, with $R\in(0,1]$
given by $M(\alpha_{\lambda_*},c,R)=1$ in
\eqref{g:saturated-mass-moment}. On the feasible domain,
$A_{\rm can}$ is jointly concave and $C_w$ jointly convex where finite.
\end{lemma}
\begin{proof}
Tonelli gives the exact identities
\begin{equation}\label{g:mass}
 \int_0^1a(t)dt=\mu,\qquad
 \alpha+\int_0^1\frac{\alpha-a(t)}{t^2}dt
                       =\Pr(X>0)\leq1.
\end{equation}
For $\alpha,c>0$ and $0<t<1$, differentiating
$U^2+2c(1-t^2)U=\alpha(\alpha+2c)(1-t^2)$ gives $U_t<0$,
$0<U_\alpha\leq1$, and $U_c\geq0$.
Moreover $\alpha\sqrt z\leq U\leq(\alpha+c)\sqrt z$,
$\alpha-U=O(t^2)$, $1-U_\alpha=O(t^2)$, and $U_c=O(t^2)$
near zero, locally uniformly in positive parameters.

For $\alpha>0$ and $0<R\le1$, define a finite measure on $(0,1]$
by
\[
 d\lambda_R(t)=\frac{-\partial_tU(\alpha,c;t)}{t}\mathbf1_{(0,R)}(t)\,dt
              +\frac{U(\alpha,c;R)}{R}\,d\delta_R(t).
\]
Direct integration gives $a_{\lambda_R}(t)=a_R(t):=U(\alpha,c;t)
\mathbf1_{[0,R)}(t)$ and $\int t\,d\lambda_R(t)=\alpha$. Its mass
and second moment are
\begin{equation}\label{g:saturated-mass-moment}
 M(\alpha,c,R)=\frac\alpha R+
             \int_0^R\frac{\alpha-U}{t^2}dt,\qquad
 \mu_R=\int_0^RUdt.
\end{equation}
When $M\le1$, $\lambda_R+(1-M)\delta_0$ is a probability law
with the same tail. For any $\lambda\in\mathcal P_{\mu,c}$ with
$\mu>0$, choose $R$ with $\int_0^R U(\alpha_\lambda,c;t)dt=\mu$,
which is possible by $a_\lambda\le U$ and \eqref{g:mass}.
For $\Delta=a_\lambda-a_R$, $\Delta\le0$ on $(0,R)$,
$\Delta\ge0$ on $(R,1)$, and $\int_0^1\Delta=0$. Thus
$[w_M(t)-w_M(R)]\Delta(t)\ge0$ and
$(t^{-2}-R^{-2})(-\Delta(t))\ge0$. Integrating gives
\begin{align}
 \int_0^1w_M(t)\Delta(t)\,dt
 &=\int_0^1[w_M(t)-w_M(R)]\Delta(t)\,dt\ge0,
 \label{g:tail-crossing}\\
 \int_0^1\frac{a_R-a_\lambda}{t^2}\,dt
 &=\int_0^1(t^{-2}-R^{-2})(a_R-a_\lambda)\,dt\ge0 .
 \label{g:tail-mass-comparison}
\end{align}
The integrals in \eqref{g:tail-mass-comparison} are finite
near zero since $0\le U-a_\lambda\le\alpha-a_\lambda$ there
and \eqref{g:mass} holds. By \eqref{g:mass} and
\eqref{g:tail-mass-comparison},
\[
 \lambda((0,1])-M=\int_0^1(a_R-a_\lambda)t^{-2}\,dt\ge0 .
\]

For fixed $c>0$, \eqref{g:saturated-mass-moment} gives
\[
 M_R=-U(\alpha,c;R)/R^2<0\quad(0<R<1),\qquad
 M_\alpha=\frac1R+\int_0^R\frac{1-U_\alpha}{t^2}\,dt\ge\frac1R.
\]
Also $M\ge\alpha/R\to\infty$ as $R\downarrow0$.
Since $U\ge\alpha\sqrt{1-t^2}$,
$M(\alpha,c,1)\le\pi\alpha/2$, whereas $M(1,c,1)>1$
because $U(1,c;t)<1$ for $0<t<1$.
Continuity and $M_\alpha>0$ determine a unique
$\alpha_{\rm crit}(c)$ by $M(\alpha_{\rm crit}(c),c,1)=1$.
For $0<\alpha\le\alpha_{\rm crit}(c)$ there is a unique
$R_c(\alpha)$ with $M(\alpha,c,R_c(\alpha))=1$. It lies in
$(0,1)$ for $\alpha<\alpha_{\rm crit}(c)$, and equals $1$
at $\alpha=\alpha_{\rm crit}(c)$ by strict decrease in $R$.
Continuity and strict monotonicity of $M$ imply that $R_c$
is continuous and strictly increasing. Hence
\[
 \mu_c(\alpha)=\int_0^{R_c(\alpha)}U(\alpha,c;t)\,dt
\]
is continuous and strictly increasing, tends to zero as
$\alpha\downarrow0$ since $\mu_c(\alpha)\le\alpha$, and at
$\alpha_{\rm crit}(c)$ equals $\mu_{\rm crit}(c)$.

For $\lambda\in\mathcal P_{\mu,c}$ and its cutoff $R$ as above,
\begin{equation}\label{g:canonical-feasibility}
 M(\alpha_\lambda,c,1)\le M(\alpha_\lambda,c,R)\le1,\qquad
 \alpha_\lambda\le\alpha_{\rm crit}(c),\qquad R\ge R_c(\alpha_\lambda).
\end{equation}
Thus $\mu_c(\alpha_\lambda)\le\mu\le
\int_0^1 U(\alpha_{\rm crit}(c),c;t)\,dt=\mu_{\rm crit}(c)$.
For $0<\mu\le\mu_{\rm crit}(c)$, take $\alpha_*=\mu_c^{-1}(\mu)$
and $R_*=R_c(\alpha_*)$. The measure $\lambda_{R_*}$ constructed
with $\alpha_*$ has mass $1$, moment $\mu$, and mean $\alpha_*$,
so it gives \eqref{g:canonical-mean}. For $\mu=0$ the only
probability measure is $\delta_0$.

If $\alpha_1\le\alpha_2$ and $\int_0^{R_j}U(\alpha_j,c;t)\,dt=\mu$,
then $U_\alpha\ge0$ gives $R_2\le R_1$. For
$a_j(t)=U(\alpha_j,c;t)\mathbf1_{[0,R_j)}(t)$,
\[
 (w_M(t)-w_M(R_2))(a_1(t)-a_2(t))\ge0,\qquad
 \int_0^1(a_1-a_2)\,dt=0.
\]
Integrating and taking $M\to\infty$ gives $\int w a_1\ge\int w a_2$.
Use $\alpha_1=\alpha_\lambda$, $\alpha_2=\alpha_*$ and
\eqref{g:tail-crossing} for $\lambda$ to conclude
\[
 \mathcal C_w(\lambda_{R_*})
 \le\mathcal C_w(\lambda_R+(1-M(\alpha_\lambda,c,R))\delta_0)
 \le\mathcal C_w(\lambda).
\]
This proves \eqref{g:canonical-cost}, including $R_*=1$ at
$\mu=\mu_{\rm crit}(c)$. In particular $w(t)=1/(1-t^2)$ is
integrable against $U\,dt$ at $R=1$, since
$U\le(\alpha+c)\sqrt{1-t^2}$.

Differentiation gives
\begin{equation}\label{g:envelope-Hessian}
 \Hess_{\alpha,c}U
 =-\frac{t^2\sqrt z}{(\alpha^2+2\alpha c+c^2z)^{3/2}}
   \begin{pmatrix}c^2&-\alpha c\\-\alpha c&\alpha^2\end{pmatrix}
 \leq0.
\end{equation}
For $\theta\in[0,1]$ and $\lambda_j\in\mathcal P_{\mu_j,c_j}$,
set $\lambda=\theta\lambda_1+(1-\theta)\lambda_2$ and
$(\mu,c)=\theta(\mu_1,c_1)+(1-\theta)(\mu_2,c_2)$. Then
$\int x^2d\lambda=\mu$ and by \eqref{g:envelope-Hessian},
\[
 a_\lambda(t)\le\theta U(\alpha_{\lambda_1},c_1;t)
 +(1-\theta)U(\alpha_{\lambda_2},c_2;t)\le U(\alpha_\lambda,c;t).
\]
Taking $\lambda_j$ to attain $A_{\rm can}$ gives
$A_{\rm can}(\mu,c)\ge\theta A_{\rm can}(\mu_1,c_1)
+(1-\theta)A_{\rm can}(\mu_2,c_2)$.
Taking $\lambda_j$ to attain $C_w$ gives
$C_w(\mu,c)\le\theta C_w(\mu_1,c_1)+(1-\theta)C_w(\mu_2,c_2)$.
\end{proof}

For a smooth cost $\phi$,
\begin{equation}\label{g:cost}
 \E\phi(X)=\phi(0)+\alpha\phi'(0)+\int_0^1
 \frac{t\phi'(t)-\phi(t)+\phi(0)}{t^2}a(t)dt.
\end{equation}
For $x>0$,
\[
 \frac{\phi(x)-\phi(0)}x
 =\phi'(0)+\int_0^x\frac{t\phi'(t)-\phi(t)+\phi(0)}{t^2}\,dt .
\]
Write $w_\phi$ for the integrand factor multiplying $a(t)$
in \eqref{g:cost}. For $\phi\in C^2([0,1])$,
$w_\phi(t)=t^{-2}\int_0^t s\phi''(s)\,ds$ is bounded.
Thus Fubini gives
\[
 \E\!\left[X\int_0^X w_\phi(t)\,dt\right]
 =\int_0^1 w_\phi(t)\E[X\mathbf 1_{\{X>t\}}]\,dt
 =\int_0^1 w_\phi(t)a(t)\,dt.
\]
For nonnegative singular $w_\phi$, the same identity for
$\min(w_\phi,M)$ gives \eqref{g:cost} by monotone convergence.
When $\phi'(0)=0$ and $w_\phi$ is nondecreasing and bounded below,
\eqref{g:canonical-cost} gives
$\min_{\lambda\in\mathcal P_{\mu,c}}\int\phi\,d\lambda
=\int\phi\,d\lambda_*$ with $\kappa=\phi(0)$.
In the three cases used here, $\phi'(0)=0$. The costs are
$X\operatorname{atanh}X$, $1-\sqrt{1-X^2}$, and
\[
 \phi_8(X)=2\sum_{n=1}^8\frac{((1-X^2)/2)^n}{2n-1}.
\]
The first two kernels are
$(1-t^2)^{-1}$ and $[\sqrt{1-t^2}(1+\sqrt{1-t^2})]^{-1}$.
For the polynomial costs, put $F_N(s)=\sum_{n=1}^Ns^n/(2n-1)$.
For $N\ge2$ and $0\le s\le1/2$,
\begin{equation}\label{g:finite-cost-kernel-polynomial}
 3F_N''-(1-2s)F_N'''=
 \sum_{k=0}^{N-3}\frac{(k+2)^2(k+1)}{2k+5}s^k
       +N(N-1)s^{N-2}>0
\end{equation}
If $\phi_N(t)=2F_N((1-t^2)/2)$ and
$w_N(t)=[t\phi_N'(t)-\phi_N(t)+\phi_N(0)]/t^2$, then
\begin{equation}\label{g:finite-cost-monotonicity}
 \phi_N'''(t)=2t[3F_N''-(1-2s)F_N''']>0,\qquad
 t^3w_N'(t)=\int_0^t u^2\phi_N'''(u)\,du>0.
\end{equation}
These assertions hold for $0<t\leq1$.
For $N=8$, $w_8(1)=2(F_8(1/2)-1)<0$, so $w_8$ is increasing and negative.

Relabel $x,y$ so that $\delta=m_y-m_x\ge0$. From
\eqref{g:normal-moments} and \eqref{g:hole},
\begin{equation}\label{g:genus-two-parameters}
 q=hv/P,\quad \mu=m_z=1-q,\quad
 m_x=(q-\delta)/2,\quad m_y=(q+\delta)/2,\quad c=\pi/(2h^2P).
\end{equation}
Set $D=\E(\nu_x^2-\nu_y^2)^2$ and
$Z=h^2P/2+\pi$. For the sphere function
$\phi(n)=n_z^2+(n_x^2-n_y^2)^2/4$, evaluated subsequently at $\nu$,
sphere differentiation gives
\begin{equation}\label{g:angular-test-derivatives}
 \Delta\phi=4-8\nu_z^2-5(\nu_x^2-\nu_y^2)^2,\qquad
 |\Hess^\circ\phi|^2
 =9(\nu_x^2-\nu_y^2)^2
         [2\nu_z^2+(\nu_x^2-\nu_y^2)^2/2].
\end{equation}
By \eqref{g:weak-gauss-divergence},
$\int_\Sigma B:\Hess\phi(\nu)\,d\sigma=0$. Subtracting the
trace part of $B$ and using \eqref{g:gauss-degree} gives
\begin{equation}\label{g:angular-stress-integrals}
 \int_\Sigma B^\circ:\Hess^\circ\phi\,d\sigma
       =hP(4\mu-2+5D/2),\qquad
 \int_\Sigma|B^\circ|^2d\sigma=Z,
\end{equation}
where $B^\circ=B-h\mathrm{Id}/2$ and the second equality in \eqref{g:angular-stress-integrals} is Gauss--Bonnet.
Since $(\nu_x^2-\nu_y^2)^2\le(1-\nu_z^2)^2$ and
$2u+(1-u)^2/2\le2$ for $u\in[0,1]$,
Cauchy--Schwarz in \eqref{g:angular-stress-integrals} gives the
first inequality below. The second is Jensen's inequality.
\begin{equation}\label{g:angular-stress-estimate}
 h^2P(4\mu-2+5D/2)^2\leq18ZD,\qquad D\geq\delta^2.
\end{equation}
When $4\mu-2\geq0$ the positive $D$ term can be dropped from the square;
otherwise use zero. This gives
\begin{equation}\label{g:D0}
 D\geq D_0:=\max\{\delta^2,
             h^2P\pos{4\mu-2}^2/(18Z)\}.
\end{equation}

\begin{lemma}\label{g:angular}
With $\ell_{\rm ang}=\sqrt{1-\nu_z^2}-(|\nu_x|+|\nu_y|)/\sqrt2$,
\begin{equation}\label{g:angular-pointwise}
 (\nu_x^2-\nu_y^2)^2[2\nu_z^2+(\nu_x^2-\nu_y^2)^2/2]
 \leq K_*\ell_{\rm ang},\qquad K_*=(32/5)(3/5)^{3/2}.
\end{equation}
Consequently
\begin{subequations}\label{g:angular-bounds}
\begin{align}
 \E\ell_{\rm ang}&\geq\Gamma_{\mathrm{ang}}:=
 \frac{h^2P\pos{4\mu-2+5D_0/2}^2}{9ZK_*},\label{g:angular-mean}\\
 L_x+L_y&\geq\E\phi_8(|\nu_z|)+D_0/6.\label{g:angular-log}
\end{align}
\end{subequations}
\end{lemma}
\begin{proof}
If $s=1-\nu_z^2=0$ or $\nu_x^2=\nu_y^2$, both sides of
\eqref{g:angular-pointwise} vanish.
Otherwise put $\eta=(\nu_x^2-\nu_y^2)^2/s^2$ and
$b_0=\sqrt{(1+\sqrt{1-\eta})/2}$.
The quotient in \eqref{g:angular-pointwise} is
\begin{equation}\label{g:angular-ratio}
 \frac{(\nu_x^2-\nu_y^2)^2[2\nu_z^2+(\nu_x^2-\nu_y^2)^2/2]}{\ell_{\rm ang}}
 =\frac{s^{3/2}\eta[2(1-s)+s^2\eta/2]}{1-b_0}.
\end{equation}
The identity
\begin{equation}\label{g:angular-ratio-reduction}
 8-\frac{5\eta}{2}-\frac{\eta}{1-b_0}
 =(1-b_0)^2(10b_0^2+16b_0+8)\geq0
\end{equation}
gives the bound on the right side of \eqref{g:angular-ratio},
\[
 \frac{s^{3/2}\eta[2(1-s)+s^2\eta/2]}{1-b_0}
 \le s^{3/2}(8-5\eta/2)[2(1-s)+s^2\eta/2]
 =:f(s,\eta).
\]
Expanding gives
$f(s,\eta)=s^{3/2}[16(1-s)+(4s^2+5s-5)\eta-(5/4)s^2\eta^2]$.
For $s>0$, the maximizer in $\eta\in[0,1]$ is
$\min\{1,\max\{0,2(4s^2+5s-5)/(5s^2)\}\}$.
The expression $2(4s^2+5s-5)/(5s^2)$ equals $0$ at $r_1=(\sqrt{105}-5)/8$ and
$1$ at $r_2=(\sqrt{55}-5)/3$. Substitution gives
\begin{equation}\label{g:angular-piecewise}
 \max_{0\leq\eta\leq1}f(s,\eta)=
 \begin{cases}
 16s^{3/2}(1-s),&0\leq s\leq r_1,\\
 Q(s)/(5\sqrt s),&r_1\leq s\leq r_2,\\
 (11/4)s^{3/2}(2-s)^2,&r_2\leq s\leq1,
 \end{cases}
 \quad Q(s)=16s^4-40s^3+65s^2-50s+25.
\end{equation}
The maximum of $16s^{3/2}(1-s)$ on $[0,r_1]$ is $K_*$ at
$s=3/5$. On $[r_2,1]$ the maximum of $(11/4)s^{3/2}(2-s)^2$
occurs at $s=6/7$ and its square divided by $K_*^2$ is
$756250/823543<1$. On $[r_1,r_2]$,
\begin{equation}\label{g:angular-middle-derivative}
 \begin{aligned}
 \left(\frac{Q(s)}{5\sqrt s}\right)'&=\frac{D(s)}{10s^{3/2}},\\
 D(s)&=112s^4-200s^3+195s^2-50s-25,\\
 D''(s)&=1344(s-25/56)^2+855/7>0.
 \end{aligned}
\end{equation}
Also $D'(13/20)=9129/125>0$ and $r_1>13/20$.
Thus $D$ increases on $[r_1,r_2]$. The derivative in
\eqref{g:angular-middle-derivative} can change sign only from
negative to positive. By \eqref{g:angular-piecewise},
\[
 \max_{r_1\le s\le r_2}\frac{Q(s)}{5\sqrt s}
 =\max\{16r_1^{3/2}(1-r_1),(11/4)r_2^{3/2}(2-r_2)^2\}\le K_*.
\]
This proves \eqref{g:angular-pointwise}.
For $s=3/5$ and $\eta\downarrow0$, the original ratio tends to $K_*$,
so this pointwise constant cannot be reduced.

Cauchy--Schwarz in the trace-free identity giving
\eqref{g:angular-stress-estimate}, now with \eqref{g:angular-pointwise},
gives
\begin{equation}\label{g:angular-mean-chain}
 h^2P^2\pos{4\mu-2+5D/2}^2\leq9ZPK_*\E\ell_{\rm ang}.
\end{equation}
Replacing $D$ by $D_0\leq D$ proves \eqref{g:angular-mean}.
For \eqref{g:angular-log}, put $a=\nu_x^2,b=\nu_y^2$;
$a,b\geq0$ and $a+b\leq1$.
Convexity gives
$[a^n+b^n-2((a+b)/2)^n]/(2n-1)\ge0$ for $n\ge1$.
For $n=2$ this difference equals $(a-b)^2/6$. Hence
\begin{equation}\label{g:balanced-log-chain}
 \sum_{n\geq1}\frac{a^n+b^n}{2n-1}
 \geq2\sum_{n=1}^8\frac{((a+b)/2)^n}{2n-1}+\frac{(a-b)^2}{6}.
\end{equation}
In \eqref{g:balanced-log-chain}, $a+b=1-\nu_z^2$ and
$\E(a-b)^2=D$. Taking expectations by monotone convergence
on the left gives $L_x+L_y\ge\E\phi_8(|\nu_z|)+D/6$.
Use $D\ge D_0$ to obtain \eqref{g:angular-log}.
\end{proof}

\subsection{An inequality for the coordinate sections}\label{g:walls}

In this subsection assume $v=b$, $P\le1$ and $h>1$.
Proposition~\ref{s:low-h} proves $h>5881/5000>1$ for the genus-two
candidates to which these estimates will be applied.
We estimate the upper and lower coordinate sections
separately, then combine the bounds through the stress identities to obtain
\eqref{g:raw-dual-eq}. In the original cube set
$S=hw(0,0)$ and
\[
 p_x=h[w(0,0)-w(0,1)]=S-B_y,\qquad
 p_y=h[w(0,0)-w(1,0)]=S-B_x.
\]
Both are nonnegative by downward monotonicity and
\eqref{g:wall-variables}.
For $i=x,y$, \eqref{g:stress-length}, \eqref{g:stress-eq} and
\eqref{g:wall-variables} give
\begin{equation}\label{g:wall-stress}
 J_i=hL_{i,1}-h^2A_{i,1}=hP(1-m_i)-h^2/\pi.
\end{equation}
\begin{equation}\label{g:scaled-lower-length}
 hL_{i,0}=\ell_i+h^2N_i .
\end{equation}
With $\sigma_i^2=\pi/4-J_i$, \eqref{g:genus-two-parameters} gives
\begin{equation}\label{g:pressure-radicands}
 \sigma_x^2=\pi/4-hP+3h^2/(2\pi)-hP\delta/2,\qquad
 \sigma_y^2=\pi/4-hP+3h^2/(2\pi)+hP\delta/2 .
\end{equation}
For $0\le t\le\pi/4$ set
\begin{equation}\label{g:wall-constants}
 d(t)=4\sqrt{\pi/16+3t/4+(1-\sin2t-\cos2t)/8}.
\end{equation}
Since $d(0)^2=\pi$ and $(d^2)'=12+4\sin2t-4\cos2t>0$, $d(t)\ge\sqrt\pi$.

\begin{lemma}\label{g:upper-calibration}
For $\gamma_{i,1}$ on the face $x_i=1$, $i=x,y$, use the quantities of \eqref{g:wall-variables}
and put $\sigma_i=\sqrt{\pi/4-J_i}$. Then for $0<t<\pi/4$,
\begin{equation}\label{g:upper-bound}
 \pi/2+2t-\sin t(B_i+R_i)-d(t)\sigma_i-\ell_i\ge0.
\end{equation}
\end{lemma}
\begin{proof}
Suppress $i$ on the quantities in \eqref{g:wall-variables}.
The kernel $k$ of \eqref{g:wall-kernel-equation} is nonnegative,
and $\int_0^{\pi/2}k(\theta,\varphi)\,d\varphi
=\sin^2\theta+\cos^2\theta=1$. Since $0\le\lambda\le1$,
\[
 0\le A(1-\lambda)=\pi/4-\ell+A(\lambda)=\sigma^2.
\]
For $\lambda=2\ell/\pi+\sum_{j\ge1}a_j\cos(2j\theta)$,
the Neumann equation \eqref{g:wall-kernel-equation} and Parseval give
\begin{equation}\label{g:calibration-area-spectrum}
 A(\lambda)=\ell^2/\pi-\frac\pi8\sum_{j\ge1}\frac{a_j^2}{4j^2-1}.
\end{equation}
Define the vector
$Z_\lambda=((\pi a_j/\sqrt{8(4j^2-1)})_{j\ge1},\sqrt\pi\,\sigma)$
in $\ell^2\oplus\mathbb R$. Equation \eqref{g:calibration-area-spectrum}
gives $\|Z_\lambda\|=\pi/2-\ell$.
Let $\lambda_t=\tfrac12\mathbf 1_{[t,\pi/2-t]}$ and apply the
same construction to $\lambda_t$. Then $\ell_t=\pi/4-t$ and
\[
 A(\lambda_t)=\frac14\int_t^{\pi/2-t}\sin\varphi(\sin\varphi-\sin t)\,d\varphi
 =\pi/16-t/4+(1-\sin2t-\cos2t)/8.
\]
In particular $\|Z_{\lambda_t}\|=\pi/4+t$ and its last coordinate
is $\sqrt\pi\,d(t)/4$.
With $g_{\lambda_t}$ from \eqref{g:wall-kernel-equation},
integration of $k(\theta,\cdot)$ on $[t,\pi/2-t]$ gives
\begin{equation}\label{g:calibration-weight}
 1-2g_{\lambda_t}(\theta)-\sin t(\cos\theta+\sin\theta)=
 \begin{cases}
 1-\cos(\theta-t),&0\le\theta\le t,\\
 0,&t\le\theta\le\pi/2-t,\\
 1-\cos(\theta-(\pi/2-t)),&\pi/2-t\le\theta\le\pi/2 .
 \end{cases}
\end{equation}
Write $w_t$ for this nonnegative function and $a_{j,t}$ for
the cosine coefficients of $\lambda_t$. By
\eqref{g:wall-kernel-equation}, the coefficients of $g_{\lambda_t}$
are $2\ell_t/\pi$ and $-a_{j,t}/(4j^2-1)$. Thus
\[
 \begin{aligned}
 2\int_0^{\pi/2}g_{\lambda_t}\lambda\,d\theta
 &=\frac4\pi\ell\ell_t-\frac\pi2\sum_{j\ge1}\frac{a_j a_{j,t}}{4j^2-1}\\
 &=\frac4\pi(\ell\ell_t-\langle Z_\lambda,Z_{\lambda_t}\rangle)+d(t)\sigma.
 \end{aligned}
\]
Substitute this identity into $\int w_t\lambda$, using
$B+R=\int(\cos\theta+\sin\theta)\lambda(\theta)d\theta$. Then
\begin{equation}\label{g:calibration-square}
 \begin{aligned}
 &\pi/2+2t-\sin t(B+R)-d(t)\sigma-\ell\\
 &\quad=\frac4\pi\bigl[(\pi/2-\ell)(\pi/4+t)
              -\langle Z_\lambda,Z_{\lambda_t}\rangle\bigr]
     +\int_0^{\pi/2}w_t\lambda\,d\theta\ge0,
 \end{aligned}
\end{equation}
by Cauchy--Schwarz and \eqref{g:calibration-weight}.
\end{proof}

\begin{lemma}\label{g:turn}
For the section $x=0$ with $h>1$ and $p=h[w(0,0)-w(0,1)]$,
\begin{equation}\label{g:turn-eq}
 hL_{x,0}\geq L_h(p):=\sup_{0\leq\theta<\pi/2}
 [h\cos\theta+(p-1)\sin\theta+\theta].
\end{equation}
For $r<1$ this implies $L_h(p)-rp\geq C_h(r)$, where
\begin{equation}\label{g:turn-dual}
 C_h(r)=
 \begin{cases}
 h\sqrt{1-r^2}+\arcsin r-r,&0\leq r<1,\\
 h,&r\leq0.
 \end{cases}
\end{equation}
\end{lemma}
\begin{proof}
In the section $x=0$, let $f(y)=w(0,0)-w(0,y)$.
By \eqref{g:genus-two-axis-height}, $w(0,y)>0$ for $0\le y\le1$,
and inactivity of $z$ gives $w(0,y)\le1-\varepsilon$.
The strip inequality $\nu_z>0$ makes
$(0,y,w(0,0)-f(y))$ a smooth graph by the implicit function theorem,
also at the endpoints after reflection. Here
$f'\ge0$, $f(0)=0$, $f(1)=p/h$, and $f'(0)=f'(1)=0$.
For $\theta=\arctan f'$, \eqref{g:lower-section-angle} and
\eqref{g:hopf} give $d\theta/ds=\kappa_{x,0}=h-B_{xx}<h$.
Let $M=\max\theta<\pi/2$. For $m\in[0,M]$ let $s_m$ be the first
arc-length parameter with $\theta(s_m)=m$. For continuous $\chi\ge0$ on $[0,M]$,
\begin{equation}\label{g:first-angle-visit}
 h\int_0^{L_{x,0}}\chi(\theta(s))\,ds
 \ge h\int_0^{s_m}\chi(\theta(s))\,ds
 \ge\int_0^{s_m}\chi(\theta(s))\theta'(s)\,ds
 =\int_0^m\chi(\theta)\,d\theta .
\end{equation}
Taking $\chi(\theta)=\tan M\cos\theta-\sin\theta$ and $m=M$ in
\eqref{g:first-angle-visit}, with
$\int_{\gamma_{x,0}}\cos\theta\,ds=1$ and
$\int_{\gamma_{x,0}}\sin\theta\,ds=p/h$, gives
\[
 h\tan M-p\ge\int_0^M(\tan M\cos\theta-\sin\theta)\,d\theta
             =\sec M-1.
\]
Put $G_h(\theta)=h\tan\theta-\sec\theta+1$. Since
$G_h'=(h-\sin\theta)/\cos^2\theta>0$, $G_h(0)=0$ and
$G_h(M)\ge p$, there is a unique $\theta_p\in[0,M]$ with $G_h(\theta_p)=p$.
Apply \eqref{g:first-angle-visit} to $1-\cos(\theta-\theta_p)$:
\begin{equation}\label{g:turning-maximizer}
 hL_{x,0}\geq h\cos\theta_p+(p-1)\sin\theta_p+\theta_p,\qquad
 [h\cos\theta+(p-1)\sin\theta+\theta]'
 =\cos\theta\,[p-G_h(\theta)].
\end{equation}
Thus $\theta_p$ is the global maximizer, proving \eqref{g:turn-eq}.
Evaluate it at $\arcsin r$ for $0\leq r<1$ to obtain \eqref{g:turn-dual}.
For $r\leq0$, $L_h(p)-rp\geq h-rp\geq h$.
If $p=0$, the graph is constant, $\theta_p=M=0$, and $hL_{x,0}=h$.
\end{proof}

\begin{lemma}\label{g:raw-dual}
For every $\eta\in[0,1]$,
\begin{equation}\label{g:height}
 S\geq1-P+\frac{3h}{2\pi}
 +\frac{\eta(p_x+p_y)/2-(\eta-F(\eta))}{h}.
\end{equation}
For $t_x,t_y\in(0,\pi/4)$ write
$(s_i,A_i,d_i)=(\sin t_i,\pi/2+2t_i,d(t_i))$ for $i=x,y$.
Put $s=s_x+s_y$, $r_1=s_y-s\eta/(2h)$ and $r_2=s_x-s\eta/(2h)$.
Then
\begin{equation}\label{g:raw-dual-eq}
 \begin{split}
 C_h(r_1)+C_h(r_2)
 &+s[1-P+3h/(2\pi)-(\eta-F(\eta))/h]\\
 &+s_xR_x+s_yR_y+d_x\sigma_x+d_y\sigma_y
 \leq A_x+A_y+h^2(N_x+N_y).
 \end{split}
\end{equation}
\end{lemma}
\begin{proof}
Define $f_x(y)=w(0,0)-w(0,y)$ and $f_y(x)=w(0,0)-w(x,0)$.
Then $f_i(1)=p_i/h$ and $A_{i,0}=w(0,0)-\int_0^1f_i$.
Substituting this $A_{i,0}$ in \eqref{g:stress-eq} gives
\begin{equation}\label{g:lower-wall-height-identity}
 P(1-m_i)-h/\pi+S
 =\int_0^1[\sqrt{1+f_i'^2}+hf_i]\,dy,\qquad i=x,y,
\end{equation}
with $y$ denoting the corresponding axis variable.
Use
\[
 \chi(y)=
 \begin{cases}
  0,&0\le y\le1-\eta/h,\\
  \eta-h(1-y),&1-\eta/h\le y\le1,
 \end{cases}
 \qquad \chi'=h\mathbf1_{(1-\eta/h,1)}\quad\text{a.e.}
\]
Then $\chi(0)=0$, $\chi(1)=\eta$, and
$\sqrt{1+f_i'^2}\ge\chi f_i'+\sqrt{1-\chi^2}$ by Cauchy--Schwarz.
Integration by parts gives
\begin{equation}\label{g:height-ramp}
 \begin{aligned}
 \int_0^1(\chi f_i'+hf_i)\,dy
 &=\eta f_i(1)+\int_0^1(h-\chi')f_i\,dy\geq\eta p_i/h,\\
 \int_0^1\sqrt{1-\chi^2}\,dy&=1-(\eta-F(\eta))/h.
 \end{aligned}
\end{equation}
Adding the two instances of \eqref{g:lower-wall-height-identity} and
using $m_x+m_y=h/(\pi P)$ proves \eqref{g:height}.

Combine \eqref{g:upper-bound}, \eqref{g:turn-eq} and
\eqref{g:scaled-lower-length}, with $B_x=S-p_y$, $B_y=S-p_x$:
\begin{equation}\label{g:combined-wall-before-height}
 \begin{split}
 L_h(p_x)+L_h(p_y)+sS-s_xp_y-s_yp_x
 &+s_xR_x+s_yR_y+d_x\sigma_x+d_y\sigma_y\\
 &\leq A_x+A_y+h^2(N_x+N_y).
 \end{split}
\end{equation}
By \eqref{g:height},
\[
 sS-s_xp_y-s_yp_x
 \ge s[1-P+3h/(2\pi)-(\eta-F(\eta))/h]-r_1p_x-r_2p_y .
\]
Substitute in \eqref{g:combined-wall-before-height} and apply
\eqref{g:turn-dual} to $(p_x,r_1)$ and $(p_y,r_2)$ to obtain
\eqref{g:raw-dual-eq}.
\end{proof}

\section{Stability bounds and preliminary exclusions}\label{sec:stability-exclusions}

\revision{We first recall the classical height estimates in
Proposition~\ref{s:height-stability} and strengthen them in
Proposition~\ref{s:q-bound}. The finite moment bound
\eqref{s:finite-moment}, combined with the classical projection
estimates, then gives $h>5881/5000$ in Proposition~\ref{s:low-h}.
Proposition~\ref{s:n3} excludes the genus-three system
\eqref{g:genus-three-eq}.}

\subsection{Stability inequalities for the height function}
\label{s:height}

Let $E\subset T^2\times\R$ be compact with a connected, smooth,
volume-constrained stable boundary $\Sigma=\partial E$ of genus two
and constant $h=\tr B>0$. Set $\sig=\dd A/8$, $P=\int_\Sigma\sig$,
$v=|E|/8$, $q=hv/P$ and $\mu=1-q$.

For the lift of a reflected minimizer, inactivity of $z$ gives
$E\subset T^2\times[-1+\varepsilon,1-\varepsilon]$ for some
$\varepsilon>0$ by Lemma~\ref{g:downset}. The quotient
$(x,y,z)\mapsto(x,y,z\bmod 2)$ is injective and isometric locally
on $T^2\times(-1,1)$. Small volume-preserving variations of the
lift therefore preserve volume on $T$ and satisfy \eqref{g:stability}.

\revision{The following estimates are the normalized form of
\cite[Proposition~12 and Theorem~13]{HPRR}, with
$A=8P$, $V=8v$, $H=h/2$ and $g=2$ in their notation.
Their argument uses only volume-constrained stability. We retain
the Jacobi calculation to record the bilinear identities needed
for Proposition~\ref{s:q-bound}.}

\begin{proposition}\label{s:height-stability}
One has
\begin{align}
 \frac1P\int_\Sigma\nu_z^4\,\sig
 &\le 1+q-3q^2-\frac{\pi v^2}{P^3},
 \label{s:height-inequalities}\\
 3q-4q^2&\ge\frac{\pi v^2}{P^3}.
 \label{s:height-perimeter}
\end{align}
In particular, $P^3\ge16\pi v^2/9$.
\end{proposition}
\begin{proof}
\revision{The height function and the vertical translation Jacobi field satisfy}
\begin{equation}\label{s:height-Jacobi}
 \Delta_\Sigma z=-h\nu_z,\qquad
 |\nabla_\Sigma z|^2=1-\nu_z^2,\qquad
 (\Delta_\Sigma+|B|^2)\nu_z=0.
\end{equation}
\revision{The divergence theorem in $E$ and on $\Sigma$ gives}
\begin{align}
 \int_\Sigma z\nu_z\,\sig&=v,\label{s:height-flux}\\
 0=\int_\Sigma\Div_\Sigma(z\nabla_\Sigma z)\,\sig
 &=P-\int_\Sigma\nu_z^2\,\sig-hv.\label{s:height-second-moment}
\end{align}
\revision{Thus $M_z=P\mu$. Testing the Jacobi equation with $z^2\nu_z$ gives}
\begin{equation}\label{s:height-Q}
 Q(z\nu_z,z\nu_z)=\int_\Sigma\nu_z^2|\nabla_\Sigma z|^2\,\sig
 =M_z-\int_\Sigma\nu_z^4\,\sig.
\end{equation}
\revision{Testing with $z$ gives}
\begin{equation}\label{s:height-Q-cross}
 -Q(z\nu_z,1)
 =\int_\Sigma |B|^2z\nu_z\,\sig
 =-\int_\Sigma\nu_z\Delta_\Sigma z\,\sig=hM_z.
\end{equation}
Gauss--Bonnet in genus two gives
\begin{equation}\label{s:height-Gauss-Bonnet}
 Q(1,1)=-\int_\Sigma(h^2-2K)\,\sig=-h^2P-\pi.
\end{equation}
\revision{For the mean-zero function $f=z\nu_z-v/P$, these identities and
volume-constrained stability give}
\begin{equation}\label{s:height-stability-computation}
 0\le P^{-1}Q(f,f)
 =\mu-\frac1P\int_\Sigma\nu_z^4\,\sig+2q\mu-q^2-\frac{\pi v^2}{P^3}.
\end{equation}
\revision{Since $\mu=1-q$, this is \eqref{s:height-inequalities}. Jensen's
inequality gives $P^{-1}\int_\Sigma\nu_z^4\,\sig\ge\mu^2$, hence}
\begin{equation}\label{s:height-completed-square}
 \frac{\pi v^2}{P^3}\le\mu-\mu^2+2q\mu-q^2
 =3q-4q^2=\frac9{16}-4(q-3/8)^2\le\frac9{16}.
\end{equation}
\end{proof}

\revision{Under $h^2P\le\pi$, \eqref{s:height-perimeter} gives $q\le3/5$.
The next proposition improves this bound by adding the mean-zero
part of $|\nu_z|$ as a second stability test.}

\begin{proposition}\label{s:q-bound}
If, in addition, $h^2P\le\pi$, then $q<7/12$. At $v=1/\pi$ this implies
\begin{equation}\label{eq:body-022}
 P>\frac{12h}{7\pi},\qquad h^3<\frac{7\pi^2}{12}.
\end{equation}
\end{proposition}
\begin{proof}
Let $\omega=\pi/(h^2P)$, $\alpha=\E|\nu_z|$, and
$k=|\nu_z|-\alpha$, where the expectation is with respect to $\sig/P$.
The functions $f=z\nu_z-v/P$ and $k$ have zero integral on $\Sigma$.
The $H^1$ chain rule gives $|\nabla_\Sigma|\nu_z||^2=|\nabla_\Sigma\nu_z|^2$
almost everywhere and
$|\nu_z|\nabla_\Sigma\nu_z=\nu_z\nabla_\Sigma|\nu_z|$.
Consequently, testing the Jacobi equation in \eqref{s:height-Jacobi}
with $\nu_z$ and $z|\nu_z|$ gives
\begin{equation}\label{s:absolute-Jacobi-tests}
 Q(|\nu_z|,|\nu_z|)=Q(\nu_z,\nu_z)=0,\qquad
 Q(z\nu_z,|\nu_z|)=Q(\nu_z,z|\nu_z|)=0.
\end{equation}
The signed Gauss-map formula gives
\begin{equation}\label{eq:body-024}
 \int_\Sigma K|\nu_z|\,\sig=-\frac18\int_{\Sph}|n_z|\,\dd\omega
 =-\frac\pi4.
\end{equation}
Thus $Q(1,|\nu_z|)=-h^2P\alpha-\pi/2$. Put
\begin{equation}\label{eq:body-025}
 A=\frac{Q(k,k)}P=h^2\alpha^2+\frac\pi P\alpha(1-\alpha),
 \qquad C=\frac{Q(f,k)}P=h\alpha\mu+\frac{\pi v}{P^2}\left(\frac12-\alpha\right).
\end{equation}
By \eqref{s:height-flux}, $v>0$ implies $\alpha>0$, so $A>0$.
Stability for $f-(C/A)k$ and the bound $\int_\Sigma\nu_z^4\,d\sigma/P\ge\mu^2$
give
\[
 \frac{C^2}{A}\le\frac{Q(f,f)}P
 \le \mu-\mu^2+2q\mu-q^2-\frac{\pi v^2}{P^3}.
\]
Using $\mu=1-q$ and $\pi v^2/P^3=\omega q^2$ yields
\begin{equation}\label{s:q-matrix}
 3q-(4+\omega)q^2\ge\frac{C^2}{A}.
\end{equation}
By \eqref{s:q-matrix}, $3q\ge(4+\omega)q^2\ge5q^2$.
Suppose $q\ge7/12$. Then
\begin{equation}\label{eq:body-026}
 \frac7{12}\le q\le\frac35,\qquad
 1\le\omega\le\frac3q-4\le\frac87.
\end{equation}
Put $s=\sqrt\mu$. Then $1/2<s<2/3$, $\mu\geq2/5$,
$\omega q\leq3-4q\leq2/3$ and $0<\alpha\leq s$.
By \eqref{eq:body-025},
\begin{equation}\label{s:q-entries}
 C/h=\alpha\mu+\omega q(1/2-\alpha),\qquad
 A/h^2=\alpha^2+\omega\alpha(1-\alpha).
\end{equation}
On $0<\alpha\le s$, \eqref{eq:body-026} gives
\begin{align*}
 \partial_\alpha(C/h)&=\mu-\omega q\le1-2q\le-1/6,\\
 \partial_\alpha(A/h^2)&=\omega-2\alpha(\omega-1)\ge2-\omega
 \ge 3\omega/4.
\end{align*}
Define $N=(C/h)|_{\alpha=s}$ and $D=(A/h^2)|_{\alpha=s}$. Thus
\[
 N=\mu s-\omega q(s-1/2)\ge1/3-4s/15>7/45,\qquad
 D=\mu+\omega(s-\mu)\ge s.
\]
The derivative signs show $C/h\ge N>0$ and $0<A/h^2\le D$, so
\begin{equation}\label{s:q-R}
 \frac{C^2}{A}\geq R(q,\omega):=
 \frac{[\mu s-\omega q(s-1/2)]^2}{\mu+\omega(s-\mu)}.
\end{equation}
Since $N\le\mu s$ and $D\ge s$, $0<N/D\le\mu\le1/2$. Thus
\begin{equation}\label{s:q-omega-monotonicity}
 -R_\omega=2q(s-1/2)N/D+(s-\mu)(N/D)^2
 \leq q/6+1/16\leq13/80.
\end{equation}
For $F(q,\omega)=3q-(4+\omega)q^2-R(q,\omega)$,
$F_\omega\leq-49/144+13/80=-8/45<0$.
At $\omega=1$, $N|_{\omega=1}=q/2-(2q-1)s>(4-5q)/6>0$ and
$R(q,1)=(N|_{\omega=1})^2/s>(4-5q)^2/24$.
Writing $r=q-7/12\geq0$ gives
\begin{equation}\label{s:q-polynomial-contradiction}
 F(q,\omega)\leq F(q,1)
 <-\frac{145q^2-112q+16}{24}
 =-\frac{1/144+(343/6)r+145r^2}{24}
 \leq-\frac1{3456}<0.
\end{equation}
This contradicts \eqref{s:q-matrix}.
At $v=1/\pi$, the inequality $q=h/(\pi P)<7/12$ gives
$P>12h/(7\pi)$, so $12h^3/(7\pi)<h^2P\le\pi$, as asserted in
\eqref{eq:body-022}.
\end{proof}

\subsection{Polynomial bounds for curvature-distribution tails}
\label{s:cap}

Use $\alpha=\E X$, $\mu=\E X^2$ and $a(t)=\E[X\mathbf1_{\{X>t\}}]$
as in Section~\ref{g:probability}. The following estimates hold for
every $[0,1]$-valued random variable satisfying \eqref{g:tail-envelope}.

\begin{lemma}\label{s:cap-polynomial}
For $\alpha,c\ge0$ and $A=\alpha+c>0$, define
\begin{equation}\label{s:finite-cap-recurrence}
 k_1=\frac{\alpha^2}{2A},\qquad
 k_j=\frac cA k_{j-1}+\frac1{2A}\sum_{i=1}^{j-1}k_i k_{j-i}\quad(j\ge2).
\end{equation}
For each positive integer $N$ and every $0\le t\le1$,
\begin{equation}\label{s:finite-cap-envelope}
 U(\alpha,c;t)\le U_N(\alpha,c;t):=\alpha-\sum_{j=1}^Nk_jt^{2j}.
\end{equation}
If \eqref{g:tail-envelope} holds, then, for every $0<R\le1$,
\begin{equation}\label{s:finite-moment}
 \mu\ge 2\alpha R-R^2+
 \sum_{j=1}^N\frac{2k_jR^{2j+1}}{(2j-1)(2j+1)}.
\end{equation}
\end{lemma}
\begin{proof}
By the definition of $U$ in \eqref{g:tail-envelope},
$U^2+2c(1-t^2)U=\alpha(\alpha+2c)(1-t^2)$.
With $s=t^2$ and $V=\alpha-U\ge0$, this becomes
\begin{equation}\label{s:finite-cap-quadratic}
 V=\frac{\alpha^2}{2A}s+\frac cA sV+\frac1{2A}V^2.
\end{equation}
Start with $V_0=0$ and define $V_{m+1}$ by the right side with $V_m$ in place of $V$. Monotonicity gives $0\le V_m\le V$. All coefficients of $V_m$ are nonnegative, and induction shows that its first $m$ coefficients are $k_1,\ldots,k_m$. Thus $V\ge\sum_{j=1}^Nk_js^j$, including at $s=1$.

For $x\in(0,R]$ and $x\in(R,1]$, respectively,
\[
 (x-R)^2=x\int_x^R(R^2/t^2-1)\,\dd t,\qquad
 (x-R)^2=x\int_R^x(1-R^2/t^2)\,\dd t.
\]
The contribution of $X=0$ is $R^2\Pr(X=0)$. Taking expectations
and applying Tonelli to the two nonnegative integrands gives
\begin{equation}\label{s:moment-nonnegative-decomposition}
 \begin{split}
 \mu-2\alpha R+R^2
 &=R^2\Pr(X=0)+\int_0^R(R^2/t^2-1)(\alpha-a(t))\,dt\\
 &\quad+\int_R^1(1-R^2/t^2)a(t)\,dt.
 \end{split}
\end{equation}
Substituting $\alpha-a(t)\ge\alpha-U(\alpha,c;t)\ge
\sum_{j=1}^Nk_jt^{2j}$ into the integral on $[0,R]$ in
\eqref{s:moment-nonnegative-decomposition} gives
\[
 \mu-2\alpha R+R^2
 \ge\sum_{j=1}^Nk_j\int_0^R(R^2/t^2-1)t^{2j}\,\dd t
 =\sum_{j=1}^N\frac{2k_jR^{2j+1}}{(2j-1)(2j+1)}.
\]
\end{proof}

\subsection{A curvature lower bound in genus two}
\label{s:low-curvature}

\begin{proposition}\label{s:low-h}
Let $E\subset Q$ be a downward monotone minimizer with
$v\ge1/\pi$ and $P\le1$. If its reflected boundary is connected
of genus two and has $h>0$, then $h>H_0:=5881/5000$.
\end{proposition}
\begin{proof}
Choose $z$ as the inactive coordinate by Lemma~\ref{g:downset}.
\revision{The projection and flux estimates used in the proof of
\cite[Proposition~17]{HPRR}, in the normalization of
\eqref{g:hole} and \eqref{g:hole-isoperimetry}, give}
\begin{equation}\label{s:hole}
 M_z=P-hv,\qquad N_z^2\le P(P-hv),\qquad
 1-N_z\le h^2N_z^2/\pi.
\end{equation}
Solving the quadratic inequality for $N_z>0$ gives
\begin{equation}\label{eq:body-031}
 N_z\ge n(h):=\frac2{1+\sqrt{1+4h^2/\pi}}.
\end{equation}
For $0<h\le1/2$, put $r=\sqrt{1-h/\pi}<1$. Then $N_z\le r$. Since $(1-N)/N^2$ decreases, \eqref{s:hole} implies $1-r\le h^2r^2/\pi$. But $1-r=h/[\pi(1+r)]$, so $hr^2(1+r)\ge1$, contradicting $h\le1/2$ and $r<1$.

\revision{To exclude $1/2\le h\le H_0$, we combine these projection
estimates with the finite moment bound in Lemma~\ref{s:cap-polynomial}.
Put $a=n(h)$ and $c_0=\pi/(2h^2)$. They satisfy}
\begin{equation}\label{eq:body-032}
 c_0=\frac{a^2}{2(1-a)},\qquad d=\frac a{2-a},
\end{equation}
and the coefficients from Lemma~\ref{s:cap-polynomial} at $(a,c_0)$ obey
\begin{equation}\label{s:low-k}
 k_1=\frac{a(1-a)}{2-a},\qquad
 k_j=dk_{j-1}+\frac{1-a}{a(2-a)}\sum_{i=1}^{j-1}k_i k_{j-i}.
\end{equation}
Fix $R=73/80$ and set
\begin{equation}\label{s:low-eleven-function}
 B_{11}(a)=2aR-R^2+\sum_{j=1}^{11}w_jk_j(a),\qquad
 w_j=\frac{2R^{2j+1}}{(2j-1)(2j+1)}.
\end{equation}
We show that $B_{11}(n(h))>1-h/\pi$ for $1/2\le h\le H_0$.
In this interval $n(H_0)\le a\le n(1/2)$, and substitution
in \eqref{eq:body-031} gives $3/4<a<14/15$, hence $3/5<d<7/8$.
For the coefficients $k_j(a)$ of \eqref{s:low-k} with $j\ge2$,
\begin{equation}\label{s:semicircle}
 k_j(a)=\frac{2d(1-d^2)}\pi
 \int_0^1\sqrt{y(1-y)}[d^2+(1-d^2)y]^{j-2}\,\dd y.
\end{equation}
Indeed $d=c_0/(a+c_0)$ and \eqref{g:tail-envelope} gives
\[
 \frac{a-U(a,c_0;t)}{a+c_0}
 =1-ds-\sqrt{(1-s)(1-d^2s)},\qquad s=t^2.
\]
For $0<s<1$, let $m=(1+d^2)/2$, $r=(1-d^2)/2$,
$A=1-sm$ and $B=sr$. Substituting $x=m+r\cos\theta$ gives
\begin{equation}\label{eq:body-034}
 \frac{s^2}\pi\int_{d^2}^1
 \frac{\sqrt{(1-x)(x-d^2)}}{1-sx}\,\dd x
 =\frac{B^2}{\pi}\int_0^\pi\frac{\sin^2\theta}{A-B\cos\theta}\,\dd\theta
 =A-\sqrt{A^2-B^2}.
\end{equation}
For the last equality, $u=\tan(\theta/2)$ gives
\[
 \int_0^\pi\frac{\dd\theta}{A-B\cos\theta}
 =2\int_0^\infty\frac{\dd u}{A-B+(A+B)u^2}
 =\frac{\pi}{\sqrt{A^2-B^2}}.
\]
Also $B^2\sin^2\theta=(A+B\cos\theta)(A-B\cos\theta)-(A^2-B^2)$.
Since $A^2-B^2=(1-s)(1-d^2s)$, expansion of $(1-sx)^{-1}$ at
$s=0$ in \eqref{eq:body-034} gives the coefficients of order $j\ge2$.
After $x=d^2+(1-d^2)y$, the factor multiplying the integral is
$(a+c_0)(1-d^2)^2/\pi=2d(1-d^2)/\pi$. This yields
\eqref{s:semicircle}.

Since $d'=(1+d)^2/2$, the derivative of the power in \eqref{s:semicircle} is nonnegative. The prefactor derivative is negative, the power is at most one, and $\int_0^1\sqrt{y(1-y)}\,\dd y=\pi/8$. Consequently
\begin{equation}\label{eq:body-036}
 k_j'(a)\ge\frac{(1+d)^2(1-3d^2)}8\ge-\frac{18675}{32768}
 \quad(j\ge2),\qquad
 k_1'(a)=1-\frac2{(2-a)^2}\ge-\frac{97}{128}.
\end{equation}
By $2/[(2j-1)(2j+1)]=(2j-1)^{-1}-(2j+1)^{-1}$,
\begin{equation}\label{s:low-weight-sum}
 \sum_{j\ge1}w_j=R-(1-R^2)\operatorname{atanh}R.
\end{equation}
Also $\operatorname{atanh}(73/80)>3/2$: it suffices to note $e<11/4$, hence $e^3<153/7$. The exponential series through $1/3!$, with remaining tail at most $(1/4!)/(1-1/5)=5/96$, proves the bound on $e$. Therefore
\begin{equation}\label{eq:body-038}
 B_{11}'(a)\ge2R-\frac{97}{128}\frac{2R^3}3
 -\frac{18675}{32768}\left[R-\frac32(1-R^2)-\frac{2R^3}3\right]
 =\frac{34045177631}{25165824000}>\frac{27}{20}.
\end{equation}
On the other hand,
\begin{equation}\label{eq:body-039}
 \left|\left(\frac{\sqrt{1-a}}{\sqrt\pi a}\right)'\right|
 =\frac{2-a}{2\sqrt\pi a^2\sqrt{1-a}}<\frac{27}{20}.
\end{equation}
The derivative of $\log[(2-a)/(a^2\sqrt{1-a})]$ is
$(12a-3a^2-8)/[2a(1-a)(2-a)]$. Its numerator is increasing for
$3/4\le a\le14/15$, so $(2-a)/(a^2\sqrt{1-a})$ attains its
maximum at one of these endpoints. The maximum is $60\sqrt{15}/49$.
With $\pi>25/8$ this bounds the square in \eqref{eq:body-039}
by $4320/2401<729/400$. Thus
\begin{equation}\label{s:low-D}
 D(a):=B_{11}(a)-1+\frac{\sqrt{1-a}}{\sqrt\pi a}
 \quad\hbox{satisfies}\quad D'(a)>0.
\end{equation}

Choose $a_0=37569/50000$. The rational bounds
\begin{equation}\label{s:pi-bounds}
 \frac{314159}{100000}<\pi<\frac{3927}{1250}
\end{equation}
give
\begin{equation}\label{eq:body-041}
 \frac{314159}{100000}(1-a_0)-H_0^2a_0^2
 =\frac{444658362479}{62500000000000000}>0,
 \qquad \frac{\sqrt{1-a_0}}{\sqrt\pi a_0}>\frac{374396}{10^6}.
\end{equation}
Hence $n(H_0)>a_0$. To verify \eqref{s:pi-bounds}, let
$A_j(x)=\sum_{k=0}^j(-1)^kx^{2k+1}/(2k+1)$.
Machin's formula and
$A_5(1/5)<\arctan(1/5)<A_5(1/5)+1/(13\cdot5^{13})$,
$A_1(1/239)<\arctan(1/239)<A_1(1/239)+1/(5\cdot239^5)$
give the rational comparisons
\begin{align*}
 \frac{314159}{100000}
 &<16A_5(1/5)-4A_1(1/239)-\frac4{5\cdot239^5}<\pi,\\
 \pi
 &<16A_5(1/5)+\frac{16}{13\cdot5^{13}}-4A_1(1/239)<\frac{3927}{1250}.
\end{align*}

The table \eqref{eq:body-042} gives integers providing $k_j(a_0)\ge L_j/10^6$ and $w_jL_j/10^6\ge T_j/10^9$:
\begin{equation}\label{eq:body-042}
\begin{array}{c|r|r@{\qquad}c|r|r}
 j&L_j&T_j&j&L_j&T_j\\\hline
 1&149611&75782841&7&22192&57634\\
 2&95962&8094730&8&18322&30298\\
 3&65356&1967331&9&15440&16784\\
 4&46951&653777&10&13232&9695\\
 5&35300&260454&11&11499&5795\\
 6&27561&117223&&&
\end{array}
\end{equation}
These entries are finite rational comparisons. For the first family, induction in \eqref{s:low-k} reduces every row $j\ge2$ precisely to
\begin{equation}\label{eq:body-043}
 L_j\,37569\,62431\,10^6\le37569^2\,10^6L_{j-1}
 +621550000\sum_{i=1}^{j-1}L_iL_{j-i}.
\end{equation}
For each $j$, the comparison $w_jL_j/10^6\ge T_j/10^9$ is
\[
 2000\cdot73^{2j+1}L_j\ge (2j-1)(2j+1)80^{2j+1}T_j.
\]
Using $2a_0R-R^2=2154449/4000000$ and summing the $T_j/10^9$ gives
\begin{equation}\label{eq:body-044}
 B_{11}(a_0)\ge\frac{625608812}{10^9},\qquad
 D(a_0)>\frac{4812}{10^9}>0.
\end{equation}
Equations \eqref{s:low-D} and \eqref{eq:body-044} prove the desired inequality, since $h/\pi=\sqrt{1-n(h)}/(\sqrt\pi n(h))$.

For $X=|\nu_z|$ with probability measure $d\sigma/P$,
\eqref{eq:body-031} and $v\ge1/\pi$ give
\begin{equation}\label{eq:body-045}
 \alpha=N_z/P\ge n(h)/P,\qquad c=c_0/P,
 \qquad P\mu=P-hv\le P-h/\pi.
\end{equation}
For fixed $c$,
\begin{equation}\label{eq:body-046}
 U_\alpha=\frac{(\alpha+c)\sqrt{1-t^2}}
 {\sqrt{\alpha^2+2\alpha c+c^2(1-t^2)}}\in[0,1].
\end{equation}
Define
\begin{equation}\label{s:low-positive-functional}
 T_R(\alpha,c)=2\alpha R+
 \int_0^R(R^2/t^2-1)(\alpha-U(\alpha,c;t))\,dt.
\end{equation}
It is homogeneous of degree one and strictly increasing in $\alpha$
for fixed $c$, because $0\leq U_\alpha\leq1$ and the kernel is nonnegative.
By homogeneity and monotonicity,
$T_R(\alpha,c)\ge T_R(n(h)/P,c_0/P)=T_R(n(h),c_0)/P$.
Equation \eqref{s:moment-nonnegative-decomposition} gives
\begin{equation}\label{s:low-scalar-comparison}
 \mu+R^2\geq T_R(\alpha,c)
 \geq P^{-1}T_R(n(h),c_0)
 \geq P^{-1}[B_{11}(n(h))+R^2].
\end{equation}
Multiplication by $P$ gives
\begin{equation}\label{eq:body-047}
 P\mu\ge B_{11}(n(h))+R^2(1-P)
 >1-h/\pi+R^2(1-P)\ge P-h/\pi,
\end{equation}
which is the required contradiction.
\end{proof}

\subsection{Exclusion of genus three}\label{s:genus-three}

\begin{proposition}\label{s:n3}
The necessary genus-three domain of Lemma~\ref{g:genus-three} is empty. Explicitly, no triple satisfies
\begin{equation}\label{eq:body-048}
 0<u\le1/\sqrt3,\quad0\le k\le\sqrt\pi,\quad0\le\ell\le1-u^2,
\end{equation}
\begin{equation}\label{eq:body-049}
 \ell\le1-u^2-k/\pi+(k/\pi)\ell^2,
 \qquad (\ell+ku)^2+\ell^2+\pi u\ge\pi.
\end{equation}
\end{proposition}
\begin{proof}
Put $p=22/7$, $r=157/50$, $a=1-u^2$, and $t=k/p$. Archimedes' bounds give $r<\pi<p$, while
\begin{equation}\label{eq:body-050}
 u\le1/\sqrt3<29/50=:U,\qquad
 t<\sqrt{7/22}<4/7=:T.
\end{equation}
Since $\ell\le a\le1$ and $k/\pi\ge t$,
\eqref{eq:body-049} gives $\ell\le a-(k/\pi)(1-\ell^2)\le a-t+t\ell^2$. For $0\le u\le29/50$ and
$0\le t\le4/7$, set $A=a-t$, $B=1-ta$ and $D=B-tA$.
Then $A\ge1-U^2-T>0$, $B\ge1-t\ge3/7$ and
$D=(1-t)^2+2t(1-a)\ge9/49$.
Since $\ell^2\le a\ell$, we have $B\ell\le A$, and therefore
\begin{equation}\label{s:n3-rational-relaxation}
 B\ell\leq AB+tB\ell^2\leq AB+tA\ell,\qquad
 D\ell\leq AB,\qquad \ell\leq L:=AB/D.
\end{equation}
Using $0\le\ell\le L$ from \eqref{s:n3-rational-relaxation} and $\pi<p$ gives
\begin{equation}\label{eq:body-051}
 (\ell+ku)^2+\ell^2+\pi u\le 2L^2+2ptuL+p^2t^2u^2+pu.
\end{equation}
It is therefore enough to prove positivity of
\begin{equation}\label{s:n3-H}
 H(u,t)=(r-pu-p^2t^2u^2)D^2-2A^2B^2-2ptuABD.
\end{equation}
Indeed the right side of \eqref{eq:body-051} is $r-H/D^2<r<\pi$.

We verify positivity by the Bernstein coefficients of this polynomial. Let $B_i^n(x)=\binom ni x^i(1-x)^{n-i}$. With $x=u/U$, $y=t/T$,
\begin{equation}\label{s:n3-B}
 H(Ux,Ty)=\sum_{i=0}^{8}\sum_{j=0}^{6}\beta_{ij}B_i^8(x)B_j^6(y).
\end{equation}
The power coefficients of \eqref{s:n3-H} are listed in \eqref{eq:body-052}:
\begin{equation}\label{eq:body-052}
C=\begin{pmatrix}
2793&-11172&16758&-11172&2793&0&0\\
-7700&15400&15400&-61600&53900&-15400&0\\
9800&-8428&-26944&88372&-135400&96800&-24200\\
0&-15400&-30800&123200&-92400&15400&0\\
-4900&29400&-18228&-67400&188700&-96800&0\\
0&0&15400&-92400&46200&0&0\\
0&-9800&19600&-9800&-96800&0&0\\
0&0&0&30800&0&0&0\\
0&0&-4900&0&0&0&0
\end{pmatrix}.
\end{equation}
Expansion of \eqref{s:n3-H} gives $2450H(u,t)=\sum_{r=0}^8\sum_{s=0}^6C_{rs}u^rt^s$. The basis identities
\begin{equation}\label{eq:body-053}
 x^r=\sum_{i=r}^8\frac{\binom ir}{\binom8r}B_i^8(x),\qquad
 y^s=\sum_{j=s}^6\frac{\binom js}{\binom6s}B_j^6(y)
\end{equation}
give the exact rational formula
\begin{equation}\label{s:n3-coeff}
 \beta_{ij}=\frac1{2450}\sum_{r=0}^i\sum_{s=0}^j
 C_{rs}\left(\frac{29}{50}\right)^r\left(\frac47\right)^s
 \frac{\binom ir\binom js}{\binom8r\binom6s}.
\end{equation}
The following integer matrix satisfies $10000\beta_{ij}\ge d_{ij}$:
\begin{equation}\label{eq:body-054}
d=\begin{pmatrix}
11400&7057&4203&2412&1341&726&384\\
9121&5212&2891&1563&827&429&219\\
7323&3809&1953&1031&578&346&216\\
6006&2826&1337&741&509&394&306\\
5136&2228&993&628&555&523&472\\
4651&1968&876&642&660&697&711\\
4452&1977&938&733&768&845&938\\
4410&2159&1112&843&797&832&888\\
4364&2371&1289&869&608&394&6
\end{pmatrix}.
\end{equation}
Since $\binom8r\mid280$ and $\binom6s\mid60$, the comparison
$10^4\beta_{ij}\ge d_{ij}$ in \eqref{s:n3-coeff} is equivalent to
\begin{multline*}
 10^4\sum_{r=0}^i\sum_{s=0}^j
 C_{rs}29^r4^s50^{8-r}7^{6-s}
 \frac{280}{\binom8r}\frac{60}{\binom6s}\binom ir\binom js\\
 \ge d_{ij}\,2450\cdot50^8\cdot7^6\cdot280\cdot60.
\end{multline*}
In particular,
\begin{equation}\label{eq:body-055}
 \beta_{8,6}=H(U,T)=\frac{1157220221572159}{1759277648925781250}>\frac6{10000}.
\end{equation}
For $(u,t)\in[0,29/50]\times[0,4/7]$, the products
$B_i^8(u/U)B_j^6(t/T)$ in \eqref{s:n3-B} are nonnegative and
sum to one. Each $\beta_{ij}\ge d_{ij}/10^4\ge3/5000$, so
$H(u,t)\ge3/5000$, which contradicts \eqref{eq:body-049}
by \eqref{eq:body-051}.
\end{proof}
\section{Global support bounds and parameter comparisons}\label{sec:parameter-comparisons}

The geometric applications in this section and the genus-two estimates
in Section~\ref{sec:range-exclusion} concern a downward monotone
isoperimetric region $E\subset Q$. Its reflected boundary
$\Sigma=\partial\widehat E\subset T$ is connected, and
\begin{equation}\label{s:geometric-hypotheses}
 \begin{gathered}
 |E|=b=1/\pi,\qquad P=\Per(E;Q)\le1,\qquad h>0,\qquad g(\Sigma)=2,\\
 A_{x,1}>0,\qquad A_{y,1}>0,\qquad A_{z,1}=0.
 \end{gathered}
\end{equation}

Lemma~\ref{s:support} gives affine bounds for the two extremal
values in \eqref{s:finite-value-functions}. At $v=1/\pi$, $P\le1$
and $h\ge H_0$, Proposition~\ref{s:body-quarter} gives $\delta<1/4$.
The inequality \eqref{g:raw-dual-eq} supplies the scalar condition
$\mathcal R(h,P,\delta)\le0$, with $\mathcal R$ as in
\eqref{s:residual-definition}. Proposition~\ref{s:perimeter} compares
$\mathcal R(h,P,\delta)$ with $\mathcal R(h,1,\delta)$ under
\eqref{s:perimeter-domain}. Lemma~\ref{s:concavity} reduces positivity
of a fixed lower bound for $\mathcal R(h,1,\delta)$ to its endpoint values in $\delta$.

\subsection{Rational support planes for moment and cost bounds}
\label{s:support-section}

\begin{lemma}\label{s:cap-concavity}
For every integer $N\geq1$, $\alpha>0$, $c\geq0$ and $0\leq t\leq1$,
$U_N(\alpha,c;t)$ is jointly concave in $(\alpha,c)$ and satisfies
$0<(U_N)_\alpha\leq1$. It is positive and strictly decreases in $t$.
\end{lemma}
\begin{proof}
Put $A=\alpha+c$, $d=c/A<1$ and $s=t^2$.
For $s<1$, direct differentiation gives
$-U_{\alpha\alpha}=d^2sf(s)/A$ with
$f(s)=\sqrt{1-s}/(1-d^2s)^{3/2}$.
Write
$q(s)=\sqrt{(1-s)/(1-d^2s)}=1-\sum_{j\geq1}q_js^j$.
The binomial series for
\[
 -q'(s)=\frac{1-d^2}{2}(1-s)^{-1/2}(1-d^2s)^{-3/2}
\]
has positive coefficients, so $q_j>0$ for $j\ge1$.
As $s\uparrow1$, $\sum_{j\ge1}q_js^j\uparrow\sum_{j\ge1}q_j=1$.

Let $b_j$ be the coefficients of $f=q/(1-d^2s)$, and set
$P_m=\sum_{j=0}^mb_js^j$ and $Q_m=1-\sum_{j=1}^mq_js^j$.
Their exact finite identity is
\begin{equation}\label{s:finite-concavity-identity}
 b_j=d^2b_{j-1}-q_j,\qquad
 (1-d^2s)P_m=Q_m-d^2b_ms^{m+1}.
\end{equation}
For $d>0$, $b_j/d^{2j}$ decreases.
If $b_m\geq0$, all preceding $b_j$ are nonnegative and $P_m\geq1$.
If $b_m<0$, the right side of \eqref{s:finite-concavity-identity} is positive
because $Q_m>0$, also at $s=1$.
For $d=0$, $P_m=Q_m>0$. Thus $P_m>0$ on the whole closed interval.
Truncating the power series of
$-U_{\alpha\alpha}=d^2sf(s)/A$ at degree $N$ gives
$-(U_N)_{\alpha\alpha}=d^2sP_{N-1}/A$.
Since $U_N$ is homogeneous of degree one in $(\alpha,c)$,
its Hessian annihilates $(\alpha,c)^{\mathsf T}$. For $c>0$ this
determines
\begin{equation}\label{s:finite-envelope-hessian}
 \Hess_{\alpha,c}U_N
 =-\frac{sP_{N-1}}{A^3}
 \begin{pmatrix}c^2&-\alpha c\\-\alpha c&\alpha^2\end{pmatrix}\leq0.
\end{equation}
The recurrence \eqref{s:finite-cap-recurrence} is smooth when
$\alpha+c>0$, so \eqref{s:finite-envelope-hessian} holds at $c=0$
by continuity. Its matrix annihilates $(\alpha,c)^{\mathsf T}$.
The derivative $(U_N)_\alpha=Q_N$ satisfies
\begin{equation}\label{s:finite-alpha-bound}
 0<(U_N)_\alpha\le1,\qquad
 (U_N)_\alpha\ge q(s)\ge\sqrt{1-t^2}.
\end{equation}
For $\alpha>0$, the recurrence has $k_j>0$.
Lemma~\ref{s:cap-polynomial} at $N+1$ gives
$U_N=U_{N+1}+k_{N+1}t^{2N+2}>0$ for $t>0$, while $U_N(0)=\alpha$.
Finally $(U_N)_t=-\sum_{j=1}^N2jk_jt^{2j-1}<0$ for $t>0$,
which proves strict decrease, despite its zero derivative at $t=0$.
\end{proof}

Fix $N\ge1$ and an even polynomial $\phi$ for which the kernel
$w_\phi$ in \eqref{g:cost} is nondecreasing. With $\alpha_\eta,a_\eta$
as in Lemma~\ref{g:canonical}, set
\[
 \mathcal P^N_{\mu,c}=
 \left\{\eta\in\mathcal P([0,1]):\int x^2\,d\eta=\mu,\quad
 a_\eta(t)\le U_N(\alpha_\eta,c;t)\quad(0\le t\le1)\right\},
 \qquad c\ge0.
\]
Use $U_N(0,0;t)=0$, its continuous extension since $0\le U_N\le\alpha$.
For $\Omega_N=\{(\mu,c):c\ge0,\ \mathcal P^N_{\mu,c}\ne\varnothing\}$,
define
\begin{equation}\label{s:finite-value-functions}
 A_N(\mu,c)=\max_{\eta\in\mathcal P^N_{\mu,c}}\alpha_\eta,\qquad
 C_N(\mu,c)=\min_{\eta\in\mathcal P^N_{\mu,c}}\int\phi(x)\,d\eta(x).
\end{equation}
The admissible sets are weakly compact. If $\eta_j\rightharpoonup\eta$
in $\mathcal P([0,1])$ with $\eta_j\in\mathcal P^N_{\mu,c}$, then
$\int x^2d\eta=\mu$. Since $x\mathbf1_{\{x>t\}}$ is lower semicontinuous,
\[
 a_\eta(t)\le\liminf_j a_{\eta_j}(t)
 \le\lim_j U_N(\alpha_{\eta_j},c;t)=U_N(\alpha_\eta,c;t)
 \qquad(0\le t\le1).
\]
For $\alpha,c>0$, $0<R<1$, the measure
\[
 -\mathbf 1_{(0,R)}(t)\frac{\partial_t U_N(\alpha,c;t)}{t}\,dt
 +\frac{U_N(\alpha,c;R)}{R}\delta_R(dt)
\]
has tail $U_N(\alpha,c;t)\mathbf 1_{[0,R)}(t)$ by direct integration.
Its mass and second moment, by \eqref{g:mass}, are
\begin{equation}\label{s:finite-canonical-mass}
 M_N=\frac\alpha R+\sum_{j=1}^N\frac{k_jR^{2j-1}}{2j-1},\qquad
 \mu_N=\alpha R-\sum_{j=1}^N\frac{k_jR^{2j+1}}{2j+1}.
\end{equation}
Given $r>0$ and $0<R<1$, let $\kappa_j=k_j(1,r)$ and set
\begin{equation}\label{s:reference-point}
 \alpha_0=\left[\frac1R+\sum_{j=1}^N\frac{\kappa_jR^{2j-1}}{2j-1}\right]^{-1},
 \quad c_0=r\alpha_0,
 \quad\mu_0=\alpha_0\left[R-\sum_{j=1}^N\frac{\kappa_jR^{2j+1}}{2j+1}\right].
\end{equation}
Thus $M_N(\alpha_0,c_0,R)=1$ and $\mu_N(\alpha_0,c_0,R)=\mu_0$.

\begin{lemma}\label{s:support}
At the point \eqref{s:reference-point}, there are affine functions
such that for every $(\mu,c)\in\Omega_N$,
\begin{align}
 A_N(\mu,c)&\le a_A\mu+b_Ac+d_A, \label{s:mean-upper-support}\\
 C_N(\mu,c)&\ge a_\phi\mu+b_\phi c+d_\phi. \label{s:cost-lower-support}
\end{align}
Both inequalities are equalities at $(\mu_0,c_0)$. If $r,R\in\mathbb Q$,
then $(a_A,b_A,d_A)\in\mathbb Q^3$. If also $\phi\in\mathbb Q[x]$, then
$(a_\phi,b_\phi,d_\phi)\in\mathbb Q^3$.
\end{lemma}
\begin{proof}
The equation $M_N=1$ determines $R$ locally as a smooth function of $(\alpha,c)$. The derivatives are $(M_N)_R=-U_N(R)/R^2<0$ and
$(M_N)_\alpha>0$ by \eqref{s:finite-alpha-bound} and \eqref{g:mass}.
At fixed $c$, $\partial_\alpha R=R^2(M_N)_\alpha/U_N(R)>0$.
\begin{equation}\label{s:finite-canonical-derivative}
 \frac{\dd\mu_N}{\dd\alpha}
 =2R+\sum_{j=1}^N\frac{2k_{j,\alpha}R^{2j+1}}{(2j-1)(2j+1)}>0.
\end{equation}
Let $b(t)=U_N(\alpha_*,c;t)\mathbf1_{[0,R)}(t)$ with
$M_N(\alpha_*,c,R)=1$ and $\mu_N(\alpha_*,c,R)=\mu$. For
$\eta\in\mathcal P^N_{\mu,c}$, write $a=a_\eta$ and
$\alpha=\alpha_\eta$. Equation \eqref{s:moment-nonnegative-decomposition}
and $a\leq U_N$ give the finite moment lower bound
$B_R(\alpha,c)=2\alpha R-R^2+
\sum_{j=1}^N2k_jR^{2j+1}/[(2j-1)(2j+1)]\leq\mu$. By Lemma~\ref{s:cap-concavity},
\[
 \partial_\alpha B_R
 =2R+\int_0^R(R^2/t^2-1)[1-(U_N)_\alpha]\,\dd t\ge2R.
\]
Since $B_R(\alpha_*,c)=\mu$ by \eqref{s:finite-canonical-mass},
strict monotonicity gives $\alpha\le\alpha_*$. Therefore
$a(t)\le b(t)$ for $t<R$ by Lemma~\ref{s:cap-concavity} and
$a(t)\ge b(t)=0$ for $t\ge R$. With $\int_0^1(a-b)=0$ and
$(w(t)-w(R))(a(t)-b(t))\ge0$,
\begin{equation}\label{eq:body-059}
 \int_0^1w(t)[a(t)-b(t)]\,\dd t
 \ge w(R)\int_0^1[a(t)-b(t)]\,\dd t=0.
\end{equation}
By \eqref{g:cost} and $\phi'(0)=0$, \eqref{eq:body-059} proves
optimality of $b$ for $C_N$ as well as for $A_N$. Thus both
values are smooth near \eqref{s:reference-point} by
\eqref{s:finite-canonical-derivative}.

If $\eta_j\in\mathcal P^N_{\mu_j,c_j}$ for $j=1,2$ and
$0\le\theta\le1$, let $\eta_\theta=(1-\theta)\eta_1+\theta\eta_2$,
$c_\theta=(1-\theta)c_1+\theta c_2$ and
$\mu_\theta=(1-\theta)\mu_1+\theta\mu_2$. Lemma~\ref{s:cap-concavity}
gives
\[
 a_{\eta_\theta}(t)\le(1-\theta)U_N(\alpha_{\eta_1},c_1;t)
   +\theta U_N(\alpha_{\eta_2},c_2;t)
 \le U_N(\alpha_{\eta_\theta},c_\theta;t).
\]
The estimate extends to $\alpha=0$ by continuity. Thus
$\eta_\theta\in\mathcal P^N_{\mu_\theta,c_\theta}$. Applying this
to the respective optimizers in \eqref{s:finite-value-functions}
proves joint concavity of $A_N$ and joint convexity of $C_N$.
Their tangent affine functions at $(\mu_0,c_0)$ therefore satisfy
\eqref{s:mean-upper-support} and \eqref{s:cost-lower-support}
on $\Omega_N$.

At $(\alpha_0,c_0,R)$ in \eqref{s:reference-point}, set
\begin{equation}\label{eq:body-062}
\begin{aligned}
 w_j&=\frac{2R^{2j+1}}{(2j-1)(2j+1)},&
 B_\alpha&=2R+\sum_jw_jk_{j,\alpha},&
 B_c&=\sum_jw_jk_{j,c}.
\end{aligned}
\end{equation}
Differentiate $M_N=1$ and $\mu=\mu_N$, using
$(M_N)_R=-U_N(R)/R^2$ and $(\mu_N)_R=U_N(R)$. Then
\[
 \dd\mu=[(\mu_N)_\alpha+R^2(M_N)_\alpha]\dd\alpha+
        [(\mu_N)_c+R^2(M_N)_c]\dd c
       =B_\alpha\dd\alpha+B_c\dd c.
\]
Thus
\begin{equation}\label{s:mean-plane}
 a_A=B_\alpha^{-1},\qquad b_A=-B_c/B_\alpha,
 \qquad d_A=\alpha_0-a_A\mu_0-b_Ac_0.
\end{equation}
For $C=\phi(0)+\int_0^RwU_N\,\dd t$, set
\begin{equation}\label{s:cost-differentials}
 C_\alpha^*=\int_0^Rw(U_N)_\alpha\,\dd t+R^2w(R)(M_N)_\alpha,
 \quad C_c^*=\int_0^Rw(U_N)_c\,\dd t+R^2w(R)(M_N)_c.
\end{equation}
Then its exact support coefficients are
\begin{equation}\label{s:cost-plane}
 a_\phi=C_\alpha^*/B_\alpha,\quad
 b_\phi=C_c^*-a_\phi B_c,\quad
 d_\phi=C-a_\phi\mu_0-b_\phi c_0.
\end{equation}
Here $w(0)$ is its continuous polynomial extension and all derivatives
are evaluated at $(\alpha_0,c_0,R)$, with $R$ fixed unless eliminated
by $M_N=1$. The sums run from $1$ to $N$.
Every integral is a finite polynomial integral.
The recurrence makes its coefficients rational for rational $r,R$
when $\phi\in\mathbb Q[x]$; arbitrary real polynomial coefficients
give supports but need not give rational coefficients. Homogeneity gives the useful exact formula
\begin{equation}\label{s:cost-intercept}
 d_\phi=\phi(0)+R^2[a_\phi-w(R)].
\end{equation}
To derive it, Euler's identity gives
\begin{equation}\label{eq:body-064}
 \alpha_0 C_\alpha^*+c_0C_c^*=C-\phi(0)+R^2w(R),
 \qquad\alpha_0B_\alpha+c_0B_c=\mu_0+R^2;
\end{equation}
substitution into \eqref{s:cost-plane} gives \eqref{s:cost-intercept}.
\end{proof}

The three costs used below are
\[
 \ell_{16}(x)=\sum_{j=1}^{16}\frac{x^{2j}}{2j-1},\qquad
 v_{16}(x)=\sum_{j=1}^{16}\frac{\binom{2j}{j}}{4^j(2j-1)}x^{2j},\qquad
 \phi_8(x)=2\sum_{j=1}^{8}\frac{((1-x^2)/2)^j}{2j-1}.
\]
They satisfy $\ell_{16}(x)\le x\operatorname{atanh}x$ and $v_{16}(x)\le1-\sqrt{1-x^2}$. Their tail kernels are increasing. For the first two the kernels are respectively $\sum_{j=1}^{16}t^{2j-2}$ and $\sum_{j=1}^{16}\binom{2j}{j}t^{2j-2}/4^j$. For $\phi_8$ this is \eqref{g:finite-cost-monotonicity}.
These three costs have rational coefficients, so the rationality
assertion of Lemma~\ref{s:support} applies to all of them.

\subsection{Projection-area bounds on horizontal asymmetry}
\label{s:body-section}

Here $v=1/\pi$ and we use \eqref{g:genus-two-parameters},
\eqref{g:wall-stress}, \eqref{g:source} and \eqref{g:pressure-radicands}.
The inequalities $m_x\ge0$ and $J_y\ge F(R_y)\ge0$ in
\eqref{g:pressure-eq} yield $\delta\le q$ and $\delta\le2-3q$.
In particular $q\le2/3$. Combining the resulting bound on $P$ with
$h^2P\le\pi$ from Lemma~\ref{g:hersch} gives
\begin{equation}\label{eq:body-069}
 \delta\le\min\{q,2-3q\},\qquad P\ge\frac{3h}{2\pi},
 \qquad h^3\le\frac{2\pi^2}{3}<(15/8)^3.
\end{equation}
Proposition~\ref{s:q-bound} also gives $q<7/12$.

\begin{proposition}\label{s:body-quarter}
Let $E$ satisfy the geometric hypotheses \eqref{s:geometric-hypotheses}.
If $h\ge H_0=5881/5000$, then $\delta<1/4$.
\end{proposition}
\begin{proof}
Take $N=3$, $r=17/16$, and $R=5/6$ in \eqref{s:mean-plane}. The exact coefficients are
\begin{equation}\label{eq:body-070}
 a=\frac{79084332549684}{143822139819265},\quad
 b=\frac{737452745800}{28764427963853},\quad
 d=\frac{10983935076345}{28764427963853}.
\end{equation}
Rounding each upward gives the valid bound
\begin{equation}\label{eq:body-071}
 \alpha\le\bar a\mu+\bar bc+\bar d,
 \qquad (\bar a,\bar b,\bar d)=\frac1{10^6}(549876,25638,381859).
\end{equation}
Consequently
\begin{equation}\label{eq:body-072}
 N_z\le\bar a(P-h/\pi)+\bar b\pi/(2h^2)+\bar dP.
\end{equation}
The right side of \eqref{eq:body-072} increases with $P$
and decreases with $h$. Hence, with $U=3927/1250$,
\begin{equation}\label{eq:body-073}
 N_z^U\le\bar a(1-H_0/U)+\bar bU/(2H_0^2)+\bar d<19/25<1.
\end{equation}
Use this $N_z^U$ in \eqref{g:body-gates} to obtain
\begin{equation}\label{eq:body-074}
\begin{aligned}
 \rho(P)&=1+\bar b+\frac{2h^2}{\pi}
 \left[(\bar a+\bar d)P-\frac{\bar ah}{\pi}\right],\\
 e(P)&=\rho(P)^2-1-4h^2/\pi,\\
 r(P)&=(1+h-\rho(P)-\sqrt{e(P)})_+.
\end{aligned}
\end{equation}
For the minimizer $E$, \eqref{g:body-gates} gives $e(P)\ge0$.
Increasing $P$ increases $\rho,e$ and decreases $r$. Since $F$ increases,
\begin{equation}\label{s:body-P}
 \delta\le2-\frac{3h/\pi+2F(r(P))/h}{P}
 \le2-\frac{3h}\pi-\frac{2F(r(1))}h.
\end{equation}

Since $q=h/(\pi P)<7/12$ and $P\leq1$, only
$H_0\leq h<7\pi/12$ is possible.
Set
\[
 K=1027/1000,\qquad x=h(7/8-3h/(2\pi)).
\]
It suffices to show $F(r(1))>x$. Squaring gives $\sqrt{1-s}\ge1-s/2-s^2/2$ for $0\le s\le1$, hence
\begin{equation}\label{eq:body-076}
 F(r)\ge r-r^3/6-r^5/10.
\end{equation}
Use $L=314159/100000<\pi<U=3927/1250$ and define the following rational polynomials:
\begin{equation}\label{eq:body-077}
\begin{aligned}
 x_U(h)&=h(7/8-3h/(2U)),\\
 \rho_L(h)&=1+\bar b+2(\bar a+\bar d)h^2/L-2\bar ah^3/L^2,\\
 u(h)&=1+h-Kx_U(h),\\
 D(h)&=u(h)^2+1+4h^2/U-2u(h)\rho_L(h).
\end{aligned}
\end{equation}
On $[H_0,7U/12]$, $x_U$ decreases and $Kx_U<2/5$. The finite rational endpoint comparison is
\begin{equation}\label{eq:body-078}
 (K-1)-K^3x_U(H_0)^2/6-K^5x_U(H_0)^4/10>1/3000.
\end{equation}
Thus, on the actual interval, $F(Kx_U)>x_U\ge x$. Also $\rho(1)\le\rho_L$, since $\rho(1)$ decreases with $\pi$; its derivative has this sign because $2\bar ah<(\bar a+\bar d)L$ on the interval.

Here are finite coefficient comparisons proving the remaining polynomial signs. Expand $u-\rho_L=\sum_{j=0}^3q_j(h-H_0)^j$. Its coefficients satisfy
\begin{equation}\label{eq:body-079}
 (q_0,q_1,q_2,q_3)>(13/100,32/100,29/100,11/100).
\end{equation}
Therefore $u-\rho_L>0$. Write $D'''(h)=\gamma_0+\gamma_1(h-H_0)+\gamma_2(h-H_0)^2$.
The exact rational coefficient bounds are
\begin{equation}\label{s:body-third-derivative}
 (\gamma_0,\gamma_1,\gamma_2)>(1,7,6),\qquad
 D'(H_0),D'(7/5)<-1/100,\qquad D(7/5)>1/2500.
\end{equation}
For $h=(1-\theta)H_0+\theta(7/5)$, $0\le\theta\le1$,
\eqref{s:body-third-derivative} gives
$D'(h)\le(1-\theta)D'(H_0)+\theta D'(7/5)<0$.
Thus $D(h)\ge D(7/5)>0$ on $[H_0,7/5]$.

By \eqref{s:body-third-derivative}, $D'''(h)>0$ for $h\geq H_0$.
The remaining rational comparisons are $D''(7/5)>1/3$
and, at $h_*=36/25$,
\begin{equation}\label{eq:body-080}
 D(h_*)>1/5000,\qquad |D'(h_*)|<1/400.
\end{equation}
Since $D''\ge1/3$, Taylor's formula with integral remainder gives
\begin{equation}\label{eq:body-081}
 D(h)\ge D(h_*)+D'(h_*)(h-h_*)+(h-h_*)^2/6
 >\frac1{5000}-\frac3{2\cdot400^2}>0.
\end{equation}
We have proved $u-\rho_L>0$ and $D>0$ everywhere needed. Consequently
\begin{equation}\label{eq:body-082}
 [u-\rho(1)]^2-e(1)
 =u^2-2u\rho(1)+1+4h^2/\pi\ge D>0,
\end{equation}
so $r(1)>Kx_U$. Hence $F(r(1))>x$, and \eqref{s:body-P} yields $\delta<1/4$.
\end{proof}

\subsection{Monotonicity in perimeter and concavity in horizontal asymmetry}\label{s:concavity-section}
\label{s:perimeter-section}

For fixed $h,\delta,t_x,t_y,\eta$ and affine coefficients,
Proposition~\ref{s:perimeter} proves
$\mathcal R(h,P_1,\delta)\ge\mathcal R(h,P_2,\delta)$
for the function in \eqref{s:residual-definition}, provided the
conditions below hold throughout $[P_1,P_2]$.

\begin{lemma}\label{s:support-slopes}
Let $(\mu_0,c_0)$ be as in \eqref{s:reference-point} with $R=R_0$.
For the coefficients in \eqref{s:mean-plane},
\begin{equation}\label{eq:body-083}
 0<a_A\leq\frac1{2\sqrt{\mu_0}},\qquad d_A=R_0^2a_A\geq0.
\end{equation}
If $\mu_0\leq5/12$, then $R_0<4/5<\sqrt3/2$ and
\begin{equation}\label{eq:body-084}
 a_A>1/2,\qquad 1\leq a_L<4,\qquad d_L\leq0,
\end{equation}
where $(a_L,b_L,d_L)$ is \eqref{s:cost-plane} for $\phi=\ell_{16}$.
For $(a_T,b_T,d_T)$ given by \eqref{s:cost-plane} with $\phi=\phi_8$,
one has $a_T\le0$ and $d_T\le\phi_8(0)<5/4$.
\end{lemma}
\begin{proof}
All quantities in this proof are evaluated at the reference point.
Eliminating the cutoff differential gives
\begin{equation}\label{s:mean-reference-slope}
 B_\alpha=\int_0^{R_0}(U_N)_\alpha\,dt+R_0^2M_\alpha
 =2R_0+\int_0^{R_0}(R_0^2/t^2-1)[1-(U_N)_\alpha]\,dt\geq2R_0.
\end{equation}
Since $\mu_0\leq R_0^2$, $a_A=1/B_\alpha\leq1/(2\sqrt{\mu_0})$.
Euler's identity $\alpha_0B_\alpha+c_0B_c=\mu_0+R_0^2$
gives $d_A=R_0^2a_A$.

The bound $U_N\ge U\ge\alpha_0\sqrt{1-t^2}$ gives
$\mu_0=\int_0^{R_0}U_N\,dt\ge\alpha_0F(R_0)$ and
\[
 1=\frac{\alpha_0}{R_0}+\int_0^{R_0}\frac{\alpha_0-U_N}{t^2}\,dt
 \le\alpha_0\bigl(\arcsin R_0+\sqrt{1-R_0^2}/R_0\bigr).
\]
Consequently
\begin{equation}\label{s:reference-cutoff-bound}
 \mu_0\geq H(R_0):=
 \frac{\arcsin R_0+R_0\sqrt{1-R_0^2}}
 {2[\arcsin R_0+\sqrt{1-R_0^2}/R_0]}.
\end{equation}
For $I(R)=(\arcsin R+R\sqrt{1-R^2})/2$ and
$m(R)=\arcsin R+\sqrt{1-R^2}/R$,
$I'=\sqrt{1-R^2}$ and $m'=-\sqrt{1-R^2}/R^2$.
Thus $H=I/m$ increases.
At $R=4/5$,
$H(R)-5/12=[\arcsin(4/5)-87/100]/[12(\arcsin(4/5)+3/4)]>0$,
since $\arcsin(4/5)>4/5+(4/5)^3/6=332/375>87/100$.
Therefore $\mu_0\leq5/12$ implies $R_0<4/5$.

Using \eqref{s:finite-alpha-bound} in \eqref{s:mean-reference-slope} gives
$B_\alpha\leq g(R_0)$, where
$g(R)=(R^2+1/2)\arcsin R+(3/2)R\sqrt{1-R^2}$ and
$g'=2(R\arcsin R+\sqrt{1-R^2})>0$.
As $\sin1>5/6>4/5$, $\arcsin(4/5)<1$ and $g(4/5)<93/50<2$.
Hence $a_A>1/2$.

For any increasing cost kernel,
\begin{equation}\label{s:cost-reference-average}
 a_\phi=
 \frac{\int_0^{R_0}w(U_N)_\alpha\,dt+w(R_0)R_0^2M_\alpha}
 {\int_0^{R_0}(U_N)_\alpha\,dt+R_0^2M_\alpha}.
\end{equation}
The weights in \eqref{s:cost-reference-average} are
$(U_N)_\alpha\,dt$ and $R_0^2M_\alpha$, both nonnegative.
For $\ell_{16}$, $1\leq w_L(t)\leq(1-t^2)^{-1}$, so
$1\leq a_L\leq w_L(R_0)<25/9<4$.
For $\phi_8$, $w_T\leq0$, so $a_T\leq0$.
The average is at most $w(R_0)$; \eqref{s:cost-intercept} gives
$d_L\leq0$ and $d_T\leq\phi_8(0)$.
Finally
$\phi_8(0)<1+1/6+1/20+1/56+1/72=1573/1260<5/4$.
\end{proof}

Fix $t_x,t_y\in(0,\pi/4)$ and $\eta\in[0,1]$. Use
$s_i,A_i,d_i,s,r_1,r_2$ from Lemma~\ref{g:raw-dual}. Set $T=1-P+3h/(2\pi)-[\eta-F(\eta)]/h$.
For $W\le s_xR_x+s_yR_y$ and $N^U\ge N_x+N_y$, define
\begin{equation}\label{s:residual-definition}
 \mathcal R(h,P,\delta)=C_h(r_1)+C_h(r_2)+sT+W
 +d_x\sigma_x+d_y\sigma_y-A_x-A_y-h^2N^U.
\end{equation}
By \eqref{g:raw-dual-eq}, $\mathcal R(h,P,\delta)\le0$ at the
parameters of a minimizer. For the comparison in $P$ use the domain
\begin{equation}\label{s:perimeter-domain}
 H_0\le h\le15/8,\qquad 0<P\le1,\qquad h^2P\le\pi,\qquad
 0\le\delta<1/4,\qquad \delta\le\min\{q,2-3q\}.
\end{equation}
Write $A_i^{\rm pl}$ for an affine upper bound from
\eqref{s:mean-upper-support}, and $L_i^{\rm pl}$, $V^{\rm pl}$,
$T^{\rm pl}$ for affine lower bounds from \eqref{s:cost-lower-support}
with $\phi=\ell_{16},v_{16},\phi_8$, respectively.
Use reference moments in $[36/625,5/12]$ for $A_x^{\rm pl},A_y^{\rm pl}$
and in $(0,5/12]$ for $L_x^{\rm pl},L_y^{\rm pl}$ whenever they occur.
Set
\begin{align}
 N_A&=P[A_x^{\rm pl}(m_x,c)+A_y^{\rm pl}(m_y,c)],
 \label{eq:body-088}\\
 N_{\rm ang}&=\sqrt2P[1-V^{\rm pl}(\mu,c)-\Gamma_{\mathrm{ang}}],\label{eq:body-089}\\
 N^U&\in\{N_A,N_{\rm ang},\min(N_A,N_{\rm ang})\}.
 \label{s:normal-choices}
\end{align}
Here $\Gamma_{\mathrm{ang}}$ is as in \eqref{eq:body-093}, and, whenever
$N_{\rm ang}$ is used, its affine coefficients satisfy $a_V+d_V<24/25$.
The estimate $N_{\rm ang}\ge N_x+N_y$ follows from
\eqref{g:angular-mean} and $v_{16}(X)\le1-\sqrt{1-X^2}$ for
$X=|\nu_z|$.

Take $W$ to be one such bound, or the maximum of finitely many:
\[
 w R_{\rm tot}(P)+\sum_j b_{x,j}r_{x,j}(P)+\sum_j b_{y,j}r_{y,j}(P),
 \qquad w,b_{i,j}\ge0,\quad w+\sum_j b_{i,j}=s_i\quad(i=x,y).
\]
The coefficients are fixed. Use $R_{\rm tot}$ from
\eqref{eq:body-094} and choose $r_{i,j}$ from \eqref{eq:body-091}
or \eqref{g:active-axis-gates}. For each $r_b$ used from
\eqref{g:active-axis-gates}, the comparison interval must satisfy
$N_z^U=P A_z^{\rm pl}(\mu,c)\in[0,1]$ and $e\ge0$ in
\eqref{g:body-gates}. A constant $r_{i,j}\le r_b(P_2)$ may replace
that $r_b(P)$ on $[P_1,P_2]$.

\begin{proposition}\label{s:perimeter}
Fix $h,\delta,t_x,t_y,\eta$ and the coefficients defining $W,N^U$
in \eqref{s:residual-definition}. If \eqref{s:perimeter-domain} and
the conditions on $W$ and \eqref{s:normal-choices} hold throughout
$[P_1,P_2]$, where $P_1<P_2$, then
\begin{equation}\label{s:perimeter-mean}
 \mathcal R(h,P_1,\delta)\ge\mathcal R(h,P_2,\delta)
 +\frac{2215}{6144}(P_2-P_1).
\end{equation}
\end{proposition}
\begin{proof}
Hold $h,\delta$ fixed. Then $Pc=\pi/(2h^2)$ is constant, while $Pm_x=h/(2\pi)-P\delta/2$ and $Pm_y=h/(2\pi)+P\delta/2$. If $N_A$ is used,
\begin{equation}\label{eq:body-090}
 N_A'=d_{A,x}+d_{A,y}+\frac\delta2(a_{A,y}-a_{A,x})\ge-\frac{25}{96},
\end{equation}
because $d_A\ge0$ and $a_A\le25/12$. The affine lower bounds for $R_x,R_y$ in \eqref{g:source} are
\begin{equation}\label{eq:body-091}
 R_x^{\rm pl}=h+(2h^2/\pi)P[L_y^{\rm pl}(m_y,c)-1],\qquad
 R_y^{\rm pl}=h+(2h^2/\pi)P[L_x^{\rm pl}(m_x,c)-1].
\end{equation}
For $i=x,y$, $(P[L_i^{\rm pl}(m_i,c)-1])'
=d_{L,i}-1\mp\delta a_{L,i}/2\le-1+\delta a_{L,i}/2<-1/2$ by
Lemma~\ref{s:support-slopes}, with the minus sign for $i=x$.
Both functions in \eqref{eq:body-091} are therefore decreasing.

For \eqref{g:D0} and \eqref{g:angular-mean} write
\begin{equation}\label{eq:body-092}
 U=h^2P,\quad Z=U/2+\pi,\quad a_0=2-4q,\quad
 D_s=U(a_0)_+^2/(18Z),\quad D_0=\max\{\delta^2,D_s\},
\end{equation}
\begin{equation}\label{eq:body-093}
 K_*=(32/5)(3/5)^{3/2},\quad B_+=(a_0+5D_0/2)_+,
 \qquad\Gamma_{\mathrm{ang}}=UB_+^2/(9ZK_*).
\end{equation}
Substitute \eqref{g:angular-log} into \eqref{g:source} to obtain
\begin{equation}\label{eq:body-094}
 R_x+R_y\ge R_{\rm tot}(P):=2h+(2h^2/\pi)P[T^{\rm pl}(1-q,c)+D_0/6-2].
\end{equation}
By \eqref{s:perimeter-domain}, $U/Z\le2/3$, $(a_0)_+<51/100$, $\delta^2<1/16$, hence $D_s<1/16$, $D_0\le1/16$, and $B_+<51/100+5/32=533/800<2/3$. For $a_0>0$, $Pa_0'=4q=2-a_0$, and
\begin{equation}\label{eq:body-095}
 (PD_s)'=\frac{U}{18Z}\left(4a_0-\frac{U}{2Z}a_0^2\right)<\frac2{25},
 \qquad PD_s'<\frac2{25}.
\end{equation}
For $a_0\le0$, $D_s=0$ and $(PD_s)'=0$ almost everywhere there.
When $\delta^2\ge D_s$, $D_0=\delta^2$ and $(PD_0)'=\delta^2$
almost everywhere there. On $\delta^2<D_s$, $D_0=D_s$.
These formulas give the almost-everywhere bounds
\begin{equation}\label{eq:body-096}
 (PD_0)'\le1/4,\qquad PD_0'\le2/25,
 \qquad PB_+'\le8/3+1/5=43/15.
\end{equation}
Where $a_0+5D_0/2>0$, differentiation and $K_*>29/10$ give
\begin{equation}\label{s:angular-perimeter-derivative}
\begin{aligned}
 (P\Gamma_{\mathrm{ang}})'&=\frac{U}{9K_*Z}
 \left[\left(2-\frac{U}{2Z}\right)B_+^2+2B_+PB_+'\right]\\
 &<\frac{2}{27(29/10)}
 \left[2(2/3)^2+2(2/3)(43/15)\right]
 =\frac{848}{7047}<\frac18.
\end{aligned}
\end{equation}
On $\{a_0+5D_0/2\le0\}$, $P\Gamma_{\mathrm{ang}}=0$ and $(P\Gamma_{\mathrm{ang}})'=0$
almost everywhere. The functions $PD_0$ and $P\Gamma_{\mathrm{ang}}$ are locally
Lipschitz for $P>0$, so their almost-everywhere derivative bounds
give finite difference bounds by integration.

Now $[PT^{\rm pl}(1-q,c)]'=a_T+d_T<5/4$, and therefore
\begin{equation}\label{eq:body-098}
 [P(T^{\rm pl}+D_0/6-2)]'<5/4+1/24-2<0.
\end{equation}
By \eqref{eq:body-098}, $R_{\rm tot}(P)$ in \eqref{eq:body-094}
is decreasing. For each $N_z^U=P A_z^{\rm pl}(\mu,c)$ used in
\eqref{g:body-gates}, $(N_z^U)'=a_{A,z}+d_{A,z}\ge0$ by
\eqref{eq:body-083}. The same sign holds under coefficientwise
upward rounding. It follows that $\rho_U$ and $e$ are
nondecreasing and $r_b(P)$ is nonincreasing on the stated interval.
Consequently, every combination defining $W$, and hence their
maximum, satisfies $W(P_2)\le W(P_1)$.

For $P\le1$ and $\delta\le1/4$, both radicands in
\eqref{g:pressure-radicands} are bounded below by
\begin{equation}\label{s:pressure-positivity}
 \frac\pi4-h+\frac{3h^2}{2\pi}-\frac{h\delta}{2}
 \ge\frac{3}{2\pi}(h-3\pi/8)^2+\frac{5\pi}{128}>0.
\end{equation}
By \eqref{s:perimeter-domain}, $\sigma_i^2<11/14+225/128<64/25$.
Since $d_i\ge\sqrt\pi>7/4$ and $h>7/6$,
\begin{equation}\label{eq:body-100}
 \partial_P(d_x\sigma_x+d_y\sigma_y)
 =-\frac{d_x h(1+\delta/2)}{2\sigma_x}
  -\frac{d_y h(1-\delta/2)}{2\sigma_y}\le-\frac{245}{192}.
\end{equation}
By $[PV^{\rm pl}(\mu,c)]'=a_V+d_V<24/25$
and \eqref{s:angular-perimeter-derivative},
\begin{equation}\label{s:loss-perimeter-slope}
 \{\sqrt2P[1-V^{\rm pl}-\Gamma_{\mathrm{ang}}]\}'
 >\sqrt2(1-24/25-1/8)>-25/96.
\end{equation}
By \eqref{eq:body-090} and \eqref{s:loss-perimeter-slope},
both $N_A+25P/96$ and $N_{\rm ang}+25P/96$ are nondecreasing
whenever used, and so is their minimum. Thus
$N^U(P_2)-N^U(P_1)\ge-25(P_2-P_1)/96$.
In \eqref{s:residual-definition}, $sT$ has derivative $-s\le0$,
while $C_h(r_1),C_h(r_2),A_x,A_y$ are independent of $P$.
Using $W(P_2)\le W(P_1)$ and integrating \eqref{eq:body-100} gives
\[
 \mathcal R(h,P_2,\delta)-\mathcal R(h,P_1,\delta)
 \le\left(-\frac{245}{192}+\frac{225}{64}\frac{25}{96}\right)(P_2-P_1)
 =-\frac{2215}{6144}(P_2-P_1).
\]
\end{proof}

\begin{remark}\label{s:stronger-mean-propagation}
Under the hypotheses of Proposition~\ref{s:perimeter}, suppose the
$y$-mean reference moments lie in $[9/50,5/12]$.
Let $\mu_{A,i}^{\rm ref}$ and $R_{A,i}^{\rm ref}$ be the moment and
cutoff of the reference point for $A_i^{\rm pl}$ in
\eqref{s:reference-point}. Then
$d_{A,i}=(R_{A,i}^{\rm ref})^2a_{A,i}
\ge\mu_{A,i}^{\rm ref}a_{A,i}$ and $a_{A,i}>1/2$.
Equation \eqref{eq:body-090} becomes
\begin{equation}\label{s:mean-reference-derivative}
 N_A'=a_{A,x}\left[(R_{A,x}^{\rm ref})^2-\frac\delta2\right]
      +a_{A,y}\left[(R_{A,y}^{\rm ref})^2+\frac\delta2\right].
\end{equation}
If $(R_{A,x}^{\rm ref})^2-\delta/2<0$, \eqref{s:mean-reference-derivative},
the reference lower bounds and $a_{A,x}\le25/12$ give
\begin{equation}\label{eq:body-102}
 N_A'\ge\frac{25}{12}\left(\frac{36}{625}-\frac\delta2\right)
 +\frac12\left(\frac9{50}+\frac\delta2\right)
 =\frac{21}{100}-\frac{19\delta}{24}\ge\frac{29}{2400}.
\end{equation}
If $(R_{A,x}^{\rm ref})^2-\delta/2\ge0$, the $x$ summand in
\eqref{s:mean-reference-derivative} is nonnegative and the $y$ summand is at least
$9/100>29/2400$. Thus $N_A'\ge29/2400$ in both cases.
When $N^U=N_A$, the bounds $W'\le0$ and $\partial_P(sT)=-s\le0$
from the proof of Proposition~\ref{s:perimeter}, together with
\eqref{s:residual-definition} and \eqref{eq:body-100}, give almost everywhere
\begin{equation}\label{s:stronger-mean-residual}
 \partial_P\mathcal R\le-\frac{245}{192}-h^2N_A'
 \le-\frac{245}{192}.
\end{equation}
\end{remark}

\begin{lemma}\label{s:loss-plane}
For $U_3$, ratio $r=3/2$, cutoff $R=9/10$, and cost $v_{16}$,
the exact support coefficients satisfy $a_V+d_V<37/40<24/25$.
\end{lemma}
\begin{proof}
The mass-one point is
$\alpha_0=439453125/585568811$, $c_0=3\alpha_0/2$.
For fixed $R$, the functions $U_3$, $M_3$ and $B_R+R^2$ are
homogeneous of degree one in $(\alpha,c)$. Thus their first
derivatives at $(\alpha_0,c_0)$ equal those at $(1,3/2)$.
Using the latter pair,
\begin{equation}\label{s:loss-reference-data}
 \begin{gathered}
 (k_{1,\alpha},k_{2,\alpha},k_{3,\alpha})
 =(8/25,104/625,296/3125)=:(q_1,q_2,q_3),\\
 M_\alpha=\frac{2550115817}{1757812500},\qquad
 B_\alpha=\frac{134750671353}{68359375000}.
 \end{gathered}
\end{equation}
Put $C=(1+R^2)M_\alpha-B_\alpha
=161096931157/246093750000>0$.
The support construction in \eqref{s:cost-intercept} gives the positive
linear functional
\begin{equation}\label{s:loss-positive-functional}
 a_V+d_V=\Xi(w_{16}),\qquad
 \Xi(w)=\frac{(1+R^2)\int_0^R w(t)(U_3)_\alpha(t)\,\dd t+
 R^2Cw(R)}{B_\alpha}.
\end{equation}
Here $(U_3)_\alpha\ge0$ and
$C=M_\alpha-\int_0^R(U_3)_\alpha\,\dd t\ge1/R-R>0$.
Since $w_{16}\le w_\infty$, where
$w_\infty(t)=1/[\sqrt{1-t^2}(1+\sqrt{1-t^2})]$,
we have $\Xi(w_{16})\le\Xi(w_\infty)$.
Write $s=\sqrt{19}/10$ and $A=\arcsin(9/10)$. Direct integration gives
\begin{align}
 I(s,A):=\int_0^Rw_\infty(U_3)_\alpha\,\dd t
 &=\frac{1-s}{R}+q_1R+
 q_2\left(\frac{Rs}{2}+\frac{R^3}{3}\right)\notag\\
 &\quad+q_3\left(\frac{Rs(3+2R^2)}8+\frac{R^5}5\right)
 -A\left(q_1+\frac{q_2}2+\frac{3q_3}8\right).
 \label{s:loss-integral}
\end{align}
In \eqref{s:loss-integral},
$\partial_A I=-(q_1+q_2/2+3q_3/8)<0$ and
$\partial_s I=-1/R+q_2R/2+q_3R(3+2R^2)/8<0$
by \eqref{s:loss-reference-data}. The rational bounds
$s>s_0=61/140$ and
$A>\pi/3+2(9/10-\sqrt3/2)>A_0=223/200$
follow from $\pi>333/106$ and $\sqrt3<97/56$.
Also $1/[s(1+s)]$ decreases with $s$. Substitution gives
\begin{equation}\label{s:loss-rational-margin}
 \frac{37}{40}B_\alpha-
 \left[(1+R^2)I(s_0,A_0)+
 \frac{R^2C}{s_0(1+s_0)}\right]
 =\frac{172434332520641}{12572314453125000}>0.
\end{equation}
Equations \eqref{s:loss-positive-functional}--\eqref{s:loss-rational-margin}
prove the claim.
\end{proof}

For $P=1$, fix $h>0$, $t_x,t_y,\eta$ and an interval $[d_-,d_+]\subset[0,1/4]$.
Suppose that $W_{\rm aff}(h,\cdot)$ and $N_{\rm aff}(h,\cdot)$ are
affine on this interval and satisfy
\begin{equation}\label{s:affine-comparisons}
 W_{\rm aff}(h,\delta)\le W(h,1,\delta),\qquad
 N_{\rm aff}(h,\delta)\ge N^U(h,1,\delta).
\end{equation}
For example, if \eqref{eq:body-091} is used, use the affine function $R_i^{\rm pl}(1)$ for $r_i$. In \eqref{eq:body-094} one
can replace $D_0$ by $2\delta_0\delta-\delta_0^2\le\delta^2\le D_0$,
with $\delta_0$ fixed, obtaining
\begin{equation}\label{eq:body-104}
 r_T=2h+\frac{2h^2}{\pi}
 \left[T^{\rm pl}(\mu,c)+\frac{2\delta_0\delta-\delta_0^2}{6}-2\right].
\end{equation}
If the combination defining $W$ uses $R_x+R_y$ with coefficient
$\min(s_x,s_y)$ and an individual bound $r_x$ or $r_y$, take
\begin{equation}\label{eq:body-105}
 W_{\rm aff}=
 \begin{cases}
 (s_x-s_y)r_x+s_y r_T,&s_x\ge s_y,\\
 s_x r_T+(s_y-s_x)r_y,&s_x\le s_y.
 \end{cases}
\end{equation}
If $N^U=N_A$, then $N_{\rm aff}=N_A(1)$ is affine in $\delta$.
For $N_{\rm ang}$ or $\min(N_A,N_{\rm ang})$ one may take
\begin{equation}\label{eq:body-106}
 N_{\rm aff}=\sqrt2[1-V^{\rm pl}(\mu,c)-\gamma_{\rm aff}(h,\delta)],
 \qquad \gamma_{\rm aff}(h,\delta)\le\Gamma_{\mathrm{ang}}(h,1,\delta).
\end{equation}
At fixed $h$ the function
\begin{equation}\label{eq:body-107}
 \Gamma_{\mathrm{ang}}(h,1,\delta)=\frac{h^2}{9(h^2/2+\pi)K_*}
 \left(2-\frac{4h}{\pi}+\frac52\max\{\delta^2,D_s(h,1)\}\right)_+^2
\end{equation}
is convex, since the maximum of the two convex functions
$\delta^2,D_s(h,1)$ is convex and $x\mapsto(x_+)^2$ is convex
and nondecreasing. Any supporting affine function
$\gamma_{\rm aff}(h,\cdot)$ satisfies \eqref{eq:body-106}.

\begin{lemma}\label{s:concavity}
Under \eqref{s:affine-comparisons}, let $\Phi_h$ be
\eqref{s:residual-definition} at $P=1$ with $W,N^U$ replaced by
$W_{\rm aff},N_{\rm aff}$. Then $\Phi_h$ is strictly concave in
$\delta$ on $[d_-,d_+]$. If $\Phi_h(d_-),\Phi_h(d_+)>0$, the
parameters of a minimizer with $\delta\in[d_-,d_+]$ cannot satisfy
the hypotheses of Proposition~\ref{s:perimeter} on $[P,1]$.
\end{lemma}
\begin{proof}
By \eqref{s:pressure-positivity},
$A_h:=\pi/4-h+3h^2/(2\pi)$ satisfies $A_h-h\delta/2>0$.
In $\Phi_h$ only $d_x\sqrt{A_h-h\delta/2}+d_y\sqrt{A_h+h\delta/2}$
has a nonzero second derivative:
\begin{equation}\label{s:asymmetry-second-derivative}
 \Phi_h''(\delta)=
 -\frac{d_xh^2}{16(A_h-h\delta/2)^{3/2}}
 -\frac{d_yh^2}{16(A_h+h\delta/2)^{3/2}}<0.
\end{equation}
Thus, for $\delta=(1-\theta)d_-+\theta d_+$ and
$0\le\theta\le1$,
\[
 \Phi_h(\delta)\ge (1-\theta)\Phi_h(d_-)+\theta\Phi_h(d_+)>0.
\]
But \eqref{g:raw-dual-eq}, \eqref{s:perimeter-mean}, and
\eqref{s:affine-comparisons} give
\begin{equation}\label{s:residual-reduction}
 0\ge\mathcal R(h,P,\delta)\ge
 \mathcal R(h,1,\delta)+\frac{2215}{6144}(1-P)
 \ge\Phi_h(\delta)+\frac{2215}{6144}(1-P),
\end{equation}
a contradiction.
\end{proof}

\section{Genus-two exclusion by joint concavity and moment inequalities}
\label{sec:range-exclusion}

\begin{theorem}\label{t:coverage-theorem}
Let $E\subset Q$ be a downward monotone minimizer at $v=1/\pi$
with $P\le1$. If its reflected boundary is connected and its
outward mean curvature $h=\kappa_1+\kappa_2$ is positive, then its
genus is at most one.
\end{theorem}

By \cite[Theorem 19]{Ros2006}, a connected reflected boundary of genus
at least two has genus two or three.
Propositions~\ref{s:n3} and~\ref{s:low-h} exclude genus three and the
genus-two interval $0<h\le H_0$. For genus two, \eqref{eq:body-069}
gives $h<15/8$, and Proposition~\ref{s:body-quarter} gives
$0\le\delta<1/4$. The remaining comparisons are as follows.

\begin{enumerate}
\item For $H_0\le h\le3/2$, Lemma~\ref{t:projection-lines} gives
$\delta\le\Theta(h)$. Proposition~\ref{s:perimeter} and
\eqref{t:shared-angular-support} give the necessary inequality
\eqref{t:low-necessary}. With $\theta=\delta/\Theta(h)$,
Lemma~\ref{t:joint-concavity} and \eqref{t:wall-interpolation}
--\eqref{t:low-profile-bounds} yield
\[
 0\ge\mathcal R_1(h,\theta u_*,\delta)
 \ge (1-\theta)f_0(h)+\theta f_+(h)>0
 \qquad\text{(Proposition~\ref{t:low-curvature}).}
\]
\item For $3/2\le h<15/8$, substitute the moment inequalities
\eqref{t:high-normal} and \eqref{t:high-log} into
\eqref{g:raw-dual-eq} at $(t_x,t_y,\eta)=(\pi/4,\pi/4,1)$.
The comparison \eqref{t:high-perimeter} and the endpoint bounds
\eqref{t:high-endpoints}, \eqref{t:high-positive} give
\[
 0\ge\mathcal R_{\rm high}(h,1,\delta)
 \ge (1-16\delta^2)B_0(h)+16\delta^2 B_q(h)>0
 \qquad\text{(Proposition~\ref{t:high-curvature}).}
\]
\end{enumerate}

\subsection{Projection estimates and longitudinal moment bounds}
\label{t:projection-preliminaries}

Throughout this section $F$ is the function in \eqref{g:pressure-eq}.
We shall use
\begin{equation}\label{t:primitive-bounds}
 r-\frac{r^3}{6}-\frac{r^5}{40}-\frac{r^7}{112}
       -\frac{5r^9}{1152(1-r^2)}
 \leq F(r)\leq r-\frac{r^3}{6}-\frac{r^5}{40}-\frac{r^7}{112},
 \qquad 0\leq r<1.
\end{equation}
The integrated binomial series has positive subtracted coefficients;
its tail is bounded by the first omitted coefficient times a geometric
series. All constants below are rationally bounded using
$L=314159/100000<\pi<U=3927/1250$, as in \eqref{s:pi-bounds}.

\begin{lemma}\label{t:projection-lines}
Every minimizer $E$ satisfying \eqref{s:geometric-hypotheses} and
$H_0\leq h\leq3/2$ satisfies
\begin{equation}\label{t:source-lines}
 R_x,R_y\geq\ell_0(h)=\frac{201}{100}-\frac{13h}{10},\qquad
 R_x,R_y\geq\ell_1(h)=\frac{53}{40}-\frac{37h}{50}.
\end{equation}
On this interval the polynomial
\begin{equation}\label{t:cubic-source}
 \ell_3=16(7/5-h)^2\ell_0+[1-16(7/5-h)^2]\ell_1
\end{equation}
satisfies $0<\ell_3(h)\le\min\{R_x,R_y\}$. Moreover
\begin{equation}\label{t:natural-asymmetry}
 0\leq\delta\leq\Theta(h):=2-\frac{3h}{\pi}
             -\frac{2F(\ell_1(h))}{h},\qquad 0<\Theta(h)<1/3.
\end{equation}
\end{lemma}
\begin{proof}
For this proof write $a=H_0$, $b=3/2$.
In \eqref{s:reference-point} and \eqref{s:mean-plane}, take $N=6$
and use the two pairs $(r,R)$ below. Rounding each coefficient of
\eqref{s:mean-upper-support} upwards gives
$N_z/P\le A_6(\mu,c)\le\alpha_j\mu+\beta_jc+\zeta_j$,
since $\mu,c\ge0$.
\begin{equation}\label{t:projection-mean-planes}
\begin{array}{c|rrr|cc}
 j&10^6\alpha_j&10^6\beta_j&10^6\zeta_j&r&R\\\hline
 0&507024&20769&410690&3/2&9/10\\
 1&512317&22331&405806&11/8&89/100.
\end{array}
\end{equation}
By \eqref{g:genus-two-parameters},
\[
 N_z\le N_{z,j}^U(P):=(\alpha_j+\zeta_j)P-\alpha_jh/\pi+
 \beta_j\pi/(2h^2),\qquad j=0,1.
\]
For each $j$, use $N_{z,j}^U$ in \eqref{g:body-gates} to define
$\rho_{U,j}(p),e_j(p),r_{b,j}(p)$ for $P\le p\le1$.
On this interval $(N_{z,j}^U)'=\alpha_j+\zeta_j>0$ and
$0\le N_{z,j}^U\le\alpha_j+\zeta_j-\alpha_ja/U+\beta_jU/(2a^2)<19/25$.
It follows that $\rho_{U,j}$ and $e_j$ are nondecreasing and
$e_j(p)\ge e_j(P)\ge0$. Hence $r_{b,j}$ is nonincreasing.
At $p=1$, differentiating $\rho_{U,j}(1)$ in $\pi$ gives
\[
 \frac{2h^2}{\pi^3}[2\alpha_jh-(\alpha_j+\zeta_j)\pi]<0
 \quad\text{by}\quad 2\alpha_j b<(\alpha_j+\zeta_j)L .
\]
Define
\begin{equation}\label{t:projection-polynomials}
\begin{aligned}
 \rho_j&=1+\beta_j+2(\alpha_j+\zeta_j)h^2/L-2\alpha_jh^3/L^2,\\
 f_j&=1+h-\ell_j,\quad Q_j=f_j-\rho_j,\quad
 D_j=f_j^2+1+4h^2/U-2f_j\rho_j.
\end{aligned}
\end{equation}
Here $\rho_j''$ decreases and $\rho_j''(b)>0$; hence $Q_j'$ decreases.
The comparisons $Q_j'(b)>0$ and $Q_j(a)>0$ give $Q_j>0$.
For $D_0$, the signs $D_0''''>0$, $D_0'''(b)<0$, and
$D_0''(b)>2$ imply $D_0''>2$ throughout $[a,b]$.
For $D_1''=Ah^2+Bh+C$, its leading coefficient is positive and
$C-B^2/(4A)>19/100$. The remaining comparisons are
\begin{equation}\label{t:projection-point-comparisons}
\begin{array}{c|c|cc}
 j&c_j&D_j(c_j)&|D_j'(c_j)|\\\hline
 0&6/5&>1/10000&<3/1000\\
 1&259/200&>1/10000&<1/2000.
\end{array}
\end{equation}
For $\kappa_0=2$, $\kappa_1=19/100$, completing the square gives
$D_j(h)\geq D_j(c_j)-D_j'(c_j)^2/(2\kappa_j)>0$.
At $p=1$ we have $\rho_{U,j}(1)\le\rho_j$. Therefore
\[
 \begin{aligned}
 f_j-\rho_{U,j}(1)&\ge Q_j>0,\\
 (f_j-\rho_{U,j}(1))^2-e_j(1)
 &=f_j^2-2f_j\rho_{U,j}(1)+1+4h^2/\pi\ge D_j>0 .
 \end{aligned}
\]
Thus $f_j-\rho_{U,j}(1)>\sqrt{e_j(1)}$. By
\eqref{g:body-gates} and \eqref{g:active-axis-gates},
$R_x,R_y\ge r_{b,j}(P)\ge r_{b,j}(1)>\ell_j(h)$, proving
\eqref{t:source-lines}. The rational comparisons are in
Appendix~\ref{app:projection-comparisons}.

The coefficients $16(7/5-h)^2$ and $1-16(7/5-h)^2$ are nonnegative
and sum to one on $[H_0,3/2]$. Since $\ell_0\ge3/50$ and
$\ell_1\ge43/200$ there, \eqref{t:cubic-source} gives
$0<\ell_3\le\max(\ell_0,\ell_1)\le\min(R_x,R_y)$.
Taking $j=1$ in \eqref{s:body-P} gives $\delta\le\Theta(h)$.
To bound $\Theta$, put $k=37/50$. We have
$\ell_1<0.455$, $\sqrt{1-\ell_1^2}>0.89$ and $F(\ell_1)>0.21$.
Consequently
\[
 \Theta'=-3/\pi+2k\sqrt{1-\ell_1^2}/h+2F(\ell_1)/h^2
 >-3/L+2k(0.89)/b+2(0.21)/b^2>0.
\]
The endpoint comparisons $\Theta(a)>0$ and $\Theta(b)<1/3$
complete the proof.
\end{proof}

The two cost functions needed here are $v_{16}$ and $\phi_8$ defined
in Subsection~\ref{s:support-section}. Their argument is the same
longitudinal normal coordinate $X=|\nu_z|$. At $N=6$ the construction
\eqref{s:reference-point}--\eqref{s:cost-plane} gives the following
global lower supports, with coefficients rounded downwards:
\begin{equation}\label{t:cost-planes}
\begin{array}{c|rrr|cc}
 &10^6a_j&10^6b_j&10^6d_j&r&R\\\hline
 V_0&1155867&-28788&-266220&11/8&89/100\\
 V_1&993934&-22441&-178367&11/10&17/20\\
 T&-1038203&-13029&1087687&3/2&9/10.
\end{array}
\end{equation}
The first two rows bound $\E v_{16}(X)$, and the last bounds
$\E\phi_8(X)$. Coefficientwise rounding is valid since
$\mu=\E X^2\geq0$ and $c\geq0$; all three coefficient comparisons
for every row are finite rational inequalities.
At fixed $h$, put $\omega=(h-H_0)/(3/2-H_0)$. Then
$\E v_{16}(|\nu_z|)\ge(1-\omega)V_0^{\rm pl}(\mu,c)+\omega V_1^{\rm pl}(\mu,c)$.
The coefficient $\omega$ is independent of $P$.
Define
\begin{equation}\label{t:common-source-normal}
\begin{aligned}
 T_\eta&=3h/(2\pi)-[\eta-F(\eta)]/h,\qquad\eta=13/20,\\
 R_T&=2h+b_T+2(a_T+d_T-2)h^2/\pi-2a_Th^3/\pi^2,\\
 Z_j&=\sqrt2\{h^2(1-a_j-d_j)+a_jh^3/\pi-b_j\pi/2\},
       \quad j=0,1,\\
 Z&=(1-\omega)Z_0+\omega Z_1,\qquad
 R_0=R_T-h^2/(300\pi).
\end{aligned}
\end{equation}
On $[H_0,3/2]$ we have $T_\eta>11h/25$ and $R_0>9/10$.
For the first use \eqref{t:primitive-bounds}; for the second,
$R_0''<0$ and both endpoint values exceed $9/10$.
These comparisons are included in Appendix~\ref{app:curve-comparisons}.

\subsection{Joint concavity in the angle parameter and moment difference}
\label{t:joint-concavity-section}

Fix $t_0=5\pi/32$, $u_*=3\pi/64$ and $\eta=13/20$.
\begin{equation}\label{t:continuous-wall-family}
 (t_x,t_y)=(t_0-u,t_0+u),\qquad 0\le u\le u_*.
\end{equation}
Use $s_i,A_i,d_i,s,r_1,r_2$ from Lemma~\ref{g:raw-dual}.
By Lemma~\ref{g:raw-dual},
$A_x+A_y=\pi+2(t_x+t_y)=13\pi/8$.
At $P=1$, \eqref{g:pressure-radicands} gives
\[
 A_h=\pi/4-h+3h^2/(2\pi),\qquad
 \sigma_x^2=A_h-h\delta/2,\qquad \sigma_y^2=A_h+h\delta/2.
\]

\begin{lemma}\label{t:joint-concavity}
Fix $H_0\leq h\leq3/2$, $T\geq11h/25$, $R\geq9/10$,
$\ell\geq0$, and $0\leq\lambda\leq h^2/(15\pi)$.
For any affine function $\Lambda_h$ of $\delta$, the function
\begin{align*}
 \mathcal W_h(u,\delta)={}&C_h(r_1)+C_h(r_2)+sT
       +s_x(R+\lambda\delta)+(s_y-s_x)\ell\\
       &+d(t_x)\sigma_x+d(t_y)\sigma_y+\Lambda_h(\delta)
\end{align*}
is strictly jointly concave on $[0,u_*]\times[0,1/3]$.
\end{lemma}

\begin{proof}

Write $a=H_0$, $b=3/2$ within this proof. For $0\le\delta\le1/3$,
\[
 \sigma_y^2\ge\sigma_x^2\ge A_h-h/6
 =\frac3{2\pi}(h-7\pi/18)^2+\frac{5\pi}{216}>0.
\]
Set
\[
 G(u)=C_h(r_1)+C_h(r_2)+sT+s_x(R+\lambda\delta)+(s_y-s_x)\ell
\]
with $\delta$ fixed. We will use
\begin{equation}\label{gm:two-curvature-bounds}
 G''(u)<-9h/5,\qquad
 w(t):=\frac{(d^2)''(t)}{2d(t)}
       =\frac{4(\sin2t+\cos2t)}{d(t)}
 .
\end{equation}
The function $w$ is positive, decreasing and concave on $[7\pi/64,13\pi/64]$.
We verify these scalar bounds explicitly.

The identities $\partial_u t_x=-1$, $\partial_u t_y=1$,
$\partial_\delta\sigma_x=-h/(4\sigma_x)$ and
$\partial_\delta\sigma_y=h/(4\sigma_y)$ give, for $i=x,y$,
\[
 \partial_{u\delta}(d(t_i)\sigma_i)=\frac{h d'(t_i)}{4\sigma_i},\quad
 \partial_{\delta\delta}(d(t_i)\sigma_i)=-\frac{h^2d(t_i)}{16\sigma_i^3},
 \quad \partial_{uu}(d(t_i)\sigma_i)=d''(t_i)\sigma_i .
\]
Using $w=d''+(d')^2/d$ from \eqref{gm:two-curvature-bounds},
their quadratic forms are
\begin{equation}\label{gm:pressure-square}
 \mathrm d^2(d(t_i)\sigma_i)=
 -\frac{d(t_i)h^2}{16\sigma_i^3}
 \left(\dot\delta-\frac{4d'(t_i)\sigma_i^2}{hd(t_i)}\dot u\right)^2
 +w(t_i)\sigma_i\dot u^2.
\end{equation}
Since $t_x\leq t_y$, $w(t_x)\geq w(t_y)$ and $\sigma_x\leq\sigma_y$.
In particular,
\[
 \partial_\delta[w(t_x)\sigma_x+w(t_y)\sigma_y]
 =\frac h4\left[\frac{w(t_y)}{\sigma_y}-\frac{w(t_x)}{\sigma_x}\right]\le0 .
\]
By concavity of $w$, the value at $\delta=0$ is at most
$2w(t_0)\sqrt{A_h}<16h/9$. In the Hessian of $\mathcal W_h$, first
treat $R+\lambda\delta$ as constant. For $\dot u=1$, the maximum of
the resulting quadratic form over $\dot\delta$ equals its Schur
complement in the $\delta\delta$ entry. By \eqref{gm:pressure-square},
the quadratic form is bounded above by
$G''+w(t_x)\sigma_x+w(t_y)\sigma_y<-9h/5+16h/9=-h/45$.

To prove the bound for $G''$ in \eqref{gm:two-curvature-bounds}, put $s_0=\sin t_0$,
$c_0=\cos t_0$, $r_0=s_0(1-\eta/h)$ and $k=r_0^2+c_0^2\le1$.
Then $r_{1,2}=r_0\cos u\pm c_0\sin u$,
$r_{1,2}''=-r_{1,2}$ and $(r_{1,2}')^2+r_{1,2}^2=k$.
For $r=r_1$ or $r_2$, differentiation gives
\[
 \frac{d^2}{du^2}C_h(r)
  =\frac{(h-r)(1-k)}{(1-r^2)^{3/2}}-h\sqrt{1-r^2}+r.
\]
For $H_0\le h\le3/2$ and $0\le u\le u_*$,
$3/40<r_1,r_2<2/5$ by the bounds in
Appendix~\ref{app:curve-comparisons}. For $r\in[3/40,2/5]$,
\[
 \frac{1-r/h}{(1-r^2)^{3/2}}\le\frac{1-2r/3}{(1-r^2)^{3/2}}<24/25.
\]
The derivative of the middle function is
$(4r-1)(2-r)/[3(1-r^2)^{5/2}]$. Its endpoint values are
less than $24/25$ by squaring, and it attains its maximum at an endpoint.
The two remaining bounds for $\partial_u^2[C_h(r_1)+C_h(r_2)]$ are
\begin{equation}\label{t:joint-calibration-bounds}
 \frac{48}{25}(1-k)+\frac{2r_0}{h}<7/10,\qquad
 \sqrt{1-r_1^2}+\sqrt{1-r_2^2}>19/10.
\end{equation}
For the first, use $1-k=s_0^2(2\eta/h-\eta^2/h^2)$ to write its
left side as $A/h-B/h^2$, with $A=(96/25)s_0^2\eta+2s_0$ and
$B=(48/25)s_0^2\eta^2+2s_0\eta$. The bounds
$0.47139<s_0<0.47140$ give
$2B-Aa<0$ and $A/a-B/a^2<7/10$. Thus
$(A/h-B/h^2)'=(2B-Ah)/h^3<0$ on $[a,b]$.

For the second inequality, use $c_0<0.882$, $\sin u<0.148$,
$r_0<0.47140(1-\eta/b)$ and $\sqrt{1-r_i^2}>0.916$. Then
\[
 \sum_{i=1}^2\sqrt{1-r_i^2}
 =2-\sum_{i=1}^2\frac{r_i^2}{1+\sqrt{1-r_i^2}}
 \ge 2-\frac{2r_0^2+2c_0^2\sin^2u}{1+0.916}>\frac{19}{10}.
\]
Using $r_1+r_2=2r_0\cos u\le2r_0$ and
\eqref{t:joint-calibration-bounds} gives
$\partial_u^2[C_h(r_1)+C_h(r_2)]<h(7/10-19/10)=-6h/5$.
Since $s_i''=-s_i$ and
\[
 sT+s_x(R+\lambda\delta)+(s_y-s_x)\ell
 >\frac25h+\frac3{10}\ge\frac35h,
\]
it follows that $G''<-9h/5$.

To verify the properties of $w$, put $Q=d^2$,
$g=\sin2t+\cos2t$, $v=\cos2t-\sin2t$ and $A=6-2v$.
Then $Q'=2A$, $A'=4g$ and
\[
 \frac{w'}w=\frac{2v}g-\frac A Q,\qquad
 \frac{Qw''}w=-4Q-\frac{4vA}g+\frac{3A^2}Q-4g.
\]
On the full angle interval,
$g>31/25$, $-0.67<v<0.14$, $Q>32/5$ and $A<7.34$.
At $t=7\pi/64$, $v=\sqrt2\sin(\pi/32)>1/10$, hence $A<29/5$.
Using $g\le\sqrt2<99/70$ and $A<7.34$, we obtain
\[
 \left(\tfrac{28}{5}Q-A^2\right)'=2A(28/5-4g)>-1,\qquad
 \left(\tfrac{28}{5}Q-A^2\right)\big|_{t=7\pi/64}
 >\tfrac{28}{5}\tfrac{32}{5}-(29/5)^2=\tfrac{11}{5}.
\]
The interval has length $3\pi/32<3/10$, so $A^2/Q<28/5$ throughout.
For $v\geq0$ use $Q<176/25$ and $A>5.72$ to get $w'<0$;
for $v\leq0$ this sign is immediate. The displayed formula for
$w''$ is negative for $v\geq0$ using $Q>32/5$.
For $v\leq0$ use $Q\geq Q(\pi/8)>7$ and
\[
 -28+\frac{4(0.67)(7.34)}{1.24}+3(5.6)-4(1.24)<0.
\]
Finally $w(t_0)<2$, since $Q(t_0)>8$ and $g^2\leq2$.
Also $(A_h/h^2)'=(h-\pi/2)/h^3<0$ on $[a,b]$ and
$A_a/a^2<16/81$, so $\sqrt{A_h}<4h/9$.
The finite comparisons in these estimates follow from the rational
rules in Appendix~\ref{app:wall-comparisons}.

Set
\[
 B_0=\partial_{u\delta}[d(t_x)\sigma_x+d(t_y)\sigma_y]
 =\frac h4\left[\frac{d'(t_x)}{\sigma_x}+
                  \frac{d'(t_y)}{\sigma_y}\right]>\frac h2,
\]
since $d<4$, $d'=A/d>5.72/4>1$ and
$\sigma_i^2\le U/4-5a/6+3b^2/(2L)<1$.
For $s_x(R+\lambda\delta)$, the mixed derivative is
$-\lambda\cos t_x$ and the second $\delta$ derivative is zero.
Thus $\partial_{uu}\mathcal W_h$ and
$\partial_{\delta\delta}\mathcal W_h$ agree with the corresponding
entries when $R+\lambda\delta$ is frozen, and
$0<B_0-\lambda\cos t_x\le B_0$, since
$\lambda\le h^2/(15\pi)<h/2$. As $\partial_{\delta\delta}\mathcal W_h<0$,
\[
 \partial_{uu}\mathcal W_h-
 \frac{(B_0-\lambda\cos t_x)^2}{\partial_{\delta\delta}\mathcal W_h}
 \le\partial_{uu}\mathcal W_h-
 \frac{B_0^2}{\partial_{\delta\delta}\mathcal W_h}<-\frac h{45}.
\]
The Hessian is therefore negative definite.
\end{proof}

\subsection{Exclusion for $H_0\le h\le3/2$ by joint concavity}
\label{t:low-curvature-section}

\begin{proposition}\label{t:low-curvature}
There is no minimizer $E$ satisfying \eqref{s:geometric-hypotheses}
and $H_0\leq h\leq3/2$.
\end{proposition}
\begin{proof}
Fix $h,\delta,u$ and use the coefficients of \eqref{t:cost-planes}.
In \eqref{s:residual-definition}, take
\[
 W=s_x R_{\rm tot}(P)+(s_y-s_x)\ell_3,\qquad
 N^U=\sqrt2P[1-((1-\omega)V_0^{\rm pl}+\omega V_1^{\rm pl})(\mu,c)
                     -\Gamma_{\mathrm{ang}}(h,P,\delta)].
\]
Use \eqref{eq:body-094} with $T^{\rm pl}$ from
\eqref{t:cost-planes}. Since $s_y\ge s_x>0$, the two coefficients
in $W$ are nonnegative. Let $r_{b,j}$ denote \eqref{g:body-gates}
with $N_z^U=N_{z,j}^U$ as in the proof of Lemma~\ref{t:projection-lines}.
For every $j$ and $P\le p\le1$, that proof gives
$r_{b,j}(p)\ge r_{b,j}(1)>\ell_j$. Hence
\[
 \ell_3\le\max_{j=0,1}\ell_j
 \le\max_{j=0,1}r_{b,j}(1)\le\max_{j=0,1}r_{b,j}(p).
\]
Choosing an index attaining the maximum at $p=1$ verifies the
constant lower bound for $R_y$ in Proposition~\ref{s:perimeter}.
Also $(1-\omega)(a_0+d_0)+\omega(a_1+d_1)<24/25$ by
\eqref{t:cost-planes}.
Along $P\le p\le1$, $h^2p\le(3/2)^2<\pi$,
$q(p)\ge h/\pi\ge H_0/U>1/4>\delta$ and
$2-3q(p)\ge2-3q(P)\ge\delta$. The conditions
$N_{z,j}^U\in[0,1]$ and $e\ge0$ in \eqref{g:body-gates} follow
from Lemma~\ref{t:projection-lines}. Proposition~\ref{s:perimeter}
therefore gives $\mathcal R(h,1,\delta)\le\mathcal R(h,P,\delta)\le0$.

At $P=1$, $D_0\ge\delta^2\ge\delta/5-1/100$.
Also $D_s(h,1)$ in \eqref{eq:body-092} is decreasing on
$[H_0,3/2]$: $h(2-4h/\pi)$ is positive and decreasing there, while
$h^2/2+\pi$ is increasing. At $H_0$ the bound is
\[
 D_s(H_0,1)\le
 \frac{H_0^2(2-4H_0/U)^2}{18(H_0^2/2+L)}<\frac1{100}.
\]
At $\delta_0=1/10$ one has $\delta_0^2>D_s(h,1)$ and
$2-4h/\pi+5\delta_0^2/2>0$. By \eqref{eq:body-107},
$\sqrt2h^2\Gamma_{\mathrm{ang}}(h,1,\delta)\ge G(h,\delta)$, with $G$
its affine tangent at $\delta_0$. Substitution in
\eqref{eq:body-089} and \eqref{eq:body-094} gives
\begin{subequations}\label{t:shared-angular-support}
\begin{align}
 R_{\rm tot}(1)&\ge R_\sharp=R_T+\frac{h^2}{3\pi}(\delta/5-1/100)
             =R_0+\frac{h^2\delta}{15\pi},\label{t:total-affine}\\
 -h^2N^U(1)&\ge-Z+G,\qquad
 G=\frac{\sqrt2h^4(81/40-4h/\pi)(77/40+\delta-4h/\pi)}
                   {9K_*(\pi+h^2/2)} .
 \label{t:normal-affine}
\end{align}
\end{subequations}
Here $K_*^2=27648/3125$ by \eqref{eq:body-093}. In particular,
\begin{equation}\label{t:weighted-total-source}
 W(h,1,\delta)\ge s_xR_\sharp+(s_y-s_x)\ell_3.
\end{equation}
At $P=1$ define
\begin{equation}\label{t:low-residual}
\begin{split}
 \mathcal R_1(h,u,\delta)={}&C_h(r_1)+C_h(r_2)+sT_\eta
     +s_xR_\sharp+(s_y-s_x)\ell_3\\
  &+d(t_x)\sigma_x+d(t_y)\sigma_y-Z-13\pi/8+G.
\end{split}
\end{equation}
By \eqref{t:shared-angular-support} and \eqref{t:weighted-total-source},
every minimizer under consideration satisfies
\begin{equation}\label{t:low-necessary}
 \mathcal R_1(h,u,\delta)\le\mathcal R(h,1,\delta)\le0
 \qquad (0\le u\le u_*).
\end{equation}
Lemma~\ref{t:joint-concavity} applies, with
$T=T_\eta$, $R=R_0$, $\lambda=h^2/(15\pi)$ and $\ell=\ell_3$.
Together with \eqref{t:natural-asymmetry}, it gives
\begin{equation}\label{t:wall-interpolation}
 \mathcal R_1\left(h,u_*\frac\delta{\Theta(h)},\delta\right)
 \geq\left(1-\frac\delta{\Theta(h)}\right)f_0(h)
             +\frac\delta{\Theta(h)}f_+(h),
\end{equation}
where $f_0(h)=\mathcal R_1(h,0,0)$ and
$f_+(h)=\mathcal R_1(h,u_*,\Theta(h))$.
Thus the two global inequalities
\begin{equation}\label{t:low-profile-bounds}
 f_0>1/200,\qquad f_+>1/10000\qquad(H_0\leq h\leq3/2)
\end{equation}
will complete the exclusion.

For $f_0$, write $a=H_0$, $b=3/2$ within this proof.

Put $s_0=\sin t_0$, $r=s_0(1-\eta/h)$ and $z=\sqrt{1-r^2}$.
Then $z>0.963$ and
\[
 \frac{d}{dh}C_h(r(h))\ge z-\frac{s_0\eta r}{hz},\qquad
 \frac{s_0}{4\eta}-\frac{r(h)}h
 =\frac{s_0(h-2\eta)^2}{4\eta h^2}\ge0.
\]
Thus $(2C_h(r))'>9/5$. Since $F(\eta)\le\eta-\eta^3/6$
and $s_0>471/1000$, one has
\[
 (2s_0T_\eta)'=2s_0\!\left[\frac3{2\pi}
                   +\frac{\eta-F(\eta)}{h^2}\right]
 >\frac{942}{1000}\left[\frac3{2U}+\frac{\eta^3}{6(3/2)^2}\right]
 >\frac9{20}.
\]
By \eqref{t:common-source-normal}, $R_0'''=-12a_T/\pi^2>0$.
Appendix~\ref{app:curve-comparisons} gives $R_0''(3/2)<0$ and
$s_0R_0'(3/2)>-3/20$. Hence $s_0R_0'(h)>-3/20$ on $[H_0,3/2]$.

Let $H(h)=d(t_0)A_h'/\sqrt{A_h}-Z'(h)$.
We have $A_h'>3/25$, $\sqrt{A_h}<3/5$, $d(t_0)>2.8736$,
$Z''''<0$, and $Z'''(b)>-3/2$. Since
\[
 (2d(t_0)\sqrt{A_h})'''=-\frac{3d(t_0)A_h'}{8A_h^{5/2}}<-3/2,
\]
$H$ is concave. Both endpoint comparisons give $H>-19/10$.

Finally write $G(h,0)=c_*h^4BC/(\pi+h^2/2)$, with
$c_*=\sqrt2/(9K_*)$, $B=81/40-4h/\pi$,
$C=77/40-4h/\pi$.
Here $B,C>0$, $B+C<1$, $BC<1/4$, $c_*<27/500$ and
$\pi+h^2/2>19/5$. Dropping the positive term in the derivative gives
\[
 G'(h,0)>-\frac{27}{500}
 \left\{\frac{4b^4}{L(19/5)}+
               \frac{b^5}{4(19/5)^2}\right\}>-1/10.
\]
Adding these bounds gives
$f_0'>9/5+9/20-3/20-19/10-1/10=1/10$.
The single endpoint comparison $f_0(a)>5675/10^6>1/200$ proves the first part of \eqref{t:low-profile-bounds}. The finite comparisons used here are recorded in Appendix~\ref{app:curve-comparisons}.

At $(u,\delta)=(u_*,\Theta(h))$, substitute \eqref{t:natural-asymmetry}
in \eqref{g:pressure-radicands} with $P=1$:
\begin{equation}\label{t:body-pressure-cancellation}
 \sigma_x^2=\pi/4-2h+3h^2/\pi+F(\ell_1(h)),\qquad
 \sigma_y^2=\pi/4-F(\ell_1(h)).
\end{equation}
Let $\Psi_h(w)=\mathcal R_1(h,u_*,2-3h/\pi-2w/h)$. By
\eqref{t:asymmetry-substitution-derivative} in
Appendix~\ref{app:body-polynomial},
\[
 f_+(h)=\Psi_h(F(\ell_1(h)))\ge\Psi_h(\underline F(\ell_1(h))),
 \qquad \underline F(r)=r-r^3/6-r^5/30.
\]
All dependence of $\Psi_h$ on $w$ is through $R_\sharp$, $G$,
and $\sigma_x,\sigma_y$.
Lemma~\ref{ac:polynomial} bounds
$h^7(\pi+h^2/2)f_+(h)$ below by a polynomial greater than $1/100$.
Since $h^7(\pi+h^2/2)<100$, this proves the second inequality in
\eqref{t:low-profile-bounds}.

With $\theta=\delta/\Theta(h)\in[0,1]$,
\eqref{t:wall-interpolation} and \eqref{t:low-profile-bounds}
give $\mathcal R_1(h,\theta u_*,\delta)>1/10000$, contradicting
\eqref{t:low-necessary}. Lemma~\ref{t:joint-concavity} applies to
the entire segment since $\Theta(h)<1/3$.
\end{proof}

\subsection{High-curvature exclusion by elementary moment inequalities}
\label{t:high-curvature-section}

\begin{proposition}\label{t:high-curvature}
There is no minimizer $E$ satisfying \eqref{s:geometric-hypotheses}
and $3/2\leq h<15/8$.
\end{proposition}
\begin{proof}
By continuity, \eqref{g:raw-dual-eq} holds at
$t_x=t_y=\pi/4$ and $\eta=1$. At the actual perimeter $P$,
Cauchy--Schwarz and \eqref{g:genus-two-parameters} give
\begin{equation}\label{t:high-normal}
 N_x+N_y\le g(P):=
 \sqrt{(hP/\pi-P^2\delta)/2}+\sqrt{(hP/\pi+P^2\delta)/2}.
\end{equation}
By Jensen's inequality for $\nu_i^4,\nu_i^6$ as powers of $\nu_i^2$,
the first three terms of $x\operatorname{atanh}x$ give
\begin{equation}\label{t:high-log}
 L_x+L_y\ge j(q,\delta):=q+(q^2+\delta^2)/6+(q^3+3q\delta^2)/20.
\end{equation}
Substitution in \eqref{g:source} yields
\begin{equation}\label{t:high-total}
 R_x+R_y\ge2h+(2h^2/\pi)P[j(h/(\pi P),\delta)-2].
\end{equation}
Let $\mathcal R_{\rm high}(h,p,\delta)$ be the left side of
\eqref{g:raw-dual-eq} minus its right side, at these angles, with
$P$ replaced by $p$, $N_x+N_y$ by $g(p)$, and $R_x+R_y$ by the
right side of \eqref{t:high-total} at $p$. At $p=P$ it is
nonpositive. For $P\le p\le1$, Proposition~\ref{s:body-quarter}
gives $p\delta/(h/\pi)\le U/6<\sqrt3/2$, where $U=3927/1250$.
Put $z=h/\pi$ and $\xi=\sqrt{1-(p\delta/z)^2}>1/2$. Then
\[
 g(p)^2=zp+p\sqrt{z^2-p^2\delta^2},\qquad
 (g^2)'(p)=z(1+2\xi-\xi^{-1})=z(2\xi-1)(\xi+1)/\xi>0 .
\]
Both roots $\sigma_i(h,p,\delta)$ in \eqref{g:pressure-radicands}
are decreasing in $p$ and positive by \eqref{s:pressure-positivity}.
At fixed $h$, the derivative of $1-p+3h/(2\pi)-(1-\pi/4)/h$
is $-1$. With $q(p)=h/(\pi p)$,
\[
 \frac{d}{dp}\{p[j(q(p),\delta)-2]\}
 =-\frac{q(p)^2}{6}-\frac{q(p)^3}{10}+\frac{\delta^2}{6}-2<0.
\]
Consequently
\begin{equation}\label{t:high-perimeter}
 \mathcal R_{\rm high}(h,1,\delta)\le
 \mathcal R_{\rm high}(h,P,\delta)\le0 .
\end{equation}

At $P=1$ put $u=\delta^2\in[0,1/16]$, and define
\begin{align*}
 C&=\sqrt{h^2/2+h-1/2},\qquad A_h=\pi/4-h+3h^2/(2\pi),\\
 T&=3h/(2\pi)-(1-\pi/4)/h,\\
 S&=2h-4h^2/\pi+2h^3/\pi^2+h^4/(3\pi^3)+h^5/(10\pi^4),\\
 W&=2C+\sqrt2T+S/\sqrt2-2\pi-h^2\sqrt{2h/\pi},\\
 J&=\frac{\pi^{3/2}\sqrt h}{4\sqrt2}
                      +\frac{\sqrt2h^2}{6\pi}+\frac{3\sqrt2h^3}{20\pi^2}.
\end{align*}
For $u\in[0,1/16]$, \eqref{s:pressure-positivity} gives
$\sigma_y^2\ge\sigma_x^2=A_h-h\sqrt u/2\ge5\pi/128$.
The function $u\mapsto\sigma_x+\sigma_y
=\sqrt{2A_h+2\sqrt{A_h^2-h^2u/4}}$ is concave, since $A_h^2-h^2u/4=\sigma_x^2\sigma_y^2>0$
is affine in $u$ and the square root is increasing and concave.
Thus
\[
 \sigma_x+\sigma_y\ge(1-16u)\,2\sqrt{A_h}
 +16u(\sqrt{A_h-h/8}+\sqrt{A_h+h/8}).
\]
At $p=1$, $g(1)=\sqrt{z+\sqrt{z^2-u}}$ is concave in $u$,
so the tangent at $u=0$ gives
$g(1)\le\sqrt{2z}-u/(4\sqrt2 z^{3/2})$.
Using $C_h((h-1)/(\sqrt2h))\ge C$ in
\eqref{t:high-perimeter}, we obtain
\begin{equation}\label{t:high-affine}
 0\ge\mathcal R_{\rm high}(h,1,\delta)
 \ge (1-16u)B_0(h)+16u B_q(h),
\end{equation}
where
\begin{equation}\label{t:high-endpoints}
 B_0=W+4\sqrt\pi\sqrt{A_h},\qquad
 B_q=W+2\sqrt\pi[\sqrt{A_h-h/8}+\sqrt{A_h+h/8}]+J/16.
\end{equation}
For $3/2\le h\le15/8$, Appendix~\ref{app:high-comparisons} gives
\[
 (2C)'\ge23/15,\qquad (\sqrt2T)'>147/220,\qquad S'/\sqrt2>-16/35 .
\]
Also $J'>0$. For fixed $\delta=0$ or $1/4$, put
$p_\pm(h)=3h^2/(2\pi)-(1\pm\delta/2)h+\pi/4$ and
$c_\pm=3/8-(1\pm\delta/2)^2/4\ge15/256$.
Since $p_\pm'>0$ on this interval,
$(\sqrt{p_\pm})'''=-3c_\pm p_\pm'/(2p_\pm^{5/2})<0$.
Therefore
\[
 D_\delta(h)=2\sqrt\pi(\sqrt{p_+}+\sqrt{p_-})'
                   -\frac52\sqrt{2/\pi}\,h^{3/2}
\]
satisfies
\[
 D_\delta''(h)=2\sqrt\pi[(\sqrt{p_+})'''+(\sqrt{p_-})''']
              -\frac{15}{8}\sqrt{2/\pi}\,h^{-1/2}<0 .
\]
The endpoint bounds
$D_\delta(3/2),D_\delta(15/8)>-3/2$ in
Appendix~\ref{app:high-comparisons} imply $D_\delta>-3/2$.
Hence $B_0',B_q'>23/15+147/220-16/35-3/2=1129/4620>0$.
The comparisons at $h=3/2$ in that appendix now give
\begin{equation}\label{t:high-positive}
 B_0(h)>1/10,\qquad B_q(h)>1/25\qquad(3/2\le h\le15/8),
\end{equation}
contradicting \eqref{t:high-affine}.
\end{proof}

\subsection{Completion of the parameter-domain exclusion}\label{t:coverage}

\begin{proof}[Proof of Theorem~\ref{t:coverage-theorem}]
If the genus is at least two, \cite[Theorem 19]{Ros2006} gives
$g(\Sigma)\in\{2,3\}$.
Proposition~\ref{s:n3} excludes genus three. In genus two,
$P\ge3h/(2\pi)$ and $h^2P\le\pi$ give $h<15/8$ by
\eqref{eq:body-069}. If $0<h\le H_0$, use Proposition~\ref{s:low-h}.
For $H_0<h\le3/2$ use Proposition~\ref{t:low-curvature} and
\eqref{t:low-necessary}. For $3/2\le h<15/8$ use
Proposition~\ref{t:high-curvature} and \eqref{t:high-perimeter}.
Both comparisons include $P=1$ and $\delta=0$.
\end{proof}

\section{The sharp profile and all equality cases}\label{sec:profile-rigidity}

Theorem~\ref{t:coverage-theorem} gives $I(b)=1$ in
Lemma~\ref{p:second-contact}. Lemma~\ref{p:first-contact} combines
\revision{the classical Hersch and profile estimates to give $I(a)=2\pi/9$
through the BV inequality \eqref{p:gap-BV-derivative}.}
Proposition~\ref{p:full-profile} then applies
\cite[Proposition 4.6]{Milman} to $[0,a]$, $[a,b]$ and $[b,1-b]$.
The classification at $a,b$ uses Lemma~\ref{p:recovery} in
Section~\ref{p:equality}.

\subsection{Profile contact at the transition volumes}\label{p:profile}

\begin{lemma}\label{p:second-contact}
The relative isoperimetric profile of the unit cube satisfies $I(b)=1$.
\end{lemma}
\begin{proof}
A coordinate slab gives $I(b)\le1$. If $P=I(b)<1=m(b)$, choose a
downward monotone minimizer as in Section~\ref{g:selection}. By
\eqref{g:profile-slope}, concavity and $I(1/2)=1$, its mean curvature
satisfies
\[
 h\ge I'_+(b)\ge\frac{1-P}{1/2-b}>0.
\]
Its reflected boundary is connected by \eqref{g:component-test}.
The low-genus classification \cite[Proposition 4.4]{Milman} and
\cite[Theorem 19]{Ros2006} give $g(\Sigma)\in\{2,3\}$, contradicting
Theorem~\ref{t:coverage-theorem}.
\end{proof}

The next implication combines Ros's genus bound
\cite[Theorem 19]{Ros2006}, the curvature--area estimate
\cite[Theorem 7(2)]{HPRR} and the profile differential inequality
\cite[\S7]{BP}.
\begin{lemma}\label{p:first-contact}
If $I(b)=1$, then $I(a)=2\pi/9$.
\end{lemma}
\begin{proof}
Suppose $I(b)=1$ and $I(a)<2\pi/9$. Let $(s,t)$ be the component of
$\{v\in(0,1/2):I(v)<m(v)\}$ containing $a$.
Continuity gives $0\leq s<a<t\leq b$.
For each $v\in(s,t)$, choose a cube minimizer and let $\Sigma_v$
be its reflected boundary
and $h_{\Sigma_v}$ its mean curvature. Put $h=I'_+$.
Concavity and \eqref{g:profile-slope} give
\begin{equation}\label{p:gap-slope-comparison}
 0<\frac{I(b)-I(v)}{b-v}\le h(v)\le h_{\Sigma_v}\le I'_-(v).
\end{equation}
By Lemma~\ref{g:reflection}, $\Sigma_v$ bounds an unrestricted torus
minimizer. Since $h_{\Sigma_v}>0$, \eqref{g:component-test} gives
connectedness. The low-genus classification in
\cite[Proposition 4.4]{Milman} and $I(v)<m(v)$ exclude genus zero and one.
\revision{Ros's bound \cite[Theorem 19]{Ros2006} therefore gives
$g(\Sigma_v)\in\{2,3\}$, and Lemma~\ref{g:hersch} yields}
\begin{equation}\label{p:gap-hersch}
 I(v)h(v)^2\le I(v)h_{\Sigma_v}^2\le\pi\qquad(s<v<t).
\end{equation}
At every differentiability point, $h_{\Sigma_v}=I'(v)=h(v)$.
The profile inequality of \cite[Theorem 9(b)]{HPRR}, in the
normalization of \cite[Proposition 4.5, Equation (4.2)]{Milman},
gives at almost every $v\in(s,t)$
\begin{equation}\label{p:gap-ODI}
 I^2h'+Ih^2\le\frac{\pi}{2}\chi(\Sigma_v)\le-\pi .
\end{equation}
Here $h'$ denotes the almost-everywhere derivative of $I'_+$.
At a point of second-order expansion, $I(v)=J(v)$ and $I\le J$
in Lemma~\ref{g:profile} give $J'(v)=h(v)$ and $J''(v)\ge h'(v)$.
For a smooth function $y$, the identity
\begin{equation}\label{p:first-integral-identity}
 \frac d{dv}\bigl[y^2(y')^2+2\pi y\bigr]
 =2y'\bigl[y^2y''+y(y')^2+\pi\bigr]
\end{equation}
gives a first integral when $y^2y''+y(y')^2=-\pi$.
On compact subintervals of $(s,t)$, $I$ is Lipschitz and
$h$ is bounded, nonincreasing and right-continuous. Put
\begin{equation}\label{p:gap-energy}
 \mathcal E(v)=I(v)^2h(v)^2+2\pi I(v).
\end{equation}
The BV chain and product rules
\cite[Theorem 3.96 and Example 3.97]{AFP}, localized to the bounded
range of $(I,h)$, give
\begin{equation}\label{p:gap-BV-derivative}
 D^{\rm ac}\mathcal E=2h(I^2h'+Ih^2+\pi)\,dv\le0,\qquad
 D^{\rm s}\mathcal E=I^2D^{\rm s}(h^2)\le0.
\end{equation}
Here $D^{\rm s}I=0$ because $I$ is locally Lipschitz, and
$D(h^2)\le0$ because $h\ge0$ is nonincreasing.
For every relatively compact Borel set $B\subset(s,t)$,
$D^{\rm s}\mathcal E(B)=\int_B I^2\,d[D^{\rm s}(h^2)]$.
At a jump $v_0$ this is
$D^{\rm s}\mathcal E(\{v_0\})
=I(v_0)^2[h(v_0+)^2-h(v_0-)^2]$.
Thus for $s<v_1<v_2<t$,
$\mathcal E(v_2)-\mathcal E(v_1)=D\mathcal E((v_1,v_2])\le0$.

Equation~\eqref{p:gap-hersch} gives
$0\le\mathcal E(v)\le3\pi I(v)$.
Since $I(s)=p_B(s)\le2\pi/9$ for $s>0$ and $I(0)=0$,
\begin{equation}\label{p:gap-left-energy}
 \mathcal E(s+)\le3\pi I(s)\le24\pi^2/36 .
\end{equation}
Since $a<t\le b$, we have $I(t)=p_C(t)$ and $I(v)\le p_C(v)$
for $a<v<t$. For $a<v<t$,
$[I(t)-I(v)]/(t-v)\ge[p_C(t)-p_C(v)]/(t-v)$. Let $v\uparrow t$ to get
\begin{equation}\label{p:gap-right-slope}
 I'_-(t)\ge p_C'(t)=\frac{\pi}{2I(t)}.
\end{equation}
Since $h(t-)=I'_-(t)$ and $I(t)=\sqrt{\pi t}\ge2\pi/9$,
\begin{equation}\label{p:gap-right-energy}
 \mathcal E(t-)\geq\frac{\pi^2}{4}+2\pi I(t)
 \geq\frac{25\pi^2}{36}.
\end{equation}
Equations \eqref{p:gap-left-energy} and \eqref{p:gap-right-energy}
contradict $\mathcal E(t-)\le\mathcal E(s+)$ from
\eqref{p:gap-BV-derivative}.
\end{proof}

\begin{proposition}\label{p:full-profile}
For $0\le v\le1/2$, $I(v)=m(v)$. The minimizing families satisfy
\begin{equation}\label{p:open-phase-models}
 \mathcal M_v=
 \begin{cases}
  \mathcal F_B(v),&0<v<a,\\
  \mathcal F_C(v),&a<v<b,\\
  \mathcal F_S(v),&b<v\le1/2.
 \end{cases}
\end{equation}
\end{proposition}
\begin{proof}
By Lemmas~\ref{g:reflection}, \ref{p:second-contact} and
\ref{p:first-contact}, $I_T=I$ and
\[
 I_T(0)=0,\qquad I_T(a)=2\pi/9,\qquad I_T(b)=I_T(1-b)=1 .
\]
These are the values of the candidate profile $m$ at the endpoints of
$[0,a]$, $[a,b]$ and $[b,1-b]$, with $m(v)=m(1-v)$ on $[1/2,1]$.
Apply \cite[Proposition 4.6]{Milman} on these three intervals.
It gives $I_T=m$ and classifies the torus minimizers in the interiors,
including $1/2\in(b,1-b)$.

A sphere centre fixed by the three coordinate reflections belongs
to $\{0,1\}^3$ modulo $2$. A shortest closed geodesic is parallel
to a coordinate axis, and the centre of its transverse disk is in
$\{0,1\}^2$ by the two transverse reflections. The ball and tube
therefore restrict to $\mathcal F_B(v)$ and $\mathcal F_C(v)$.
In the slab case the two tori have normalized area $1/2$ each,
so their primitive integer normal has length one. Reflection
invariance places the centre of the normal interval at $0$ or $1$,
giving $\mathcal F_S(v)$.
\end{proof}

\subsection{Equality classification at the transition volumes}\label{p:equality}

It remains to prove the classifications at $v=a,b$:
\begin{equation}\label{p:transition-models}
 \mathcal M_a=\mathcal F_B(a)\cup\mathcal F_C(a),\qquad
 \mathcal M_b=\mathcal F_C(b)\cup\mathcal F_S(b).
\end{equation}

Let $E$ be any cube minimizer at $v\in\{a,b\}$ and let
$\Sigma=\partial\widehat E\subset T$. At $v=a$,
Proposition~\ref{p:full-profile} and \eqref{g:profile-slope} give
\begin{equation}\label{p:first-transition-curvature}
 P=\frac{2\pi}{9},\qquad \frac94=I'_+(a)\le h\le I'_-(a)=3.
\end{equation}
Hence $\Sigma$ is connected by \eqref{g:component-test}.
If $g(\Sigma)\ge2$, Ros's bound \cite[Theorem 19]{Ros2006} and
Lemma~\ref{g:hersch} imply
\begin{equation}\label{p:first-transition-exclusion}
 \pi\ge h^2P\ge\left(\frac94\right)^2\frac{2\pi}{9}
 =\frac{9\pi}{8}>\pi.
\end{equation}
This is the curvature--area contradiction used in the proof of
\cite[Corollary 15]{HPRR}, in the present normalization.
Thus $g(\Sigma)\le1$, and \cite[Proposition 4.4]{Milman} gives a
sphere or a shortest-axis cylinder. The reflection argument in
Proposition~\ref{p:full-profile} places $E$ in
$\mathcal F_B(a)\cup\mathcal F_C(a)$.

At $v=b$, \eqref{g:profile-slope} gives
$0=I'_+(b)\le h\le I'_-(b)=\pi/2$.
If $h=0$ and $B\not\equiv0$, then $\Sigma$ is connected by
\eqref{g:component-test}. Gauss--Bonnet gives
$2\pi\chi(\Sigma)=\int_\Sigma K\,dA=-\frac12\int_\Sigma|B|^2\,dA<0$.
For the outward parallel region in $T$, \eqref{g:parallel} gives
\[
 P(\tau)=1+\pi\chi(\Sigma)\tau^2/4<1,\qquad
 V(\tau)=b+\tau+\pi\chi(\Sigma)\tau^3/12\in(b,1/2)
\]
for sufficiently small $\tau>0$. This contradicts
$1=I(V(\tau))\le P(\tau)$ by Proposition~\ref{p:full-profile}.
Hence $B\equiv0$ whenever $h=0$. By the classical flat case
\cite[Theorem 5]{HPRR}, in the form \cite[Proposition 4.4]{Milman},
$\Sigma$ consists of two parallel totally geodesic tori. If $\nu_0$
is a common unit normal to their lifts, coordinate reflection invariance
gives
\[
 R_i\nu_0\in\{\nu_0,-\nu_0\}\quad(i=x,y,z)
 \quad\Longrightarrow\quad
 \nu_0\in\{\pm e_x,\pm e_y,\pm e_z\}.
\]
For this coordinate $i$, the region is
$\widehat E=\{x\in T:x_i\in J\}$ modulo null sets, where
$J\subset\R/2\mathbb Z$ is an interval of length $2b$.
The identity $-J=J$ fixes its centre at $0$ or $1$ modulo $2$.
Consequently,
\begin{equation}\label{p:zero-curvature-slab}
 E=Q\cap\{x_i<b\}\quad\text{or}\quad
 E=Q\cap\{x_i>1-b\}\qquad\text{modulo null sets}.
\end{equation}

For a minimizer at $b$, set $E^{(0)}=E$ and
$E^{(j)}=S_jE^{(j-1)}$ for $j=1,2,3$, where $S_1,S_2,S_3$
are $S_x,S_y,S_z$ from Section~\ref{g:selection}.
By \eqref{g:compression-perimeter}, every $E^{(j)}$ minimizes at $b$
and $E^{(3)}$ is downward monotone. Each mean curvature lies in
$[0,\pi/2]$ by \eqref{g:profile-slope}.
For $E^{(3)}$, the case $h=0$ gives a slab by
\eqref{p:zero-curvature-slab}. If $h>0$,
\eqref{g:component-test} gives connectedness.
Theorem~\ref{t:coverage-theorem} excludes $g\ge2$. Hence $g\le1$. In the classification of
Section~\ref{g:foundations}, a ball has perimeter
$(9/(2\pi))^{1/3}>1$ and a slab has $h=0$. Thus $E^{(3)}$
reflects to a shortest-axis cylinder.

For the equality conditions in one-dimensional Steiner
symmetrization, see \cite[Theorems 1.1 and 1.3]{CCF}.
The relative-perimeter proof below treats interior fibre endpoints,
full fibres and the face to which a fibre is attached.

The BV slicing estimate is classical \cite[Proposition 3.4]{BCF};
see \cite[Proposition 2.5]{Milman} for the slab form.
The reflection statement in
\cite[Remark 2.17]{Milman} concerns the two-dimensional weighted
model slab.

\begin{lemma}\label{p:recovery}
Let $E$ be a cube minimizer at $a$ or $b$. If $M=S_iE$ is one of
the transition models, then $|E\triangle M|=0$ or
$|E\triangle\rho_iM|=0$, where $\rho_i$ sends $x_i$ to $1-x_i$.
\end{lemma}
\begin{proof}
Write $Q=D\times(0,1)$ in direction $e_i$. The set $M$ is
downward monotone in that direction. Let $\ell(x)=|M_x|=|E_x|$
for almost every $x\in D$, using the model representative of $\ell$.
By \eqref{g:compression-perimeter}, $\Per(E;Q)=\Per(M;Q)$.
If $M$ is invariant in $e_i$, then $\ell\in\{0,1\}$ almost
everywhere and $|E_x\triangle M_x|=0$ for almost every $x$.

Otherwise $B_*=\{x\in D:0<\ell(x)<1\}$ is connected and open.
For the corner ball, the transverse cylinder and the slab normal
to $e_i$ it is respectively a quarter disk, a strip and $D$.
The function $\ell$ is smooth on $B_*$ and
$\Per(M;Q)=\int_{B_*}\sqrt{1+|\nabla\ell|^2}\,dx$.
Let $\widehat\pi_i:\partial\widehat E\to(\mathbb R/2\mathbb Z)^2$
be coordinate projection on the compact reflected boundary.
Its critical set is $\{\nu_i=0\}$, whose image has area zero by Sard.
For a regular value $x$, the preimage is a compact discrete set,
hence finite. Put $k(x)=\#(\partial E\cap(\{x\}\times(0,1)))$.
Locally write these points as $(x,t_j(x))$ and put
$\epsilon_j=\operatorname{sign}(\nu_i(x,t_j(x)))$. For regular
$x$, the interval lengths in the fibre sum to
$\ell(x)=c+\sum_{j=1}^{k(x)}\epsilon_jt_j(x)$ locally, where
$c\in\{0,1\}$ accounts for the endpoint $1$ if the fibre meets that face.
Thus
\[
 \nabla\ell=\sum_{j=1}^{k(x)}\epsilon_j\nabla t_j,\qquad
 \sum_{j=1}^{k(x)}|(1,\epsilon_j\nabla t_j)|
 \ge \sqrt{k(x)^2+|\nabla\ell|^2}.
\]
A transverse interior endpoint implies $0<|E_x|<1$, and for
almost every $x\in B_*$ there is at least one such endpoint.
On the free boundary $\Gamma=\partial E\cap Q$, the projection
has Jacobian $|\nu_i|$. Set $F_{\rm vert}=\Gamma\cap\{\nu_i=0\}$.
The area formula with weight $1/|\nu_i|$ on
$\Gamma\setminus F_{\rm vert}$ and the preceding inequality give
\begin{equation}\label{p:fibre-deficit}
 \begin{aligned}
 0&=\Per(E;Q)-\Per(M;Q)\\
 &\ge\int_{B_*}\left[\sqrt{k(x)^2+|\nabla\ell|^2}
                   -\sqrt{1+|\nabla\ell|^2}\right]dx
           +\mathcal H^2(F_{\rm vert})\ge0 .
 \end{aligned}
\end{equation}
Hence $k(x)=1$ almost everywhere on $B_*$ and
$\mathcal H^2(F_{\rm vert})=0$. The sole endpoint is $\ell(x)$
or $1-\ell(x)$, and its gradient has norm $|\nabla\ell(x)|$.
The area formula therefore gives, at area-almost every
$p\in\Gamma$ with $\pi_i(p)\in B_*$,
\begin{equation}\label{p:nonvertical-normal}
 |\nu_i(p)|=(1+|\nabla\ell(\pi_i(p))|^2)^{-1/2}.
\end{equation}
Let $\Gamma_*=\Gamma\cap\pi_i^{-1}(B_*)$. The equality
\eqref{p:nonvertical-normal} and the relation
\[
 (p_i-\ell(\pi_i p))(p_i-1+\ell(\pi_i p))=0
\]
hold $\mathcal H^2$-almost everywhere on $\Gamma_*$. Both hold
everywhere there by continuity and positive area of every nonempty
open surface patch.

For a compact $K\subset B_*$, set $\varepsilon=\min_{x\in K}
\min\{\ell(x),1-\ell(x)\}>0$. Then $\pi_i^{-1}(K)\cap\Gamma_*$
is a closed subset of $K\times[\varepsilon,1-\varepsilon]$.
Thus $\pi_i:\Gamma_*\to B_*$ is proper. By
\eqref{p:nonvertical-normal} it is also a local diffeomorphism.
Its image is nonempty, open and closed in the connected set $B_*$.
Hence it is a covering of $B_*$ with constant fibre cardinality.
Since $k=1$ almost everywhere, it is a diffeomorphism. In particular,
$\nu_i$ has constant nonzero sign on $\Gamma_*$. For almost every $x\in B_*$,
$E_x=(0,\ell(x))$ when the sign is positive, and
$E_x=(1-\ell(x),1)$ when it is negative, modulo null sets.
Taking $M_*=M$ or $M_*=\rho_iM$ according to this sign, we get
$|E_x\triangle(M_*)_x|=0$ almost everywhere on $B_*$. The same
identity holds on $D\setminus B_*$ since $\ell\in\{0,1\}$ there.
Thus $|E\triangle M_*|=\int_D|E_x\triangle(M_*)_x|\,dx=0$.
\end{proof}

\begin{proof}[Completion of the proof of Theorem~\ref{thm:main}]
Proposition~\ref{p:full-profile} gives $I=m$ and the classification
away from $a,b$ for $0<v\le1/2$. At $a$,
\eqref{p:first-transition-exclusion} and the classical low-genus
classification already place every original minimizer in
$\mathcal F_B(a)\cup\mathcal F_C(a)$.
At $b$, the preceding argument and the reflection constraints in
Proposition~\ref{p:full-profile} give $E^{(3)}\in\mathcal M_b$.
For $j=3,2,1$, Lemma~\ref{p:recovery} and reflection invariance give
\[
 E^{(j)}\in\mathcal M_b\quad\Longrightarrow\quad
 E^{(j-1)}\in\{E^{(j)},\rho_j E^{(j)}\}\subset\mathcal M_b
 \quad\text{modulo null sets}.
\]
Thus $E=E^{(0)}\in\mathcal M_b$. Each listed model has perimeter
$m(v)$, proving both identities in \eqref{p:transition-models}. Together with
\eqref{p:open-phase-models}, this proves the classification for
$0<v\le1/2$. Complementation gives the classification for $v>1/2$,
with the empty and full sets at $v=0,1$.
\end{proof}

\subsection{The cuboid profile}\label{p:cuboid-profile}

Use $Q_\beta,I_\beta,m_\beta$ and $a_\beta$ as in
Section~\ref{intro:cuboid}. In the reflected torus
$T_\beta=\R^3/(2\beta\mathbb Z\times2\mathbb Z\times2\mathbb Z)$
the normalized volume and area measures are $dx/(8\beta)$ and
$d\sigma_\beta=dA/(8\beta)$.

\begin{lemma}\label{p:cuboid-reduction}
\begin{equation}\label{p:cuboid-comparison}
 I_1(v)\le I_\beta(v)\le m_\beta(v),\qquad 0\le v\le1/2.
\end{equation}
If $I_\beta(v)<m_\beta(v)$ for $0<v<1/2$, the reflected boundary
of every cuboid minimizer in $T_\beta$ is connected, has genus two
or three, and has $h>0$.
\end{lemma}
\begin{proof}
Apply \cite[Lemma 2.33]{Milman} to slabs with base $(0,1)^2$ and
widths $T_1=\beta\le T_2=1$. With the normalized profiles defined
in Section~\ref{intro:cuboid}, that lemma gives
\begin{equation}\label{p:cuboid-stretch}
 \beta I_\beta(v)\le I_1(v)\le I_\beta(v)
 \qquad(0\le v\le1).
\end{equation}
For $0\le v\le a_\beta$, the corner eighth-ball has radius
$(6\beta v/\pi)^{1/3}\le2\beta/3<\beta$ and normalized perimeter
$(9\pi/(2\beta))^{1/3}v^{2/3}$. For $a_\beta\le v\le b$,
the quarter-cylinder along the edge of length $\beta$ has radius
$2\sqrt{v/\pi}\le2/\pi<1$ and normalized perimeter $\sqrt{\pi v}$.
For $b\le v\le1/2$, the slab $Q_\beta\cap\{x_3<v\}$ has normalized
perimeter $1$. These are competitors of normalized volume $v$,
so $I_\beta(v)\le m_\beta(v)$. In particular, Theorem~\ref{thm:main}
and \eqref{p:cuboid-comparison} give
\[
 1=I_1(v)\le I_\beta(v)\le m_\beta(v)=1\qquad(b\le v\le1/2).
\]

The three applications of \cite[Remark 2.4]{Milman} as in
Lemma~\ref{g:reflection}, with lengths $\beta,1,1$, identify
$I_\beta$ with the normalized profile on $T_\beta$ and reflect each
cuboid minimizer to an unrestricted torus minimizer. With $d\sigma$
replaced by $d\sigma_\beta$, the variation formulas
\eqref{g:first-variation}, \eqref{g:stability} and the slope bounds
\eqref{g:profile-slope} apply. Concavity follows from
\cite[Proposition 2.3]{Milman}. If $P=I_\beta(v)<m_\beta(v)\le1$,
then Theorem~\ref{thm:main} and \eqref{p:cuboid-comparison} give $v<b$, and
\begin{equation}\label{eq:body-258}
 h\ge I'_{\beta,+}(v)\ge
 \frac{I_\beta(b)-P}{b-v}
 =\frac{1-P}{b-v}>0 .
\end{equation}
Since $h>0$, the component test \eqref{g:component-test} gives
connectedness. The low-genus classification
\cite[Proposition 4.4]{Milman} excludes genus zero and one, and
Ros's bound \cite[Theorem 19]{Ros2006} gives $g(\Sigma)\in\{2,3\}$.
\end{proof}

\begin{proof}[Proof of Corollary~\ref{cor:cuboid-profile}]
By \eqref{p:cuboid-comparison}, $I_\beta(b)=1$.
Suppose $I_\beta(a_\beta)<2\pi\beta/9$ and let $(s,t)$
be the component of $\{v\in(0,1/2):I_\beta(v)<m_\beta(v)\}$
containing $a_\beta$. Continuity gives
$0\leq s<a_\beta<t\leq b$. On $(s,t)$ put $h=I'_{\beta,+}$. For the boundary $\Sigma_v$
from Lemma~\ref{p:cuboid-reduction}, let $h_{\Sigma_v}$
be its mean curvature. By \eqref{eq:body-258} and
\eqref{g:profile-slope}, $0<h(v)\le h_{\Sigma_v}$,
with equality where $I_\beta$ is differentiable.
For $\Sigma_v\subset T_\beta$, the physical-area estimate
\eqref{g:hersch-physical} applies with
$\mathcal H^2(\Sigma_v)=8\beta I_\beta(v)$ and $H=h_{\Sigma_v}/2$.
This gives \eqref{p:cuboid-hersch}; the profile inequality
\eqref{p:cuboid-differential} is
\cite[Proposition 4.5, Equation (4.2)]{Milman} in this normalization:
\begin{align}
 I_\beta(v)h(v)^2 &\le I_\beta(v)h_{\Sigma_v}^2\le \pi/\beta
 \qquad(v\in(s,t)),\label{p:cuboid-hersch}\\
 I_\beta^2h'+I_\beta h^2 &\le\frac{\pi}{2\beta}\chi(\Sigma_v)
 \le-\pi/\beta \qquad\text{almost everywhere on }(s,t).
 \label{p:cuboid-differential}
\end{align}
By the BV rules \cite[Theorem 3.96 and Example 3.97]{AFP},
\eqref{p:gap-BV-derivative} with $\pi/\beta$ in place of $\pi$ gives
\begin{equation}\label{p:cuboid-energy}
 \mathcal E_\beta=I_\beta^2h^2+\frac{2\pi}{\beta}I_\beta,
 \qquad D\mathcal E_\beta\le0\quad\text{on }(s,t).
\end{equation}
By \eqref{p:cuboid-hersch} and continuity,
\[
 \mathcal E_\beta(s+)\le(3\pi/\beta)I_\beta(s)
 \le(3\pi/\beta)(2\pi\beta/9)=24\pi^2/36.
\]
Here $I_\beta(s)=m_\beta(s)$ if $s>0$, and $I_\beta(0)=0$.
At $t\in(a_\beta,b]$, the left slope comparison with $\sqrt{\pi v}$
gives $h(t-)=I'_{\beta,-}(t)\ge \pi/(2I_\beta(t))$ and
$I_\beta(t)=\sqrt{\pi t}\ge 2\pi\beta/9$. Thus
\[
 \mathcal E_\beta(t-)
 \ge\pi^2/4+(2\pi/\beta)I_\beta(t)\ge25\pi^2/36,
\]
contradicting $D\mathcal E_\beta\le0$. Therefore
$I_\beta(a_\beta)=2\pi\beta/9$.

The normalized torus profile $I_\beta$ satisfies
\[
 I_\beta(0)=0,\qquad I_\beta(a_\beta)=2\pi\beta/9,\qquad
 I_\beta(b)=I_\beta(1-b)=1 .
\]
Extend $m_\beta$ by complement symmetry. On $[0,a_\beta]$,
$[a_\beta,b]$ and $[b,1-b]$, the endpoint equalities
$I_\beta=m_\beta$ allow us to apply \cite[Proposition 4.6]{Milman}.
Thus $I_\beta=m_\beta$ on $[0,1]$ by symmetry. The competitors
used for \eqref{p:cuboid-comparison} attain $m_\beta$ on their
respective closed intervals, so both adjacent models minimize
at each transition. The empty and full sets minimize at $0,1$.
\end{proof}

\section{Quantitative stability and competition at the transitions}\label{sec:stability}
\revision{Near either transition volume, both adjacent model families remain
strictly stable modulo translations after reflection, and their
$L^1$ separation stays positive. Their perimeter difference tends to
zero linearly as the volume approaches the transition. The estimates
below relate this competition between shapes to the degeneration of
the best constant for distance to the minimizing family.}

\revision{We first establish uniform local estimates for each family.}
Proposition~\ref{stab:prop:local} gives
$\Per(E;Q)-\Per(M;Q)\ge\kappa d^2$ for $d<\rho$, where $d$ is the
distance in symmetric-difference volume to the model family of $M$.
The constants $\rho,\kappa$ are uniform on compact model families.
\revision{For competitors outside these neighbourhoods, Theorem~\ref{thm:main}
and BV compactness give a positive infimum of $\delta(E)$ under}
$\alpha_v(E)\ge\rho$ at fixed $v$, or under $d_1(E,v),d_2(E,v)\ge\rho$
\revision{near a transition. Combining the local and global bounds gives the
two-family estimate in Theorem~\ref{stab:thm:multiwell} and the linear
decay in Corollary~\ref{stab:cor:degeneration}.}

\subsection{Model perimeters and radii}
\revision{For $0<v\le1/2$, recall the model perimeters introduced in
Section~\ref{sec:introduction},}
\begin{equation}\label{stab:eq:branches}
 p_B(v)=(9\pi/2)^{1/3}v^{2/3},\qquad
 p_C(v)=\sqrt{\pi v},\qquad p_S(v)=1.
\end{equation}
At minimizing ball volumes $r_B\le2/3<1$, and at minimizing cylinder
volumes $4/9\le r_C\le2/\pi<1$. Thus the reflected models are embedded.

\subsection{Reflection and the second-variation gaps}
On $T$, $P_T$ and $|\cdot|$ now denote physical perimeter
and volume. For finite-perimeter sets $E,F\subset Q$, even
reflection gives the volume and perimeter in
\eqref{g:reflection-normalization} and multiplies the volume of the
symmetric difference by eight. In the notation of this section,
\begin{equation}\label{stab:eq:reflection}
 |\widehat E|=8|E|,\qquad P_T(\widehat E)=8\Per(E;Q),\qquad
 |\widehat E\triangle\widehat F|=8|E\triangle F|.
\end{equation}
On each reflection plane $S$ the two traces of $\chi_{\widehat E}$
agree $\mathcal H^2$-almost everywhere. In the BV gluing formula
$D\chi_{\widehat E}\lfloor S
=(\operatorname{Tr}_S^+\chi_{\widehat E}-
   \operatorname{Tr}_S^-\chi_{\widehat E})\nu_S\,\mathcal H^2\lfloor S=0$,
where $\nu_S$ points from the minus side to the plus side.

For a reflected model $K$, let
\begin{equation}\label{stab:eq:form}
 \mathcal Q_K(u)=\int_{\partial K}(|\nabla_\tau u|^2-q u^2),
 \qquad
 X_K=\left\{u\in H^1(\partial K):\int_{\partial K}u=0,
             \ \int_{\partial K}u\nu_i=0\ (i=1,2,3)\right\},
\end{equation}
where $q=|B|^2$ is constant for each model. The first condition fixes
volume to first order; the other conditions remove normal components
of translations. An identically zero normal component adds no
constraint.

For the sphere of radius $r$, the spherical-harmonic eigenvalues are
$[\ell(\ell+1)-2]/r^2$. The constant is the volume mode and the
degree-one modes are translations. The gap on $X_K$ is therefore
$4/r^2$. For a cylinder with axial period two, the Fourier modes
$e^{im\theta}e^{i\pi ks}$ give
\begin{equation}\label{stab:eq:cylinder-spectrum}
 \lambda_{m,k}=(m^2-1)/r^2+\pi^2k^2 .
\end{equation}
Since $\nu=(\cos\theta,\sin\theta,0)$ in axial coordinates,
\[
 X_K=\{u\in H^1(\partial K):\widehat u_{0,0}
                  =\widehat u_{1,0}=\widehat u_{-1,0}=0\},
\]
where $\widehat u_{m,k}$ denotes the corresponding Fourier coefficient.
Hence
\[
 \lambda_C=\min_{(m,k)\in\mathbb Z^2\setminus\{(0,0),(1,0),(-1,0)\}}
       \lambda_{m,k}=\min\{3/r^2,\pi^2-1/r^2\}.
\]
For $v\in[a,b]$ the radius satisfies $r=2\sqrt{v/\pi}\ge4/9>1/\pi$,
so $\lambda_C>0$.

For a reflected slab, write $\partial K=\Sigma_+\cup\Sigma_-$,
where the outward normals are $\pm e_i$. The two conditions
$\int_{\partial K}u=\int_{\partial K}u\nu_i=0$ in
\eqref{stab:eq:form} give $\int_{\Sigma_+}u=\int_{\Sigma_-}u=0$.
On each component the nonconstant Fourier modes have eigenvalues
$\pi^2|k|^2$, $k\in\mathbb Z^2\setminus\{0\}$. Since $q=0$,
the gap on $X_K$ is $\pi^2$. The gaps at $a,b$ are
\begin{center}
\begin{tabular}{lll}
\toprule
Volume and model & Radius & Gap on $X_K$\\
\midrule
$a$, ball & $2/3$ & $9$\\
$a$, cylinder & $4/9$ & $\pi^2-81/16$\\
$b$, cylinder & $2/\pi$ & $3\pi^2/4$\\
$b$, slab & --- & $\pi^2$\\
\bottomrule
\end{tabular}
\end{center}
In particular both shapes at either transition remain strictly
stable modulo translations. Since $q$ is constant, an $L^2$ gap
$\lambda>0$ gives the $H^1$ bound
\begin{equation}\label{stab:eq:coercivity}
 \|u\|_{H^1}^2=\mathcal Q_K(u)+(1+q)\|u\|_2^2
 \leq\left(1+\frac{1+q}{\lambda}\right)\mathcal Q_K(u),
 \qquad u\in X_K.
\end{equation}

\subsection{Uniform local minimality for arbitrary competitors}
For a single smooth CMC set on the unit flat torus, positivity of
the second variation on $X_K$ yields a quadratic estimate modulo
translations by \cite[Corollary 1.2]{AFM}. Under $x\mapsto x/2$,
perimeter is divided by $4$ and symmetric-difference volume by $8$.
A constant $c$ in that estimate therefore gives $c/16$ on $T$.

We use the following form of \cite[Theorem 4.2]{AFM}. Suppose
$|F_j\triangle F|\to0$, $\sup_jP_T(F_j)<\infty$ and $\partial F$
is smooth. Suppose also that some $\Lambda,r_0>0$ satisfy
\begin{equation}\label{stab:eq:quasiminimality}
 P_T(F_j)\le P_T(G)+\Lambda|F_j\triangle G|
\end{equation}
whenever $F_j\triangle G\Subset B_r(x)$ with $0<r<r_0$.
Write $F'_j$ and $G'$ for their images on the unit torus under
$x\mapsto x/2$. If $F'_j\triangle G'\Subset B_r(x')$ with
$r<\min(r_0/2,1/4)$, then
\[
 P_{\R^3/\mathbb Z^3}(F'_j)-P_{\R^3/\mathbb Z^3}(G')\le 2\Lambda|F'_j\triangle G'|
                    \le (8\pi\Lambda/3)r^3 .
\]
For larger $r$, the uniform perimeter bound is at most a fixed
multiple of $r^3$. Hence $F'_j$ satisfy the common $\omega$-minimality
condition of \cite[Definition 2.5]{AFM}. Theorem 4.2 there, followed
by rescaling to $T$, gives
\begin{equation}\label{stab:eq:normal-convergence}
 \partial F_j=\{x+u_j(x)\nu_F(x):x\in\partial F\},
 \qquad \|u_j\|_{C^1(\partial F)}\longrightarrow0
\end{equation}
for the regular representatives and all sufficiently large $j$.

The following compact-parameter form of \cite[Corollary 1.2]{AFM}
uses the penalization argument in the proof of
\cite[Theorem 1.1]{AFM}. We verify uniformity in the model parameter.
\begin{lemma}\label{stab:lem:uniform-torus}
Let $K_s$ be a compact smooth parameter family of reflected balls
with $0<r<1$, cylinders of axial period two with $1/\pi<r<1$, or
coordinate slabs of volume fraction in $(0,1)$, of a fixed orientation.
There exist $\varepsilon,c>0$, independent of $s$, such that
\begin{equation}\label{stab:eq:torus-local}
 P_T(A)-P_T(K_s)\geq c\,d_{\rm tr}(A,K_s)^2,\qquad
 d_{\rm tr}(A,K_s)=\min_{t\in T}|A\triangle(K_s+t)|,
\end{equation}
whenever $|A|=|K_s|$ and $d_{\rm tr}(A,K_s)<\varepsilon$.
The competitors are arbitrary finite-perimeter sets.
\end{lemma}
\begin{proof}
The compact family has a common positive tubular radius, uniform
$C^3$ bounds and a uniform bound in \eqref{stab:eq:coercivity}.
Let $K_{s,u}$ be a $C^1$ normal graph with $|K_{s,u}|=|K_s|$
and $\|u\|_{C^1}$ small. We impose the translation orthogonality
as in \cite[Lemma 3.8]{AFM}. For $t$ in the effective translation
subspace $\{a:a\cdot\nu_s\equiv0\}^{\perp}$, let $u_t$ be the
normal height of $K_{s,u}+t$ over $\partial K_s$.
In an orthonormal basis $(e_i)$ of that subspace define
\[
 \Phi_{s,i}(u,t)=\int_{\partial K_s}u_t\,(\nu_s\cdot e_i)\,dA .
\]
At $(u,t)=(0,0)$, the derivative of the graph height under
translation by $\tau e_j$ is $\nu_s\cdot e_j$, by differentiation
of the signed distance to $\partial K_s$. Hence
\begin{equation}\label{stab:eq:translation-jacobian}
 \partial_{t_j}\Phi_{s,i}(0,0)=
 \int_{\partial K_s}(\nu_s\cdot e_i)(\nu_s\cdot e_j)\,dA.
\end{equation}
For every nonzero $a$ in the effective subspace,
$\int_{\partial K_s}(a\cdot\nu_s)^2\,dA>0$.
The least eigenvalue of \eqref{stab:eq:translation-jacobian}
on that subspace has a positive lower bound by compactness.
For small $\|u\|_{C^1}+|t|$, the derivatives
$\partial_t\Phi_s(u,t)$ converge uniformly in $s$ to the matrices
in \eqref{stab:eq:translation-jacobian}. Indeed
\[
 \partial_{t_j}u_t(x)=
 \frac{\nu_{s,u,t}(x+u_t(x)\nu_s(x))\cdot e_j}
      {\nu_{s,u,t}(x+u_t(x)\nu_s(x))\cdot\nu_s(x)},
\]
where $\nu_{s,u,t}$ is the normal of the translated graph. This
follows by differentiating its normal-coordinate parametrization
and taking the normal component. The implicit function theorem
now gives, for a common radius in $\|u\|_{C^1}$, a translation
$t$ with $\Phi_s(u,t)=0$. Volume, perimeter and $d_{\rm tr}$
are translation invariant. Relabel the translated height $u$.
For $\bar u=|\partial K_s|^{-1}\int_{\partial K_s}u$,
$\int_{\partial K_s}\nu_s\,dA=0$ gives $u-\bar u\in X_{K_s}$.
The tubular and graph-area Jacobians give
\begin{align}
 |K_{s,u}|-|K_s|&=\int_{\partial K_s}u+O(\|u\|_2^2),\label{stab:eq:graph-volume}\\
 P_T(K_{s,u})-P_T(K_s)-h_s(|K_{s,u}|-|K_s|)
 &=\tfrac12\mathcal Q_{K_s}(u)+R_s(u),\nonumber\\
 |R_s(u)|&\leq C\|u\|_{C^1}\|u\|_{H^1}^2.\label{stab:eq:graph-expansion}
\end{align}
Here $h_s=\kappa_1+\kappa_2$ is the sum of the outward principal
curvatures and $q_s=\kappa_1^2+\kappa_2^2$.
To verify \eqref{stab:eq:graph-expansion}, put
$g_s=\kappa_1\kappa_2$. The exact tubular and graph densities are
\begin{equation}\label{stab:eq:jacobians}
 J_{\rm vol}(t)=1+h_st+g_st^2,\qquad
 J_{\rm area}(u)=\det(I+uB)
 \sqrt{1+|(I+uB)^{-1}\nabla_\tau u|^2}.
\end{equation}
Integration and second-order expansion give
\begin{align}
 |K_{s,u}|-|K_s|
 &=\int_{\partial K_s}(u+h_su^2/2+g_su^3/3),\notag\\
 P_T(K_{s,u})-P_T(K_s)
 &=\int_{\partial K_s}(h_su+g_su^2+|\nabla_\tau u|^2/2)
   +O(\|u\|_{C^1}\|u\|_{H^1}^2).
 \label{stab:eq:second-order-densities}
\end{align}
Subtracting $h_s(|K_{s,u}|-|K_s|)$ from
$P_T(K_{s,u})-P_T(K_s)$ in \eqref{stab:eq:second-order-densities} leaves
$\frac12\int(|\nabla_\tau u|^2+(2g_s-h_s^2)u^2)
=\frac12\mathcal Q_{K_s}(u)$, since $2g_s-h_s^2=-q_s$.
The compact $C^3$ bounds and uniform tubular radius give the
remainder estimate in \eqref{stab:eq:graph-expansion}.

Equal volume implies $|\bar u|\leq C\|u\|_2^2$.
With $u_0=u-\bar u\in X_{K_s}$ and constant $q_s$,
$\mathcal Q(u)=\mathcal Q(u_0)-q_s|\partial K_s|\bar u^2$.
Using \eqref{stab:eq:coercivity} and
$\|u-\bar u\|_{H^1}^2=\|u\|_{H^1}^2-|\partial K_s|\bar u^2$
in \eqref{stab:eq:graph-expansion} gives
\[
 P_T(K_{s,u})-P_T(K_s)
 \ge (c-C\|u\|_{C^1}-C\|u\|_2^2)\|u\|_{H^1}^2
\]
with uniform $c>0$. For sufficiently small $\|u\|_{C^1}$ this implies
\begin{equation}\label{stab:eq:graph-lower}
 P_T(K_{s,u})-P_T(K_s)\geq c_0\|u\|_{H^1}^2
 \geq c_1d_{\rm tr}(K_{s,u},K_s)^2.
\end{equation}
The distance bound in \eqref{stab:eq:graph-lower} follows from
$d_{\rm tr}\leq |K_{s,u}\triangle K_s|\leq C\|u\|_1$ and
the uniform area bound.

For arbitrary finite-perimeter competitors, use the linear perimeter
bound of \cite[Lemma 4.1]{AFM}. To make its constant uniform in $s$,
take the smooth field
$Z_s(x+t\nu_s(x))=\zeta(t)\nu_s(x)$ in a common tubular neighbourhood,
extended by zero. Here $0\le\zeta\le1$ is a smooth compactly supported
function in the tubular interval with $\zeta(0)=1$. Then $|Z_s|\leq1$,
$Z_s=\nu_s$ on $\partial K_s$, and $\|\Div Z_s\|_\infty\leq D$
uniformly. The BV divergence theorem gives, for any $F$,
\begin{equation}\label{eq:body-269}
 P_T(F)-P_T(K_s)\geq
 \int_{\partial^*F}Z_s\cdot\nu_F-\int_{\partial K_s}Z_s\cdot\nu_s
 =\int_{T}(\chi_F-\chi_{K_s})\Div Z_s.
\end{equation}
Choose $t\in T$ with $|(F-t)\triangle K_s|=d_{\rm tr}(F,K_s)$
and apply \eqref{eq:body-269} to $F-t$:
\begin{equation}\label{stab:eq:calibration}
 P_T(F)-P_T(K_s)\geq-Dd_{\rm tr}(F,K_s).
\end{equation}

If~\eqref{stab:eq:torus-local} failed uniformly, we could choose
$s_j\to s_*$ and equal-volume sets $A_j$ with
$\epsilon_j=d_{\rm tr}(A_j,K_{s_j})\in(0,1/j)$ and
$P_T(A_j)-P_T(K_{s_j})<\epsilon_j^2/j$.
Choose $\Lambda_1>D$ and minimize over all finite-perimeter sets:
\begin{equation}\label{stab:eq:selection}
 \mathcal J_j(F)=P_T(F)+\Lambda_1|d_{\rm tr}(F,K_{s_j})-\epsilon_j|
       +\Lambda_2\bigl||F|-|K_{s_j}|\bigr|.
\end{equation}
The direct method applies: both distances are $L^1$ continuous and
perimeter is lower semicontinuous. If $F_j$ is a minimizer and
$d_j=d_{\rm tr}(F_j,K_{s_j})$, comparison with $A_j$ gives
\begin{equation}\label{stab:eq:selected-upper}
 P_T(F_j)-P_T(K_{s_j})+\Lambda_1|d_j-\epsilon_j|
 +\Lambda_2\bigl||F_j|-|K_{s_j}|\bigr|\leq\epsilon_j^2/j.
\end{equation}
Equations \eqref{stab:eq:calibration} and \eqref{stab:eq:selected-upper}
give $\Lambda_1|d_j-\epsilon_j|-D d_j\le\epsilon_j^2/j$.
Using $|d_j-\epsilon_j|\ge\epsilon_j-d_j$ and
$|d_j-\epsilon_j|\ge d_j-\epsilon_j$ yields
\begin{equation}\label{stab:eq:comparable}
 \frac{\Lambda_1\epsilon_j-\epsilon_j^2/j}{\Lambda_1+D}
 \leq d_j\leq
 \frac{\Lambda_1\epsilon_j+\epsilon_j^2/j}{\Lambda_1-D}.
\end{equation}
Both denominators are positive by $\Lambda_1>D$.
Translate $F_j$ to achieve $|F_j\triangle K_{s_j}|=d_j$. Then
$|F_j\triangle K_{s_*}|\le d_j+|K_{s_j}\triangle K_{s_*}|\to0$. By \eqref{stab:eq:selected-upper},
$P_T(F_j)\le\sup_s P_T(K_s)+\epsilon_j^2/j$ independently of $\Lambda_2$.
In fact lower semicontinuity and~\eqref{stab:eq:selected-upper} imply
$P_T(F_j)\to P_T(K_{s_*})$.

We use the volume-penalty argument of \cite[Proposition 2.7]{AFM}.
To choose $\Lambda_2$, let $\Phi_{j,t}$ be the flow of $Z_{s_j}$.
For the translated $F_j$, the volume derivative satisfies
\begin{equation}\label{stab:eq:flow-volume}
 \left.\frac{d}{dt}|\Phi_{j,t}F_j|\right|_{t=0}
 =\int_{F_j}\Div Z_{s_j}
 \geq P_T(K_{s_j})-Dd_j\geq c_2>0\qquad\text{for all large }j.
\end{equation} Let $J_j(t,x)=\det D_x\Phi_{j,t}(x)$ and
$V_j(t)=\int_{F_j}J_j(t,x)\,dx$. Differentiating gives
\begin{align*}
 \partial_tJ_j&=J_j(\Div Z_{s_j})\circ\Phi_{j,t},\\
 \partial_t^2J_j&=J_j[(\Div Z_{s_j})^2+
                  Z_{s_j}\cdot\nabla\Div Z_{s_j}]\circ\Phi_{j,t}.
\end{align*}
The uniform $C^2$ bound on $Z_s$ gives $|V_j''(t)|\le C$ on
a common time interval. By \eqref{stab:eq:flow-volume},
$V_j'(t)\ge c_2/2$ for $|t|\le\tau$, with $\tau>0$ independent
of $j,\Lambda_2$. Since $|e_j|=\bigl||K_{s_j}|-|F_j|\bigr|
\le d_j\to0$, there is $t_j\in[-\tau,\tau]$ with
$V_j(t_j)=|K_{s_j}|$ and $|t_j|\le2|e_j|/c_2$.

On this interval $D_x\Phi_{j,t}=\operatorname{Id}+O(|t|)$
uniformly, and the area formula gives
\[
 P_T(\Phi_{j,t}F_j)=
 \int_{\partial^*F_j}|\operatorname{cof}D_x\Phi_{j,t}\,\nu_{F_j}|\,dA
 \le(1+C|t|)P_T(F_j).
\]
With $w_t=\chi_{\Phi_{j,t}F_j}$ the transport equation is
\begin{equation}\label{stab:eq:bv-transport}
 \partial_tw_t+\Div(Z_{s_j}w_t)=w_t\Div Z_{s_j},
 \qquad \partial_tw_t=-Z_{s_j}\cdot Dw_t.
\end{equation}
Indeed, for $\varphi\in C^\infty(T)$, changing variables gives
\[
 \frac d{dt}\int_T w_t\varphi\,dx
 =\frac d{dt}\int_{F_j}(\varphi\circ\Phi_{j,t})J_j\,dx
 =\int_T w_t\Div(\varphi Z_{s_j})\,dx .
\]
Integrating \eqref{stab:eq:bv-transport} against tests with
$\|\varphi\|_\infty\le1$ gives
\[
 |\Phi_{j,t}F_j\triangle F_j|\le \|Z_{s_j}\|_\infty
 \int_{\min(0,t)}^{\max(0,t)}P_T(\Phi_{j,\xi}F_j)\,d\xi
 \le C|t|P_T(F_j).
\]
By the perimeter bound and $|t_j|\le2|e_j|/c_2$, these two estimates
give, for constants independent of $\Lambda_2$,
\begin{equation}\label{stab:eq:flow-estimates}
 P_T(\Phi_{j,t_j}F_j)-P_T(F_j)\le C|e_j|,
 \qquad |\Phi_{j,t_j}F_j\triangle F_j|\le C|e_j|.
\end{equation}
By \eqref{stab:eq:flow-estimates} and the triangle inequality for
$d_{\rm tr}$ as a function of its first argument,
\[
 0\le\mathcal J_j(\Phi_{j,t_j}F_j)-\mathcal J_j(F_j)
 \le(C_P+\Lambda_1C_\Delta-\Lambda_2)|e_j|.
\]
The constants $C_P,C_\Delta$ in the two bounds of
\eqref{stab:eq:flow-estimates} are independent of $\Lambda_2$.
Taking $\Lambda_2>C_P+\Lambda_1C_\Delta$ gives $e_j=0$.

For any finite-perimeter $F,G$, the unweighted penalties satisfy
\begin{equation}\label{stab:eq:penalty-lipschitz}
 |d_{\rm tr}(F,K)-d_{\rm tr}(G,K)|\le|F\triangle G|,
 \qquad\bigl||F|-|G|\bigr|\le|F\triangle G|.
\end{equation}
The outer absolute values in \eqref{stab:eq:selection} preserve these Lipschitz constants. Minimality thus gives
\begin{equation}\label{stab:eq:selected-quasiminimality}
 P_T(F_j)\leq P_T(G)+(\Lambda_1+\Lambda_2)|F_j\triangle G|.
\end{equation}
Thus \eqref{stab:eq:quasiminimality} holds with
$\Lambda=\Lambda_1+\Lambda_2$. The convergence $F_j\to K_{s_*}$
and the uniform perimeter bound then give
\eqref{stab:eq:normal-convergence} with $F=K_{s_*}$.
Since $\partial K_{s_j}\to\partial K_{s_*}$ smoothly, the
$\partial F_j$ are normal graphs with heights tending to zero in
$C^1$ over $\partial K_{s_j}$. Apply \eqref{stab:eq:graph-lower}
after the translation determined by \eqref{stab:eq:translation-jacobian}.
With \eqref{stab:eq:selected-upper} this gives
$c_1d_j^2\le P_T(F_j)-P_T(K_{s_j})\le\epsilon_j^2/j$.
But \eqref{stab:eq:comparable} gives
$\liminf_j d_j/\epsilon_j\ge\Lambda_1/(\Lambda_1+D)>0$. This proves the uniform assertion on the
period-two torus.
\end{proof}

\subsection{Reflection-invariant positions and global quantitative stability}
We compare $d_{\rm tr}(A,K)$ with the distance to those translates
of $K$ that are invariant under the three coordinate reflections.

\begin{lemma}\label{stab:lem:pinning}
Let $K$ be a reflected model centered at the origin, and let
$A\subset T$ be invariant under all coordinate reflections.
Let $\mathcal O_{\rm ref}(K)$ be the finite family of its distinct
reflection-invariant translates. Then
\begin{equation}\label{stab:eq:pinning}
 \min_{K'\in\mathcal O_{\rm ref}(K)}|A\triangle K'|
 \leq C_K d_{\rm tr}(A,K),\qquad
 C_K=1+\frac{P_T(K)}2
       \left(\sum_{i\in\mathcal I_K}k_i^{-2}\right)^{1/2},
\end{equation}
where $\mathcal I_K$ is the set of effective translation directions
and $k_i=\int_K\cos(\pi x_i)\,dx>0$ for those directions.
The constants are bounded on compact model families in the ranges
of Lemma~\ref{stab:lem:uniform-torus}.
\end{lemma}
\begin{proof}
The effective directions are all three for a ball, the two transverse
directions for a cylinder, and the normal direction for a slab.
For a ball, with $z=\pi r<\pi$,
\begin{equation}\label{stab:eq:ball-coefficient}
 k_i=\frac4{\pi^2}(\sin z-z\cos z)>0,
\end{equation}
because the parenthesized expression has derivative $z\sin z>0$.
For a cylinder the transverse slice at $x_i=s$ has area
$4\sqrt{r^2-s^2}$. Integration by parts gives
\begin{equation}\label{stab:eq:cylinder-coefficient}
 k_i=\frac8\pi\int_0^r
          \frac{s\sin(\pi s)}{\sqrt{r^2-s^2}}\,ds>0,
 \qquad r<1.
\end{equation}
For a slab $|x_j|<v$, one has $k_j=(8/\pi)\sin(\pi v)>0$.

For $i\in\mathcal I_K$ and $\eta=|A\triangle(K+t)|$, evenness of
$\chi_A$ and $\chi_K$ in $x_i$ gives
\[
 |k_i\sin(\pi t_i)|=
 \left|\int_T(\chi_{K+t}-\chi_A)\sin(\pi x_i)\,dx\right|\le\eta .
\] Choose $q_i\in\mathbb Z$
nearest to $t_i$, modulo two. With $d_i=|t_i-q_i|\leq1/2$,
$\sin(\pi d_i)\geq2d_i$, so $d_i\leq\eta/(2k_i)$.
Set $t_i=q_i=0$ in all ineffective directions, which leave $K$ unchanged. Thus the norm of $t-q$ contains only effective coordinates. The BV translation inequality yields
\begin{equation}\label{stab:eq:active-translation}
 |(K+t)\triangle(K+q)|\leq P_T(K)|t-q|
 \leq\frac{P_T(K)}2
      \left(\sum_{i\in\mathcal I_K}k_i^{-2}\right)^{1/2}\eta.
\end{equation}
The translate $K+q$ is reflection invariant and descends to a cube
model. The triangle inequality and minimization over $t$ prove
\eqref{stab:eq:pinning}. These are all reflection-invariant translates:
for such a translate the same sine coefficients force the effective
coordinates to be integral. Positivity and continuity of the $k_i$
give the asserted compact-family bound.
\end{proof}

\begin{proposition}\label{stab:prop:local}
For corner balls with $0<r<1$, edge cylinders with $1/\pi<r<1$,
and coordinate slabs of volume in $(0,1)$, there are $\rho,\kappa>0$
such that every equal-volume finite-perimeter competitor whose
distance $d$ to the finite model family is less than $\rho$ satisfies
\begin{equation}\label{stab:eq:box-local}
 \Per(E;Q)-\Per(M;Q)\geq\kappa d^2.
\end{equation}
The constants are uniform on compact subfamilies of these ranges.
No global minimality of the reference model is assumed.
\end{proposition}
\begin{proof}
For each orientation $o$ let $\mathcal F_o$ be its finite cubical
position family, and put $d_o=\min_{M\in\mathcal F_o}|E\triangle M|$.
Choose $o$ with $d=d_o$ and let $K_o$ be the reflected model of
orientation $o$ centered at the origin. With $A=\widehat E$,
\eqref{stab:eq:reflection} and \eqref{stab:eq:pinning} give
\begin{equation}\label{stab:eq:orientation-distance}
 8d=8d_o
 =\min_{K'\in\mathcal O_{\rm ref}(K_o)}|A\triangle K'|
 \le C_{K_o}d_{\rm tr}(A,K_o),\qquad
 d_{\rm tr}(A,K_o)\le8d.
\end{equation}
If $8d<\varepsilon$ from Lemma~\ref{stab:lem:uniform-torus}, then
\begin{equation}\label{stab:eq:local-transfer}
 8\{\Per(E;Q)-\Per(M;Q)\}
 \ge c\,d_{\rm tr}(A,K_o)^2
 \ge\frac{64c}{C_{K_o}^2}d^2.
\end{equation}
Taking $\rho=\varepsilon/8$ and
$\kappa=\min_o8c_o/C_{K_o}^2$ proves the result.
There are finitely many orientations, and the bounds are uniform
on compact parameter families.
\end{proof}

\begin{proof}[Proof of Theorem~\ref{stab:thm:global}]
Complementation reduces the fixed-volume assertion to $v\leq1/2$.
The finite classification and Proposition~\ref{stab:prop:local}
give $\rho,\kappa>0$ such that \eqref{intro:stability-quadratic} holds when
$\alpha_v(E)<\rho$, with constant $\kappa$. Define
\[
 \eta_v(\rho)=\inf\{\delta(E):E\in\mathcal A_v,\ \alpha_v(E)\ge\rho\}
 \quad(\inf\varnothing=+\infty).
\]
Then $\eta_v(\rho)>0$. Otherwise equal-volume sets $E_j$ could satisfy
$\alpha_v(E_j)\geq\rho$ and $\delta(E_j)\to0$. BV compactness in
the bounded Lipschitz cube gives an $L^1$ subsequential limit
$E_\infty$ of volume $v$. Lower semicontinuity gives
\begin{equation}\label{eq:body-272}
 I(v)\leq\Per(E_\infty;Q)
 \leq\liminf_j\Per(E_j;Q)=I(v).
\end{equation}
By Theorem~\ref{thm:main}, $E_\infty\in\mathcal M_v$, so
$\rho\le\alpha_v(E_j)\le|E_j\triangle E_\infty|\to0$, a contradiction. As $\alpha_v(E)\le2\min(v,1-v)\le1$,
$C_*(v)\ge\min\{\kappa,\eta_v(\rho)\}>0$.

For sharpness fix $M\in\mathcal M_v$ and put $K=\widehat M$.
Choose the following function $\phi$ on $\partial K$, using coordinates
centred on the model:
\[
 \phi=
 \begin{cases}
 (y_1^2-y_2^2)/r^2 &\text{on a sphere of radius }r,\\
 \cos(\pi s) &\text{on a cylinder with axial coordinate }s,\\
 \cos(\pi y_1) &\text{on both slab tori with tangential coordinate }y_1 .
 \end{cases}
\]
These functions are reflection invariant, lie in $X_K$ and satisfy
$\mathcal Q_K(\phi)>0$ by \eqref{stab:eq:coercivity}. For
$V(t,s)=|K_{t\phi+s}|-|K|$, \eqref{stab:eq:graph-volume} gives
\[
 \partial_tV(0,0)=\int_{\partial K}\phi\,dA=0,\qquad
 \partial_sV(0,0)=\mathcal H^2(\partial K)>0 .
\]
The implicit function theorem gives $s(t)=O(t^2)$ with
$V(t,s(t))=0$. The graph of $u_t=t\phi+s(t)$ is reflection invariant.
Let $E_t=K_{u_t}\cap Q$ and $\Gamma=\partial M\cap Q$.
Write $\mathcal Q_\Gamma(\phi)$ for the integral in
\eqref{stab:eq:form} over $\Gamma$ with physical area. Then
\begin{equation}\label{stab:eq:optimal-exponent-expansion}
 \delta(E_t)=\frac{t^2}{2}\mathcal Q_\Gamma(\phi)+o(t^2),\qquad
 |E_t\triangle M|=|t|\int_\Gamma|\phi|+o(|t|).
\end{equation}
The first expansion is \eqref{stab:eq:graph-expansion} divided by
eight. In the second, the symmetric-difference volume on $T$ is
$\int_{\partial K}|\int_0^{u_t(x)}J_{\rm vol}(r)\,dr|\,dA$ by
\eqref{stab:eq:jacobians}. For small $t$,
$2|E_t\triangle M|<\min_{M'\in\mathcal M_v,\,M'\ne M}|M'\triangle M|$,
so $\alpha_v(E_t)=|E_t\triangle M|$ by the triangle inequality.
Since $\int_\Gamma|\phi|>0$, the ratio
$\delta(E_t)/\alpha_v(E_t)^p=O(|t|^{2-p})$ tends to zero for $p<2$.

For compact $J$ in one open phase, the radii or slab widths for
$v\in J$ range in a compact subset of the admissible parameters
of Proposition~\ref{stab:prop:local}, so take common $\rho,\kappa$.
If $\inf_{v\in J}\eta_v(\rho)=0$, choose $v_j\to v_\infty\in J$
and $E_j$ with $|E_j|=v_j$, $\alpha_{v_j}(E_j)\ge\rho$ and
$\delta(E_j)\to0$. BV compactness gives a subsequential $L^1$ limit
$E_\infty$ with $|E_\infty|=v_\infty$ and
\[
 \Per(E_\infty;Q)\le\liminf_j(I(v_j)+\delta(E_j))=I(v_\infty).
\]
Thus $M_\infty=E_\infty\in\mathcal M_{v_\infty}$. The model $M_j$
of the same type, position and orientation at $v_j$ is nested
with $M_\infty$, so $|M_j\triangle M_\infty|=|v_j-v_\infty|$.
Consequently
\begin{equation}\label{stab:eq:varying-volume-distance}
 \alpha_{v_j}(E_j)\le |E_j\triangle M_j|
 \le|E_j\triangle M_\infty|+|v_j-v_\infty|\longrightarrow0,
\end{equation}
contradicting $\alpha_{v_j}(E_j)\ge\rho$. The constants for
volumes in the complementary phases follow by $E\mapsto Q\setminus E$.
\end{proof}

\subsection{Uniform quantitative estimates for competing model families}
Fix $v_*\in\{a,b\}$ and a sufficiently small closed interval
$J$ around $v_*$, contained in $(0,1/2)$. Denote the adjacent model
families by $\mathcal F_1(v),\mathcal F_2(v)$ and their perimeters
by $p_1(v),p_2(v)$. Both branches extend smoothly to $J$ and remain
embedded and strictly stable. Define
\begin{equation}\label{stab:eq:phase-distances}
 d_i(E,v)=\min_{M\in\mathcal F_i(v)}|E\triangle M|,
 \qquad g_i(v)=p_i(v)-\min\{p_1(v),p_2(v)\}\geq0.
\end{equation}
By Theorem~\ref{thm:main}, $I(v)=\min(p_1(v),p_2(v))$ on $J$.

\begin{theorem}\label{stab:thm:multiwell}
After reducing $J$ there is $c_J>0$, independent of $v\in J$,
such that every finite-perimeter $E\subset Q$ of volume $v$ satisfies
\begin{equation}\label{stab:eq:multiwell}
 \delta(E)\geq c_J\min_{i=1,2}\{g_i(v)+d_i(E,v)^2\}.
\end{equation}
\end{theorem}
\begin{proof}
The two compact model families give uniform $\rho,\kappa>0$ in
Proposition~\ref{stab:prop:local}. If $d_i(E,v)<\rho$, then
\begin{equation}\label{stab:eq:competition-local}
 \delta(E)=\Per(E;Q)-p_i(v)+g_i(v)
 \geq\kappa d_i(E,v)^2+g_i(v),
\end{equation}
which gives \eqref{stab:eq:multiwell} with constant $\min(1,\kappa)$.
For the remaining competitors define
\[
 \eta_J=\inf\{\delta(E):v\in J,\ E\in\mathcal A_v,\
                         d_1(E,v),d_2(E,v)\ge\rho\}.
\]
If $\eta_J=0$, choose $v_j\to v_\infty\in J$ and $E_j$ in this
class with $\delta(E_j)\to0$. As in \eqref{eq:body-272}, now using
$I(v_j)\to I(v_\infty)$, an $L^1$ subsequential limit $E_\infty$
has perimeter $I(v_\infty)$. Choose $i$ with
$E_\infty\in\mathcal F_i(v_\infty)$ by Theorem~\ref{thm:main}.
The model $M_j\in\mathcal F_i(v_j)$ of the same position and
orientation satisfies $|M_j\triangle E_\infty|=|v_j-v_\infty|$,
so
\[
 \rho\le d_i(E_j,v_j)\le|E_j\triangle M_j|
 \le |E_j\triangle E_\infty|+|v_j-v_\infty|\longrightarrow0.
\]
Hence $\eta_J>0$, with $\inf\varnothing=+\infty$. Since $d_i\le1$,
the choice $c_J=\min\{1,\kappa,\eta_J/(1+\max_{i,v\in J}g_i(v))\}$
gives \eqref{stab:eq:multiwell} on both classes.
\end{proof}

\begin{corollary}\label{stab:cor:degeneration}
For $v_*\in\{a,b\}$ there are $c,C,\varepsilon>0$ such that
\begin{equation}\label{stab:eq:degeneration}
 c|v-v_*|\leq C_*(v)\leq C|v-v_*|,
 \qquad 0<|v-v_*|<\varepsilon.
\end{equation}
At $v=v_*$, with distance to both minimizing families,
$C_*(v_*)>0$.
\end{corollary}
\begin{proof}
At $a$ the two model families have $h=3$ and $h=9/4$, respectively.
At $b$ they have $h=\pi/2$ and $h=0$. Therefore the two finite
families at $v_*$ are disjoint modulo null sets. At a fixed position
and orientation a model changes by symmetric-difference volume
$|v-v_*|$ when its volume changes from $v_*$ to $v$. By the triangle
inequality,
\begin{equation}\label{stab:eq:transition-separation}
 s_*:=\min_{M_i\in\mathcal F_i(v_*)}|M_1\triangle M_2|>0,
 \quad\min_{M_i\in\mathcal F_i(v)}|M_1\triangle M_2|
 \ge s_*-2|v-v_*|\ge s_*/2
\end{equation}
after shrinking $J$. For $v\ne v_*$ set $g(v)=|p_1(v)-p_2(v)|$.
Choose $i$ with $p_i(v)>I(v)$ and take $M_v\in\mathcal F_i(v)$.
Then $\delta(M_v)=g(v)$ and \eqref{stab:eq:transition-separation}
gives $\alpha_v(M_v)\ge s_*/2$. It follows that
$C_*(v)\le 4g(v)/s_*^2$.
For the lower bound, let $i_0$ be the minimizing index.
Since $d_{i_0}(E,v)=\alpha_v(E)\le1$ and $0\le g(v)\le1$,
\begin{equation}\label{stab:eq:best-constant-lower}
 \delta(E)\ge c_J\min\{\alpha_v(E)^2,g(v)+d_{3-i_0}(E,v)^2\}
 \ge c_J g(v)\alpha_v(E)^2 .
\end{equation}
Thus
\begin{equation}\label{stab:eq:best-constant-comparison}
 c_J g(v)\le C_*(v)\le 4g(v)/s_*^2 .
\end{equation}
Differentiating \eqref{stab:eq:branches},
$p_B'(a)=3$, $p_C'(a)=9/4$, $p_C'(b)=\pi/2$ and $p_S'(b)=0$.
Consequently
\begin{align}
 |p_B(v)-p_C(v)|&=\frac34|v-a|+O(|v-a|^2),\label{stab:eq:gapa}\\
 |p_C(v)-1|&=\frac\pi2|v-b|+O(|v-b|^2).\label{stab:eq:gapb}
\end{align}
Let $\gamma_*=3/4$ if $v_*=a$ and $\gamma_*=\pi/2$ if $v_*=b$.
By \eqref{stab:eq:gapa} and \eqref{stab:eq:gapb}, shrinking $J$ gives
$\gamma_*|v-v_*|/2\le g(v)\le3\gamma_*|v-v_*|/2$.
Combining with \eqref{stab:eq:best-constant-comparison} yields
\[
 \frac{c_J\gamma_*}{2}|v-v_*|\le C_*(v)\le
        \frac{6\gamma_*}{s_*^2}|v-v_*| \qquad(v\ne v_*).
\]
At $v_*$, $g_1(v_*)=g_2(v_*)=0$ and
$\alpha_{v_*}(E)=\min_i d_i(E,v_*)$, so \eqref{stab:eq:multiwell}
gives $C_*(v_*)\ge c_J$. The model $M_v$ in the upper estimate,
continued in the same position and family as $v\to v_*$ from
one side, satisfies $|M_v\triangle M_{v_*}|=|v-v_*|$ and
$\alpha_{v_*}(M_{v_*})=0$.
\end{proof}

\begin{remark}\label{stab:rem:cuboid-stability}
The conclusions of Theorem~\ref{stab:thm:global},
Theorem~\ref{stab:thm:multiwell} and
Corollary~\ref{stab:cor:degeneration} extend to every fixed
$Q_\beta=(0,\beta)\times(0,1)^2$, $0<\beta\le1$.
Let $\mathcal M_{\beta,v}$ be the full minimizing family at
$|E|=\beta v$, and set
\begin{align}
 \delta_\beta(E)&=\frac{\Per(E;Q_\beta)}\beta-I_\beta(v),
 \qquad
 \alpha_{\beta,v}(E)=\frac1\beta
       \min_{M\in\mathcal M_{\beta,v}}|E\triangle M|,
 \label{cub:eq:deficit-distance}\\
 C_{*,\beta}(v)&=\inf_{\substack{E\subset Q_\beta,\ |E|=\beta v\\
                                  \alpha_{\beta,v}(E)>0}}
       \frac{\delta_\beta(E)}{\alpha_{\beta,v}(E)^2}.
 \label{cub:eq:best-constant}
\end{align}
All competitors in this remark have finite perimeter. For every
$v\in(0,1)$,
\begin{gather}
 0<C_{*,\beta}(v)<\infty,\qquad
 \delta_\beta(E)\ge C_{*,\beta}(v)\alpha_{\beta,v}(E)^2,
 \label{cub:eq:global}\\
 \inf_{\substack{E\subset Q_\beta,\ |E|=\beta v\\
                         \alpha_{\beta,v}(E)>0}}
       \frac{\delta_\beta(E)}{\alpha_{\beta,v}(E)^p}=0
       \qquad(p<2).
 \label{cub:eq:optimal-exponent}
\end{gather}
With $a_\beta=4\pi\beta^2/81$ and $b=1/\pi$, the five open phases are
\begin{equation}\label{cub:eq:open-phases}
 (0,a_\beta),\quad(a_\beta,b),\quad(b,1-b),\quad
 (1-b,1-a_\beta),\quad(1-a_\beta,1).
\end{equation}
For every nonempty compact subset $J$ of one interval in
\eqref{cub:eq:open-phases},
$\inf_{v\in J}C_{*,\beta}(v)>0$.

For $v_*=a_\beta$, let $\mathcal F_{\beta,1}(v)$ be all corner
balls of radius $(6\beta v/\pi)^{1/3}$ and
$\mathcal F_{\beta,2}(v)$ all quarter-cylinders of radius
$2\sqrt{v/\pi}$ along edges of length $\beta$.
For $v_*=b$, take these cylinders and the slabs of width $v$
normal to the unit-length directions, respectively.
When $\beta=1$, all coordinate directions are included.
Their normalized perimeters are, respectively,
$(9\pi/(2\beta))^{1/3}v^{2/3}$, $\sqrt{\pi v}$ and $1$.
Denote the adjacent perimeters by $p_{\beta,1},p_{\beta,2}$ and define
\begin{align}
 d_{\beta,i}(E,v)&=\frac1\beta
             \min_{M\in\mathcal F_{\beta,i}(v)}|E\triangle M|,
 \label{cub:eq:family-distance}\\
 g_{\beta,i}(v)&=p_{\beta,i}(v)
                     -\min\{p_{\beta,1}(v),p_{\beta,2}(v)\}.
 \label{cub:eq:family-gap}
\end{align}
On a sufficiently small closed neighbourhood $J\subset(0,1/2)$ of
$v_*$ there is $c_{\beta,J}>0$ such that
\begin{equation}\label{cub:eq:competition}
 \delta_\beta(E)\ge c_{\beta,J}\min_{i=1,2}
       \{g_{\beta,i}(v)+d_{\beta,i}(E,v)^2\}
       \qquad(v\in J).
\end{equation}
Moreover, for some $c,C,\varepsilon>0$ depending on $\beta,v_*$,
\begin{equation}\label{cub:eq:degeneration}
 c|v-v_*|\le C_{*,\beta}(v)\le C|v-v_*|
 \quad(0<|v-v_*|<\varepsilon),\qquad C_{*,\beta}(v_*)>0.
\end{equation}
At $v_*$ the distance in \eqref{cub:eq:deficit-distance} is to
$\mathcal M_{\beta,v_*}
=\mathcal F_{\beta,1}(v_*)\cup\mathcal F_{\beta,2}(v_*)$.
Complementation gives the assertions at $1-b$ and $1-a_\beta$.
The constants may depend on $\beta$; no uniformity as $\beta\uparrow1$
is asserted.
\end{remark}

\appendix
\section{Rational bounds for the curvature intervals}
\label{app:wall-comparisons}

The bounds here are used in Lemmas~\ref{t:projection-lines},
\ref{t:joint-concavity} and Propositions~\ref{t:low-curvature},
\ref{t:high-curvature}. Throughout the appendix, put $a=H_0=5881/5000$, $b=3/2$,
$L=314159/100000$ and $U=3927/1250$.

\subsection{Rational enclosures and supporting planes}
\label{app:rational-rules}

Machin's identity $\pi=16\arctan(1/5)-4\arctan(1/239)$ and the
alternating series give
\begin{equation}\label{app:fine-pi}
 3.141592653589<\pi<3.141592653590.
\end{equation}
Writing $T_n(x)=\sum_{j=0}^{n-1}(-1)^jx^{2j+1}/(2j+1)$ in
this comparison, the lower bound is less than
$16T_{12}(1/5)-4T_5(1/239)$ and the upper bound exceeds
$16T_{13}(1/5)-4T_4(1/239)$.
For $0\leq t<13/10$, write
$S_n(t)=\sum_{j=0}^{n-1}(-1)^jt^{2j+1}/(2j+1)!$ and
$C_n(t)=\sum_{j=0}^{n-1}(-1)^jt^{2j}/(2j)!$.
Then
$S_{12}(t)\leq\sin t\leq S_{11}(t)$ and
$C_{12}(t)\leq\cos t\leq C_{11}(t)$.
To enclose a positive square root on a rational grid, choose the integer
$k$ satisfying $k^2\leq10^{40}r<(k+1)^2$; then
$k/10^{20}\leq\sqrt r<(k+1)/10^{20}$.

The primitive and arcsine values can also be enclosed by rational sums.
Set
\[
 A_n(r)=r-\sum_{j=1}^n
 \frac{\binom{2j}{j}r^{2j+1}}{4^j(2j-1)(2j+1)},\qquad
 B_n(r)=\sum_{j=0}^n\frac{\binom{2j}{j}r^{2j+1}}{4^j(2j+1)}.
\]
For $0\leq r<1$, the decreasing positive coefficients imply
\begin{align*}
 A_n(r)-\frac{\binom{2n+2}{n+1}r^{2n+3}}
 {4^{n+1}(2n+1)(2n+3)(1-r^2)}&\leq F(r)\leq A_n(r),\\
 B_n(r)\leq\arcsin r&\leq B_n(r)+
 \frac{\binom{2n+2}{n+1}r^{2n+3}}{4^{n+1}(2n+3)(1-r^2)}.
\end{align*}
Below one may use $n=20$ for $F$ and $n=12$ for $\arcsin$.

All supporting planes in \eqref{t:projection-mean-planes} and
\eqref{t:cost-planes} are obtained from the finite six-term recurrence
of Lemma~\ref{s:cap-polynomial}, with the listed rational $(r,R)$,
and \eqref{s:mean-plane}, \eqref{s:cost-differentials},
\eqref{s:cost-plane}. In the mean table, each listed coefficient
exceeds the corresponding exact coefficient; in the cost table, each
listed coefficient is smaller. These are fifteen rational coefficient
comparisons, with denominators positive by
\eqref{s:finite-canonical-derivative}. They establish the required
global supports by Lemma~\ref{s:support}.

\subsection{Projection and joint-concavity comparisons}
\label{app:projection-comparisons}\label{app:curve-comparisons}

The polynomials in \eqref{t:projection-polynomials} have rational
coefficients. Besides \eqref{t:projection-point-comparisons}, their
signs used in Lemma~\ref{t:projection-lines} are
\[
 \rho_j''(b)>0,\quad Q_j'(b)>0,\quad Q_j(a)>0\quad(j=0,1),
 \qquad D_0''''>0,\quad D_0'''(b)<0,\quad D_0''(b)>2.
\]
For $D_1''=Ah^2+Bh+C$, $A>0$ and $C-B^2/(4A)>19/100$.
With the derivative bounds in \eqref{t:projection-point-comparisons},
the two completed-square remainders are positive:
\[
 \frac1{10000}-\frac{(3/1000)^2}{4}>0,\qquad
 \frac1{10000}-\frac{(1/2000)^2}{2(19/100)}>0.
\]
Each comparison follows by substitution in the displayed rational
polynomials and multiplication by their positive denominators.

For Lemma~\ref{t:joint-concavity}, the enclosures
$0.47139<\sin(5\pi/32)<0.47140$,
$0.88192<\cos(5\pi/32)<0.88193$, and $u_*<0.148$
give $3/40<r_{1,2}<2/5$. The endpoint comparisons for
$(1-r/b)/(1-r^2)^{3/2}<24/25$, at $r=3/40,2/5$, follow by
squaring positive quantities.
For $w=4(\sin2t+\cos2t)/d(t)$ in
\eqref{gm:two-curvature-bounds}, the bounds used are
\begin{gather*}
 -0.67<\cos2t-\sin2t<0.14,\quad
 \sin2t+\cos2t>31/25,\quad Q>32/5,\\
 5.72<A<7.34,\quad A^2/Q<28/5,
 \quad 7<Q(\pi/8)<176/25,\quad Q(t_0)>8.
\end{gather*}
They follow from the trigonometric enclosures in
Appendix~\ref{app:rational-rules} and the derivative of
$(28/5)Q-A^2$ displayed in the proof of Lemma~\ref{t:joint-concavity}.
The final signs reduce to
\[
 \frac{5.72}{7.04}-\frac{2(0.14)}{1.24}>0,\qquad
 -28+\frac{4(0.67)(7.34)}{1.24}+3(5.6)-4(1.24)<0.
\]

For $R_0$ and $Z$ in \eqref{t:common-source-normal}, the
endpoint bounds at $\delta=0$ are
\begin{equation}\label{app:zero-endpoint-comparisons}
\begin{array}{c|cc}
 h&a&b\\\hline
 R_0(h)&>9623/10000&>9007/10000\\
 d(t_0)A_h'/\sqrt{A_h}-Z'(h)&>-1883/1000&>-1656/1000.
\end{array}
\end{equation}
Also $R_0''(b)<0$, $T_\eta/h>11/25$, $A_h'(a)>3/25$,
$A_h(b)<9/25$, $d(t_0)>2.8736$, $Z''''<0$,
$Z'''(b)>-3/2$ and $[\sin(t_0)R_0]'(b)>-3/20$.
The angular derivative bound in the proof of Proposition~\ref{t:low-curvature}
uses
\[
 \frac{27}{500}\left\{\frac{4b^4}{L(19/5)}+
                   \frac{b^5}{4(19/5)^2}\right\}<1/10.
\]
Finally the one required function value satisfies
$f_0(a)>5675/10^6$. These strict comparisons have rational enclosures;
the interval does not need to be subdivided.

\subsection{A polynomial lower bound on the projection boundary}
\label{app:body-polynomial}

At $u=u_*$ the angles are $t_x=7\pi/64$, $t_y=13\pi/64$. Write $D_x=d(t_x)$ and $D_y=d(t_y)$.
For fixed $h$, let $\Psi_h(w)$ be the expression \eqref{t:low-residual} at $u=u_*$ with $\delta=2-3h/\pi-2w/h$. Thus $f_+(h)=\Psi_h(F(\ell_1(h)))$. The identity \eqref{t:body-pressure-cancellation} gives
\[
 p_x=\sigma_x^2=\pi/4-2h+3h^2/\pi+w,\qquad
 p_y=\sigma_y^2=\pi/4-w.
\]
Differentiating \eqref{t:low-residual} gives
\begin{equation}\label{t:asymmetry-substitution-derivative}
 \Psi_h'(w)=\frac{D_x}{2\sqrt{p_x}}-
 \frac{D_y}{2\sqrt{p_y}}-
 \frac{2s_xh}{15\pi}-\frac{2k_0(h)}h,
 \quad k_0(h)=\frac{\sqrt2h^4(81/40-4h/\pi)}
                         {9K_*(\pi+h^2/2)}.
\end{equation}
The exact bounds
\[
 D_x>5/2,\quad D_y<13/4,\quad \sqrt{p_x}<11/25,
 \quad\sqrt{p_y}>7/12,\quad s_x<7/20,\quad k_0<3/200
\]
show that this derivative is positive, since
\[
 \Psi_h'(w)>
 \frac{17}{308}-\frac{7b}{150L}-\frac{3}{100a}>0.
\]
At $w=F(\ell_1(h))$, $p_x''=6/\pi-(37/50)^2\ell_1/\sqrt{1-\ell_1^2}>0$
and $p_y'=(37/50)\sqrt{1-\ell_1^2}>0$. The endpoint bounds are
$\max\{p_x(a),p_x(b)\}<(11/25)^2$ and $p_y(a)>(7/12)^2$. For the angular bound, $h^4(81/40-4h/\pi)$ has maximum
$(81/200)(81\pi/200)^4$; replace $\pi$ by $U$ in this numerator
and by $L$ in the positive denominator.

We may therefore lower $w$ to
\[
 \underline F(\ell_1),\qquad
 \underline F(r)=r-r^3/6-r^5/30\leq F(r)\quad(0\leq r\leq1/2).
\]
The primitive bound follows by bounding the tail beginning with $r^5/40$
by $r^5/[40(1-r^2)]\leq r^5/30$.
Set $\widetilde p_i=p_i|_{w=\underline F(\ell_1)}$. Then
$\widetilde p_x''>8/5$.
at $c=7/5$, $|\widetilde p_x'(c)|<7/200$ and
\[
 \widetilde p_x(c)-\frac{(7/200)^2}{2(8/5)}>(3/8)^2.
\]
Hence $\sqrt{\widetilde p_x}>3/8$ throughout the interval.
Also $\widetilde p_y\geq p_y>(7/12)^2$.
These bounds justify differentiation on the whole segment $\underline F(\ell_1)\leq w\leq F(\ell_1)$. Along it $p_x(w)\leq p_x(F(\ell_1))$ and $p_y(w)\geq p_y(F(\ell_1))$, so the same strict lower bound for $\Psi_h'(w)$ applies throughout.

For the two square roots, if
$\sqrt p\geq r>0$ and $s>0$, then
\begin{equation}\label{ac:sqrt-minorant}
 M_{s,r}(p):=s+\frac{p-s^2}{2s}-
               \frac{(p-s^2)^2}{2s(s+r)^2}\leq\sqrt p.
\end{equation}
Indeed $s+(p-s^2)/(2s)-\sqrt p=(p-s^2)^2/[2s(s+\sqrt p)^2]$.
Use $(s,r)=(3/7,3/8)$ for $p_x$ and $(3/5,7/12)$ for $p_y$.
For $0\le r\le1/2$ define $C_h^-(r)$ by
\begin{equation}\label{ac:calibration-minorant}
 C_h^-(r)=h(1-r^2/2-r^4/8-r^6/12)+r^3/6+3r^5/40+5r^7/112
 \le C_h(r).
\end{equation}
This uses the positive arcsine series and
$\sqrt{1-r^2}\geq1-r^2/2-r^4/8-r^6/[16(1-r^2)]$.
Both $r_{1,2}$ lie in $(0,1/2)$ on $[a,b]$.
Put $\delta_-(h)=2-3h/\pi-2\underline F(\ell_1(h))/h$. Then
\begin{align*}
 \underline f(h)={}&C_h^-(r_1)+C_h^-(r_2)
  +s\left[\frac{3h}{2\pi}-\frac{\eta-6007/10000}{h}\right]\\
 &+s_x R_\sharp(h,\delta_-)+(s_y-s_x)\ell_3
  +D_x M_{3/7,3/8}(\widetilde p_x)+D_y M_{3/5,7/12}(\widetilde p_y)\\
 &-Z-\frac{13\pi}{8}+G(h,\delta_-).
\end{align*}
Here the angles are $(7\pi/64,13\pi/64)$, and $R_\sharp,G$ are
given in \eqref{t:shared-angular-support}. By
\eqref{t:asymmetry-substitution-derivative}, \eqref{ac:sqrt-minorant},
\eqref{ac:calibration-minorant} and $F(\eta)>6007/10000$,
$\underline f(h)\le\Psi_h(\underline F(\ell_1(h)))\le f_+(h)$.

The product $h^7(\pi+h^2/2)\underline f(h)$ is a polynomial of
degree nineteen. Let $I_j$ be the interval for its $h^j$ coefficient
obtained by the rational rules of Appendix~\ref{app:rational-rules},
rounding arithmetic endpoints outwards to $10^{-30}$. Define
\[
 P(h)=\sum_{j=0}^{19}10^{-10}\lfloor10^{10}\inf I_j\rfloor h^j
 \le h^7(\pi+h^2/2)\underline f(h)\qquad(h>0).
\]
The coefficient integers $10^{10}[h^j]P$ are, in increasing order,
\begin{align*}
 (&-659095,\ 7097944,\ -51500268,\ 234724914,\ -1018219619,\\
 &2909715689,\ -8868399755,\ -64743277305,\ 186404275462,\\
 &-100241803400,\ -101368358330,\ 124844779689,\ -43065427418,\\
 &2565927252,\ 6028739742,\ -4884983441,\ 1575597023,\\
 &-300128073,\ 31832917,\ -1777839).
\end{align*}

\begin{lemma}\label{ac:polynomial}
The preceding polynomial satisfies $P(h)>1/100$ on $[a,b]$.
Consequently $f_+(h)>1/10000$ there.
\end{lemma}
\begin{proof}
Put $x=h-6/5\in(-1/40,3/10]$ and
$P(6/5+x)=\sum_{j=0}^{19}p_jx^j$.
Only the following seven rational coefficient comparisons are needed:
\begin{gather*}
 p_0>3/250,\qquad |p_1|<1/8,\qquad p_2>21/8,\\
 73/4<p_3<19,\qquad p_4>29,\qquad
 \sum_{j=5}^{19}|p_j|(3/10)^{j-5}<320.
\end{gather*}
For $x\geq0$ they give
$P\geq p_0+p_1x+x^2B(x)$, with
\[
 B(x)=21/8+(73/4)x+29x^2-320x^3.
\]
If $M=3/10$, subtracting the chord of $B$ gives exactly
\[
 B(x)-\operatorname{chord}_{[0,M]}B(x)
       =x(M-x)[320(x+M)-29]\geq0.
\]
Since $B(0)=21/8>2$ and $B(M)=207/100>2$, we obtain $B>2$.
For $-1/40<x<0$, $p_3x^3\ge-19|x|^3$, $p_4x^4\ge0$ and
$\bigl|\sum_{j=5}^{19}p_jx^j\bigr|\le320|x|^5$. Consequently
\[
 P(6/5+x)\ge p_0+p_1x+x^2(21/8-19/40-320/40^3)
             =p_0+p_1x+(429/200)x^2 .
\]
Thus, on both sides,
\[
 P\geq p_0+p_1x+2x^2\geq p_0-p_1^2/8
          >3/250-1/512>1/100.
\]
Finally $h^7(\pi+h^2/2)\leq b^7(U+b^2/2)<100$.
The asserted bound for $f_+$ follows.
\end{proof}

\subsection{Comparisons for the elementary high-curvature argument}
\label{app:high-comparisons}

For the notation of Proposition~\ref{t:high-curvature}, put
$z=h/\pi\leq(15/8)/L<3/5$. Then
\[
 S'=2-8z+6z^2+\tfrac43z^3+\tfrac12z^4
       \geq2-8z+6z^2\geq-16/25.
\]
This gives $S'/\sqrt2>-16/35$. The other uniform comparisons are
$(2C)'=(h+1)/C\geq1+1/h\geq23/15$ and
$(\sqrt2T)'>147/220$.
For the derivative at each endpoint, use
$p_\pm(h)\le U/4-(1\pm\delta/2)h+3h^2/(2L)$ and
$p_\pm'(h)\ge3h/U-1\mp\delta/2>0$. With square roots rounded
outwards to $10^{-6}$ as in Appendix~\ref{app:rational-rules},
the expression
\[
 \sqrt L\sum_{\pm}\frac{3h/U-1\mp\delta/2}
                   {\sqrt{U/4-(1\pm\delta/2)h+3h^2/(2L)}}
 -\frac52\sqrt{2h^3/L}
\]
gives the following strict lower bounds for $D_\delta$:
\begin{equation}\label{app:high-derivative-comparisons}
\begin{array}{c|rr}
 &\delta=0&\delta=1/4\\\hline
 h=3/2&-111/100&-102/100\\
 h=15/8&-148/100&-136/100
\end{array}
\end{equation}
For \eqref{t:high-endpoints} at $h=3/2$, use
\begin{align*}
 T&\ge \frac{3h}{2U}-\frac{1-L/4}{h}>0,\\
 S&\ge 2h-\frac{4h^2}{L}+\frac{2h^3}{U^2}
                  +\frac{h^4}{3U^3}+\frac{h^5}{10U^4}>0,\\
 A_h&\ge L/4-h+3h^2/(2U).
\end{align*}
In $W$ substitute the lower bounds for $T,S$ above, use
$\pi\le U$ in $-2\pi$ and $\pi\ge L$ in $-h^2\sqrt{2h/\pi}$.
In $J$, use $\pi\ge L$ in the numerator of the first term
and $\pi\le U$ in the denominators of the other two.
With $\sqrt\pi\ge\sqrt L$ in the terms multiplying
$\sqrt{A_h}$ and $\sqrt{A_h\pm h/8}$,
rounding the square roots outwards to $10^{-6}$ gives
$B_0(3/2)>119/1000$ and $B_q(3/2)>50/1000$.

\section*{Acknowledgments}
\revision{We thank Prof.\ Xi-Nan Ma for bringing this problem to our attention.
We are grateful to Prof.\ Emanuel Milman, whose work on isoperimetric
inequalities on slabs motivated this project and provided the reduction to
the transition volumes. We thank Prof.\ Frank Morgan for sharing our
preprint with colleagues.}

This work was supported by the National Key R\&D Program of China (2025YFA1017603).

\section*{Declaration of AI use}
During the preparation of this manuscript, the authors used ChatGPT (OpenAI) as an auxiliary tool for language editing, exploration of some of the proof strategies and calculations. All mathematical statements, proofs, references, and conclusions appearing in the manuscript were independently verified by the authors, who take full responsibility for the content.

\medskip
{\small
\noindent
(Bin Deng) School of Mathematics and Statistics, Wuhan University, Wuhan,
Hubei 430072, P.R. China.
\revision{Email:} \href{mailto:dbmath@whu.edu.cn}{dbmath@whu.edu.cn}

\smallskip
\noindent
\revision{(Jiahuan Li and Yilu Liu)} School of Mathematical Sciences, University of
Science and Technology of China, Hefei, Anhui 230026, P.R. China.
\revision{Emails:} \href{mailto:jiahuan@mail.ustc.edu.cn}{jiahuan@mail.ustc.edu.cn},
\href{mailto:liuylgeoanaly@mail.ustc.edu.cn}{liuylgeoanaly@mail.ustc.edu.cn}
\par}


\begin{thebibliography}{99}
\bibitem{AFM} E. Acerbi, N. Fusco and M. Morini,
\emph{Minimality via second variation for a nonlocal isoperimetric problem},
Comm. Math. Phys. 322 (2013), 515--557,
doi: \href{https://doi.org/10.1007/s00220-013-1733-y}{10.1007/s00220-013-1733-y}.
\bibitem{AFP}
L. Ambrosio, N. Fusco and D. Pallara,
\emph{Functions of Bounded Variation and Free Discontinuity Problems},
Oxford Mathematical Monographs, Clarendon Press, Oxford, 2000.
\bibitem{BCF} M. Barchiesi, F. Cagnetti and N. Fusco,
\emph{Stability of the Steiner symmetrization of convex sets},
J. Eur. Math. Soc. 15 (2013), no. 4, 1245--1278.
\href{https://doi.org/10.4171/JEMS/391}{doi:10.4171/JEMS/391}.
\bibitem{BM} F. Barthe and B. Maurey,
\emph{Some remarks on isoperimetry of Gaussian type},
Ann. Inst. H. Poincar\'e Probab. Statist. 36 (2000), no. 4, 419--434,
doi: \href{https://doi.org/10.1016/S0246-0203(00)00131-X}{10.1016/S0246-0203(00)00131-X}.
\bibitem{BP}
C. Bavard and P. Pansu,
\emph{Sur le volume minimal de $\mathbb R^2$},
Ann. Sci. \'Ecole Norm. Sup. (4) \textbf{19} (1986), no. 4, 479--490,
\href{https://doi.org/10.24033/asens.1514}{doi:10.24033/asens.1514}.
\bibitem{BR} V. Bayle and C. Rosales,
\emph{Some isoperimetric comparison theorems for convex bodies in
Riemannian manifolds}, Indiana Univ. Math. J. 54 (2005),
no. 5, 1371--1394,
doi: \href{https://doi.org/10.1512/iumj.2005.54.2575}{10.1512/iumj.2005.54.2575}.
\bibitem{Berg}
C. Berg,
\emph{Corps convexes et potentiels sph\'eriques},
Mat.-Fys. Medd. Danske Vid. Selsk. \textbf{37} (1969), no. 6, 1--64.
\bibitem{BrezisOpen}
H. Brezis,
\emph{Some of my favorite open problems},
Rend. Lincei Mat. Appl. \textbf{34} (2023), no. 2, 307--335,
\href{https://doi.org/10.4171/RLM/1008}{doi:10.4171/RLM/1008}.
\bibitem{BrezisBruckstein}
H. Brezis and A. Bruckstein,
\emph{A sharp relative isoperimetric inequality for the square},
C. R. Math. Acad. Sci. Paris \textbf{359} (2021), no. 9, 1191--1199,
\href{https://doi.org/10.5802/crmath.243}{doi:10.5802/crmath.243}.
\bibitem{ChambersMouille}
G. R. Chambers and L. Mouill\'e,
\emph{On the relative isoperimetric problem for the cube},
Indiana Univ. Math. J., to appear; arXiv:2302.04382v2 (2024),
\url{https://arxiv.org/abs/2302.04382}.
\bibitem{CCF}
M. Chleb\'ik, A. Cianchi and N. Fusco,
\emph{The perimeter inequality under Steiner symmetrization: Cases of equality},
Ann. of Math. (2) \textbf{162} (2005), no. 1, 525--555,
\href{https://doi.org/10.4007/annals.2005.162.525}{doi:10.4007/annals.2005.162.525}.
\bibitem{CL} M. Cicalese and G. P. Leonardi,
\emph{A selection principle for the sharp quantitative isoperimetric
inequality}, Arch. Ration. Mech. Anal. 206 (2012), no. 2,
617--643,
doi: \href{https://doi.org/10.1007/s00205-012-0544-1}{10.1007/s00205-012-0544-1}.
\bibitem{Frensel}
K. R. Frensel,
\emph{Stable complete surfaces with constant mean curvature},
Bol. Soc. Brasil. Mat. (N.S.) \textbf{27} (1996), no. 2, 129--144.
\bibitem{FMP} N. Fusco, F. Maggi and A. Pratelli,
\emph{The sharp quantitative isoperimetric inequality},
Ann. of Math. (2) 168 (2008), no. 3, 941--980,
doi: \href{https://doi.org/10.4007/annals.2008.168.941}{10.4007/annals.2008.168.941}.
\bibitem{Glaudo} F. Glaudo,
\emph{On the isoperimetric profile of the hypercube},
Math. Ann. 390 (2024), 477--492,
doi: \href{https://doi.org/10.1007/s00208-023-02765-w}{10.1007/s00208-023-02765-w}.
\bibitem{Hadwiger} H. Hadwiger,
\emph{Gitterperiodische Punktmengen und Isoperimetrie},
Monatsh. Math. 76 (1972), 410--418.
\bibitem{HPRR}
L. Hauswirth, J. P\'erez, P. Romon and A. Ros,
\emph{The periodic isoperimetric problem},
Trans. Amer. Math. Soc. \textbf{356} (2004), no. 5, 2025--2047,
\href{https://doi.org/10.1090/S0002-9947-03-03362-2}{doi:10.1090/S0002-9947-03-03362-2}.
\bibitem{Hersch}
J. Hersch,
\emph{Quatre propri\'et\'es isop\'erim\'etriques de membranes sph\'eriques homog\`enes},
C. R. Acad. Sci. Paris S\'er. A-B \textbf{270} (1970), A1645--A1648.
\bibitem{HHM} H. Howards, M. Hutchings and F. Morgan,
\emph{The isoperimetric problem on surfaces},
Amer. Math. Monthly 106 (1999), no. 5, 430--439,
doi: \href{https://doi.org/10.1080/00029890.1999.12005065}{\nolinkurl{10.1080/00029890.1999.12005065}}.
\bibitem{Maggi}
F. Maggi,
\emph{Sets of Finite Perimeter and Geometric Variational Problems},
Cambridge University Press, 2012.
\bibitem{Milman}
E. Milman,
\emph{Isoperimetric inequalities on slabs with applications to cubes and Gaussian slabs},
Comm. Pure Appl. Math. \textbf{79} (2026), no. 4, 1012--1072,
\href{https://doi.org/10.1002/cpa.70020}{doi:10.1002/cpa.70020}.
\bibitem{Mityagin}
B. S. Mityagin,
\emph{The zero set of a real analytic function},
Math. Notes \textbf{107} (2020), 529--530,
\href{https://doi.org/10.1134/S0001434620030189}{doi:10.1134/S0001434620030189}.
\bibitem{Morrey}
C. B. Morrey, Jr.,
\emph{Multiple Integrals in the Calculus of Variations},
Grundlehren der mathematischen Wissenschaften 130, Springer, 1966.
\bibitem{MJ} F. Morgan and D. L. Johnson,
\emph{Some sharp isoperimetric theorems for Riemannian manifolds},
Indiana Univ. Math. J. 49 (2000), no. 3, 1017--1041,
doi: \href{https://doi.org/10.1512/iumj.2000.49.1929}{10.1512/iumj.2000.49.1929}.
\bibitem{MorganPoly} F. Morgan,
\emph{In polytopes, small balls about some vertex minimize perimeter},
J. Geom. Anal. 17 (2007), no. 1, 97--106,
\href{https://doi.org/10.1007/BF02922085}{doi:10.1007/BF02922085}.
\bibitem{OpenAI2026}
OpenAI,
\emph{The Isoperimetric Conjecture for the Cubic Flat Three-Torus},
preprint, September 24, 2026,
\href{https://github.com/openai/math/blob/main/preprints/The-Isoperimetric-Conjecture-for-the-Cubic-Flat-Three-Torus-September-24-2026/article.pdf}{OpenAI mathematics preprints}.
\bibitem{PR} R. H. L. Pedrosa and M. Ritor\'e,
\emph{Isoperimetric domains in the Riemannian product of a circle
with a simply connected space form and applications to free boundary
problems}, Indiana Univ. Math. J. 48 (1999), no. 4, 1357--1394,
doi: \href{https://doi.org/10.1512/iumj.1999.48.1614}{10.1512/iumj.1999.48.1614}.
\bibitem{RitoreRos}
M. Ritor\'e and A. Ros,
\emph{Stable constant mean curvature tori and the isoperimetric problem in three space forms},
Comment. Math. Helv. \textbf{67} (1992), 293--305,
\href{https://doi.org/10.1007/BF02566501}{doi:10.1007/BF02566501}.
\bibitem{RitoreThesis}
M. Ritor\'e,
\emph{Superficies con curvatura media constante},
doctoral thesis, Universidad de Granada, 1994, revised version,
\url{https://www.ugr.es/~ritore/preprints/tesis.pdf}.
\bibitem{Ros}
A. Ros,
\emph{The isoperimetric problem}, in \emph{Global Theory of Minimal Surfaces},
Clay Math. Proc. 2, American Mathematical Society, 2005, 175--209.
\bibitem{Ros2006}
A. Ros,
\emph{One-sided complete stable minimal surfaces},
J. Differential Geom. \textbf{74} (2006), no. 1, 69--92,
\href{https://doi.org/10.4310/jdg/1175266182}{doi:10.4310/jdg/1175266182}.
\bibitem{Ros2007Periodic}
A. Ros,
\emph{Stable periodic constant mean curvature surfaces and mesoscopic phase separation},
Interfaces Free Bound. \textbf{9} (2007), no. 3, 355--365,
\href{https://doi.org/10.4171/IFB/168}{doi:10.4171/IFB/168}.
\bibitem{Schneider}
R. Schneider,
\emph{Convex Bodies: The Brunn--Minkowski Theory}, second edition,
Cambridge University Press, 2014.
\bibitem{SZ} P. Sternberg and K. Zumbrun,
\emph{On the connectivity of boundaries of sets minimizing perimeter
subject to a volume constraint}, Comm. Anal. Geom. 7 (1999),
no. 1, 199--220,
doi: \href{https://doi.org/10.4310/CAG.1999.v7.n1.a7}{10.4310/CAG.1999.v7.n1.a7}.
\bibitem{YangYau}
P. C. Yang and S.-T. Yau,
\emph{Eigenvalues of the Laplacian of compact Riemann surfaces and minimal submanifolds},
Ann. Scuola Norm. Sup. Pisa Cl. Sci. (4) \textbf{7} (1980), no. 1, 55--63,
\href{https://www.numdam.org/item/ASNSP_1980_4_7_1_55_0/}{Numdam}.
\end{thebibliography}
\end{document}